\documentclass[12pt,titlepage]{amsart}
\pdfoutput=1
\expandafter\let\csname ver@amsthm.sty\endcsname\relax
\let\theoremstyle\relax

\let\qedhere\relax

\usepackage[ascii]{inputenc}
\usepackage{hyperref}
\usepackage{bbm,bm,braket,microtype,aliascnt,cite,mathrsfs,amsmath,amssymb,color,amsthm,fullpage,graphicx,chngcntr,enumitem,longtable,booktabs}
\usepackage[capitalise]{cleveref}
\usepackage{amsaddr}
\usepackage[T1]{fontenc}
\usepackage[USenglish]{babel}
\swapnumbers
\makeatletter
\def\swappedhead#1#2#3{%
  \thmnumber{\@upn{\the\thm@headfont\normalfont(#2)\@ifnotempty{#1}{}}}%
  \thmname{\@ifnotempty{#2}{~}#1}%
  \thmnote{ {\the\thm@notefont(#3)}}}
\makeatother
\newtheorem{theorem}[equation]{Theorem}\crefname{theorem}{Theorem}{Theorems}
\newtheorem{lemma}[equation]{Lemma}\crefname{lemma}{Lemma}{Lemmas}
\newtheorem{proposition}[equation]{Proposition}\crefname{proposition}{Proposition}{Propositions}
\newtheorem{corollary}[equation]{Corollary}\crefname{corollary}{Corollary}{Corollaries}

\theoremstyle{definition}
\newtheorem{definition}[equation]{Definition}\crefname{definition}{Definition}{Definitions}
\newtheorem{example}[equation]{Example}\crefname{example}{Example}{Examples}
\newtheorem{remark}[equation]{Remark}\crefname{remark}{Remark}{Remarks}
\newtheorem*{remark*}{Remark}
\newtheorem*{example*}{Example}
\allowdisplaybreaks[4]
\setenumerate{label=(\roman*)}

\DeclareMathOperator{\tdim}{tdim}
\DeclareMathOperator{\edim}{edim}
\DeclareMathOperator{\kdim}{kdim}
\DeclareMathOperator{\flag}{Flag}
\DeclareMathOperator{\GL}{GL}
\DeclareMathOperator{\SL}{SL}

\DeclareMathOperator{\grass}{Gr}
\DeclareMathOperator{\im}{im}
\DeclareMathOperator{\Pos}{Pos}
\DeclareMathOperator{\kPos}{kPos}

\DeclareMathOperator{\tr}{tr}
\DeclareMathOperator{\linspan}{span}

\DeclareMathOperator{\Hom}{Hom}
\DeclareMathOperator{\Good}{Good}

\DeclareMathOperator{\Kirwan}{Kirwan}
\DeclareMathOperator{\Subsets}{Subsets}
\DeclareMathOperator{\Intersecting}{Intersecting}
\DeclareMathOperator{\Horn}{Horn}
\DeclareMathOperator{\id}{id}
\newcommand{\ot}{\otimes}
\newcommand{\op}{\oplus}
\newcommand{\gl}{\mathfrak{gl}}
\newcommand{\p}{\mathfrak p}
\renewcommand{\t}{\mathfrak t}

\newcommand{\h}{\mathfrak h}
\newcommand{\C}{\mathbb C}
\newcommand{\R}{\mathbb R}
\newcommand{\Z}{\mathbb Z}

\renewcommand{\P}{\mathbb P}

\newcommand{\CE}{\mathcal E}
\newcommand{\CF}{\mathcal F}
\newcommand{\CG}{\mathcal G}

\newcommand{\CI}{\mathcal I}
\newcommand{\CJ}{\mathcal J}
\newcommand{\CK}{\mathcal K}
\newcommand{\CL}{\mathcal L}
\newcommand{\CM}{\mathcal M}
\newcommand{\CO}{\mathcal O}

\newcommand{\CX}{\mathcal X}
\newcommand{\CY}{\mathcal Y}
\newcommand{\CZ}{\mathcal Z}
\newcommand{\Pa}{\mathrm P}

\newcommand{\Pt}{\mathrm P_{\mathrm t}}
\newcommand{\Bt}{\mathrm B_{\mathrm t}}
\newcommand{\Pk}{\mathrm P_{\mathrm k}}
\newcommand{\Mk}{\mathrm M_{\mathrm k}}
\newcommand{\Pkt}{\mathrm P_{\mathrm{kt}}}
\newcommand{\Bkt}{\mathrm B_{\mathrm{kt}}}
\newcommand{\Mkp}{\mathrm M_{\mathrm{kp}}}
\newcommand{\Pkp}{\mathrm P_{\mathrm{kp}}}
\newcommand{\Pkpt}{\mathrm P_{\mathrm{kpt}}}
\newcommand{\Bkpt}{\mathrm B_{\mathrm{kpt}}}
\newcommand{\PF}{\mathrm P_\CF}
\newcommand{\PFt}{\mathrm P_{\CF,\mathrm t}}
\newcommand{\BFt}{\mathrm B_{\CF,\mathrm t}}
\newcommand{\PFkpt}{\mathrm P_{\CF,\mathrm{kpt}}}
\newcommand{\BFkpt}{\mathrm B_{\CF,\mathrm{kpt}}}
\newcommand{\PFkp}{\mathrm P_{\CF,\mathrm{kp}}}

\newcommand{\PFkpC}{\mathrm P_{\CF,\mathrm{kp},C}}
\newcommand{\PFkpCF}{\mathrm P_{\CF,\mathrm{kp},C_\CF}}
\newcommand{\cint}{c_{\mathrm{int}}}
\newcommand{\one}{\mathbbm 1}

\begin{document}

\title{The Horn inequalities from a geometric point of view}
\author{Nicole Berline}
\address{Centre de Math\'ematiques Laurent Schwartz, Ecole Polytechnique}
\email{nicole.berline@math.cnrs.fr}
\author{Mich\`ele Vergne}
\address{Institut de Math\'ematiques de Jussieu, Paris Rive Gauche}
\email{michele.vergne@imj-prg.fr}
\author{Michael Walter}
\address{Korteweg-de Vries Institute for Mathematics, University of Amsterdam \& QuSoft \\
Stanford Institute for Theoretical Physics, Stanford University}
\email{m.walter@uva.nl}

\begin{abstract}
  We give an exposition of the Horn inequalities and their triple role characterizing tensor product invariants, eigenvalues of sums of Hermitian matrices, and intersections of Schubert varieties.
  We follow Belkale's geometric method, but assume only basic representation theory and algebraic geometry, aiming for self-contained, concrete proofs.
  In particular, we do not assume the Littlewood-Richardson rule nor an a priori relation between intersections of Schubert cells and tensor product invariants.
  Our motivation is largely pedagogical, but the desire for concrete approaches is also motivated by current research in computational complexity theory and effective algorithms.
\end{abstract}
\maketitle
\tableofcontents

\counterwithin{equation}{section}
\section{Introduction}\label{sec:intro}

The possible eigenvalues of Hermitian matrices $X_1,\dots,X_s$ such that $X_1+\dots+X_s = 0$ form a convex polytope.
They can thus be characterized by a finite set of linear inequalities, most famously so by the inductive system of linear inequalities conjectured by Horn~\cite{horn1962eigenvalues}.
The very same inequalities give necessary and sufficient conditions on highest weights $\lambda_1,\dots,\lambda_s$ such that the tensor product of the corresponding irreducible $\GL(r)$-representations $L(\lambda_1),\dots,L(\lambda_s)$ contains a nonzero invariant vector, i.e.,
$c(\vec\lambda) := \dim (L(\lambda_1) \ot \cdots \ot L(\lambda_s))^{\GL(r)} > 0.$
For $s=3$, the multiplicities $c(\vec\lambda)$ can be identified with the \emph{Littlewood-Richardson coefficients}.
Since the Horn inequalities are linear, $c(\vec\lambda) > 0$ if and only if $c(N \vec\lambda) > 0$ for any integer $N > 0$.
This is the celebrated \emph{saturation property} of $\GL(r)$, first established combinatorially by Knutson and Tao~\cite{knutson1999honeycomb} building on work by Klyachko~\cite{klyachko1998stable}.
Some years after, Belkale has given an alternative proof of the Horn inequalities and the saturation property~\cite{belkale2006geometric}.
His main insight is to `geometrize' the classical relationship between the invariant theory of $\GL(r)$ and the intersection theory of Schubert varieties of the Grassmannian.
In particular, by a careful study of the tangent space of intersections, he shows how to obtain a geometric basis of invariants.

The aim of this text is to give a self-contained exposition of the Horn inequalities, assuming only linear algebra and some basic representation theory and algebraic geometry, similar in spirit to the approach taken in~\cite{vergne2014inequalities}.
We also discuss a proof of Fulton's conjecture which asserts that $c(\vec\lambda)=1$ if and only if $c(N\vec\lambda)=1$ for any integer $N\geq1$.
We follow Belkale's geometric method~\cite{belkale2004invariant,belkale2006geometric,belkale2007geometric}, as recently refined by Sherman~\cite{sherman2015geometric}, and do not claim any originality.
Instead, we hope that our text might be useful by providing a more accessible introduction to these topics, since we tried to give simple and concrete proofs of all results.
In particular, we do not use the Littlewood-Richardson rule for determining $c(\vec\lambda)$, and we do not discuss the relation of a basis of invariants to the integral points of the hive polytope~\cite{knutson1999honeycomb}.
Instead, we describe a basis of invariants that can be identified with the Howe-Tan-Willenbring basis, which is constructed using determinants associated to Littlewood-Richardson tableaux, as we explained in~\cite{vergne2014inequalities}.
We will come back to this subject in the future.
We note that Derksen and Weyman's work~\cite{derksen2000semi} can be understood as a variant of the geometric approach in the context of quivers.
For alternative accounts we refer to the work by Knutson and Tao~\cite{knutson1999honeycomb} and Woodward~\cite{knutson2004honeycomb}, Ressayre~\cite{ressayre2010geometric,ressayre2009short} and to the expositions by Fulton and Knutson~\cite{fulton2000eigenvalues,knutson1999symplectic}.

The desire for concrete approaches to questions of representation theory and algebraic geometry is also motivated by recent research in computational complexity and the interest in efficient algorithms.
Indeed, the saturation property implies that deciding the nonvanishing of a Littlewood-Richardson coefficient can be decided in polynomial time~\cite{mulmuley2012geometric}.
In contrast, the analogous problem for the Kronecker coefficients, which are not saturated, is NP-hard, but believed to simplify in the asymptotic limit~\cite{ikenmeyer2015vanishing,burgisser2015membership}.
We refer to~\cite{mulmuley2011p,burgisser2011overview} for further detail.

These notes are organized as follows:
In \cref{sec:easy}, we start by motivating the triple role of the Horn inequalities characterizing invariants, eigenvalues, and intersections.
Then, in \cref{sec:linalg}, we collect some useful facts about positions and flags. 
This is used in \cref{sec:horn necessary,sec:horn sufficient} to establish Belkale's theorem characterizing intersecting Schubert varieties in terms of Horn's inequalities.
In \cref{sec:invariants}, we explain how to construct a geometric basis of invariants from intersecting Schubert varieties.
This establishes the Horn inequalities for the Littlewood-Richardson coefficients, and thereby the saturation property, as well as for the eigenvalues of Hermitian matrices that sum to zero.
In \cref{sec:fulton}, we sketch how Fulton's conjecture can be proved geometrically by similar techniques.
Lastly, in the \hyperref[app:examples horn]{appendix}, we have collected the Horn inequalities for three tensor factors and low dimensions.

\subsubsection*{Notation}

We write $[n] := \{1,\dots,n\}$ for any positive integer $n$.
For any group $G$ and representation $M$, we write $M^G$ for the linear subspace of $G$-invariant vectors.
For any subgroup $H\subseteq G$, we denote by $G/H = \{ gH \}$ the right coset space.
If $F$ is an $H$-space, we denote by $G \times_H F$ the quotient of $G\times F$ by the equivalence relation $(g,f) \sim (gh^{-1}, hf)$ for $g\in G$, $f\in F$, $h\in H$.
Note that $G \times_H F$ is a $G$-space fibered over $G/H$, with fiber~$F$.
If~$F$ if a subspace of a $G$-space $X$, then $G\times_H F$ is identified by the $G$-equivariant map $[g,f] \mapsto (gH,gf)$ with the subspace of $G/H \times X$ (equipped with the diagonal $G$-action) consisting of the~$(gH,x)$ such that $g^{-1}x\in F$.

\counterwithin{equation}{section}
\section{A panorama of invariants, eigenvalues, and intersections}\label{sec:easy}

In this section we give a panoramic overview of the relationship between invariants, eigenvalues, and intersections.
Our focus is on explaining the intuition, connections, and main results.
To keep the discussion streamlined, more difficult proofs are postponed to later sections (in which case we use the numbering of the later section, so that the proofs can easily be found).
The rest of this article, from \cref{sec:linalg} onwards, is concerned with developing the necessary mathematical theory and giving these proofs.

We start by recalling the basic representation theory of the general linear group $\GL(r) := \GL(r,\C)$. 
Consider $\C^r$ with the ordered standard basis $e(1),\dots,e(r)$ and standard Hermitian inner product.
Let $H(r)$ denote the subgroup of invertible matrices $t \in \GL(r)$ that are diagonal in the standard basis, i.e., $t\,e(i) = t(i)\cdot e(i)$ with all $t(i)\neq0$.
We write $t = (t(1),\dots,t(r))$ and thereby identify $H(r) \cong (\C^*)^r$.
To any sequence of integers $\mu = (\mu(1),\dots,\mu(r))$, we can associate a character of $H(r)$ by $t \mapsto t^\mu := t(1)^{\mu(1)}\cdots t(r)^{\mu(r)}$.
We say that $\mu$ is a weight and call $\Lambda(r) = \Z^r$ the weight lattice.
A weight is dominant if $\mu(1)\geq\dots\geq\mu(r)$, and the set of all dominant weights form a semigroup, denoted by $\Lambda_+(r)$.
We later also consider antidominant weights $\omega$, which satisfy $\omega(1)\leq\dots\leq\omega(r)$.

For any dominant weight $\lambda \in \Lambda_+(r)$, there is an unique irreducible representation $L(\lambda)$ of $\GL(r)$ with highest weight $\lambda$.
That is, if $B(r)$ denotes the group of upper-triangular invertible matrices (the standard Borel subgroup of $\GL(r)$) and $N(r) \subseteq B(r)$ the subgroup of upper-triangular matrices with all ones on the diagonal (i.e., the corresponding unipotent), then $L(\lambda)^{N(r)} = \C v_\lambda$ is a one-dimensional eigenspace of $B(r)$ of $H(r)$-weight $\lambda$.
We say that $v_\lambda$ is a highest weight vector of $L(\lambda)$.
In \cref{subsec:borel-weil} we describe a concrete construction of $L(\lambda)$ due to Borel and Weil. 
Now let $U(r)$ denote the group of unitary matrices, which is a maximally compact subgroup of $\GL(r)$.
We can choose an $U(r)$-invariant Hermitian inner product $\braket{\cdot,\cdot}$ (by convention complex linear in the second argument) on each $L(\lambda)$ so that the representation $L(\lambda)$ restricts to an irreducible unitary representation of $U(r)$. Any two such representations of $U(r)$ are pairwise inequivalent, and, by Weyl's trick, any irreducible unitary representation can be obtained in this way.
Let us now decompose their Lie algebras as $\gl(r) = \mathfrak u(r) \op i \mathfrak u(r)$, where $i = \sqrt{-1}$, and likewise $\h(r) = \t(r) \op i \t(r)$, where we write $\t(r)$ for the Lie algebra of $T(r)$, the group of diagonal unitary matrices, and similarly for the other Lie groups.
Here, $i \mathfrak u(r)$ denotes the space of Hermitian matrices and $i \t(r)$ the subspace of diagonal matrices with real entries.
We freely identify vectors in $\R^r$ with the corresponding diagonal matrices in $i\t(r)$ and denote by $(\cdot,\cdot)$ the usual inner product of $i\t(r)\cong\R^r$.
For a subset $J\subseteq[r]$, we write $T_J$ for the vector (diagonal matrix) in $i\t(r)$ that has ones in position $J$, and otherwise zero.

Now let $\CO_\lambda$ denote the set of Hermitian matrices with eigenvalues $\lambda(1)\geq\dots\geq\lambda(r)$.
By the spectral theorem, $\CO_\lambda$ is a $U(r)$-orbit with respect to the adjoint action, $u \cdot X := u X u^*$, and so $\CO_\lambda = U(r) \cdot \lambda$, where we identify $\lambda$ with the diagonal matrix with entries $\lambda(1)\geq\dots\geq\lambda(r)$.
On the other hand, recall that any invertible matrix $g \in \GL(r)$ can be written as a product $g = u b$, where $u \in U(r)$ is unitary and $b \in B(r)$ upper-triangular.
Since $v_\lambda$ is an eigenvector of $B(r)$, it follows that, in projective space $\P(L(\lambda))$, the orbits of $[v_\lambda]$ for $\GL(r)$ and $U(r)$ are the same!
Moreover, it is not hard to see that the $U(r)$-stabilizers of $\lambda$ and of $[v_\lambda]$ agree, so we obtain a $U(r)$-equivariant diffeomorphism
\begin{equation}
\label{eq:coadjoint orbit}
  \CO_\lambda \to U(r) \cdot [v_\lambda] = \GL(r) \cdot [v_\lambda] \subseteq \P(L(\lambda)), \quad u \cdot \lambda \mapsto u \cdot [v_\lambda] = [u \cdot v_\lambda]
\end{equation}
which also allows us to think of the adjoint orbit $\CO_\lambda$ as a complex projective $\GL(r)$-variety. An important observation is that
\begin{equation}
\label{eq:moment map equivariance}
  \tr\bigl((u \cdot \lambda) A\bigr) = \frac {\braket{u \cdot v_\lambda, \rho_\lambda(A) (u \cdot v_\lambda)}} {\lVert v_\lambda \rVert^2}
\end{equation}
for all complex $r \times r$-matrices $A$, i.e., elements of the Lie algebra $\mathfrak{gl}(r)$ of $\GL(r)$; $\rho_\lambda$ denotes the Lie algebra representation on $L(\lambda)$.
To see that~\eqref{eq:moment map equivariance} holds true, we may assume that $\lVert v_\lambda \rVert = 1$ as well as that $u = 1$, the latter by $U(r)$-equivariance.
Now $\tr(A \lambda) = \braket{v_\lambda, \rho_\lambda(A) v_\lambda}$ is easily be verified by decomposing $A = L + H + R$ with $L$ strictly lower triangular, $R$ strictly upper triangular, and $H \in \h(r)$ diagonal and comparing term by term.
These observations lead to the following fundamental connection between the eigenvalues of Hermitian matrices and the invariant theory of the general linear group:

\begin{proposition}[Kempf-Ness,~\cite{kempf1979length}]
\label{prp:kempf-ness}
  Let $\lambda_1,\dots,\lambda_s$ be dominant weights for $\GL(r)$ such that $(L(\lambda_1) \ot \cdots \ot L(\lambda_s))^{\GL(r)} \neq \{0\}$.
  Then there exist Hermitian matrices $X_k \in \mathcal O_{\lambda_k}$ such that $\sum_{k=1}^s X_k = 0$.
\end{proposition}
\begin{proof}
  Let $0 \neq w \in (L(\lambda_1)\ot\cdots\ot L(\lambda_s))^{\GL(r)}$ be a nonzero invariant vector.
  Then, $P(v) := \braket{w, v}$ is a nonzero linear function on $L(\lambda_1)\ot\cdots\ot L(\lambda_s)$ that is invariant under the diagonal action of $\GL(r)$; indeed, $\braket{w, g \cdot v} = \braket{g^* \cdot w, v} = \braket{w, v}$.
  Since the $L(\lambda_k)$ are irreducible, they are spanned by the orbits $U(r) v_{\lambda_k}$.
  Thus we can find $u_1, \dots, u_s \in U(r)$ such that $P(v) \neq 0$ for $v = (u_1 \cdot v_{\lambda_1}) \ot \cdots \ot (u_s \cdot v_{\lambda_s})$.

  Consider the class $[v]$ of $v$ in the corresponding projective space $\P(L(\lambda_1)\ot\cdots\ot L(\lambda_s))$.
  The orbit of $[v]$ under the diagonal $\GL(r)$-action is contained in the $\GL(r)^s$-orbit, which is the closed set $[U(r) \cdot v_{\lambda_1}\ot\cdots\ot U(r) \cdot v_{\lambda_s}]$ according to the discussion preceding~\eqref{eq:coadjoint orbit}.
  It follows that $\GL(r) \cdot v$ and its closure, $\overline{\GL(r) \cdot v}$ (say, in the Euclidean topology), are contained in the closed set $\{ \kappa (u'_1 \cdot v_{\lambda_1}) \ot \cdots \ot (u'_s \cdot v_{\lambda_s}) \}$ for $\kappa\in\C$ and $u'_1,\dots,u'_s\in U(r)$.

  Since $P$ is $\GL(r)$-invariant, $P(v') = P(v) \neq 0$ for any vector $v'$ in the diagonal $\GL(r)$-orbit of $v$.
  By continuity, this is also true in the orbits' closure, $\overline{\GL(r) \cdot v}$.
  On the other hand, $P(0) = 0$.
  It follows that $0\not\in\overline{\GL(r) \cdot v}$, i.e., the origin does not belong to the orbit closure.
  Consider then a nonzero vector $v'$ of minimal norm in $\overline{\GL(r) \cdot v}$.
  By the discussion in the preceding paragraph, this vector is of the form $v' = \kappa (u'_1 \cdot v_{\lambda_1}) \ot \cdots \ot (u'_s \cdot v_{\lambda_s})$ for some $0 \neq \kappa \in \C$ and $u'_1,\dots,u'_s \in U(r)$. 
  By rescaling $v$ we may moreover assume that $\kappa = 1$, so that $v'$ is a unit vector.

  The vector $v'$ is by construction a vector of minimal norm in its own $\GL(r)$-orbit.
  It follows that, for any Hermitian matrix $A$,
  \begin{align*}
  0 &= \frac12 \partial_{t=0} \lVert (e^{At}\ot\cdots\ot e^{At}) \cdot v' \rVert^2 \\
  &= \braket{v', \bigl(\rho_{\lambda_1}(A) \ot I\ot\cdots\ot I + \dots + I\ot\cdots\ot I\ot\rho_{\lambda_s}(A)\bigr) v'} \\
  &= \sum_{k=1}^s \braket{u'_k \cdot v_{\lambda_k}, \rho_{\lambda_k}(A) (u'_k \cdot v_{\lambda_k})}
  = \sum_{k=1}^s \tr\bigl(A (u'_k \cdot \lambda_k)\bigr) = \sum_{k=1}^s \tr(A X_k),
  \end{align*}
  where we have used \cref{eq:moment map equivariance} and set $X_k := u'_k \cdot \lambda_k$ for $k\in[s]$.
  This implies at once that $\sum_{k=1}^s X_k = 0$.
\end{proof}

The adjoint orbits $\CO_\lambda = U(r) \cdot \lambda$ (but not the map~\eqref{eq:coadjoint orbit}) can be defined not only for dominant weights $\lambda$ but in fact for arbitrary Hermitian matrices. 
Conversely, any Hermitian matrix is conjugate to a unique element $\xi \in i\t(r)$ such that $\xi(1)\geq\dots\geq\xi(r)$.
The set of all such $\xi$ is a convex cone, known as the positive Weyl chamber $C_+(r)$, and it contains the semigroup of dominant weights.
Throughout this text, we only ever write $\CO_\xi = U(r)\cdot\xi$ for $\xi$ that are in the positive Weyl chamber.
For example, if $\xi \in C_+(r)$ then $-\xi \in \CO_{\xi^*}$, where $\xi^* = (-\xi(r),\dots,-\xi(1)) \in C_+(r)$.
If $\lambda$ is a dominant weight then $\lambda^*=(-\lambda(d),\dots,-\lambda(1))$ is the highest weight of the dual representation of $L(\lambda)$, i.e., $L(\lambda^*) \cong L(\lambda)^*$.

\begin{remark*}
  Using the inner product $(A,B) := \tr(AB)$ on Hermitian matrices we may also think of $\lambda$ as an element in $i\t(r)^*$ and of $\CO_\lambda$ as a \emph{co}adjoint orbit in $i\mathfrak u(r)^*$.
  From the latter point of view, the map $(X_1,\dots,X_s)\mapsto\sum_{k=1}^s X_k$ is the moment map for the diagonal $U(r)$-action on the product of Hamiltonian manifolds $\CO_{\lambda_k}$, $k\in[s]$.
  \Cref{prp:kempf-ness} thus relates the existence of nonzero invariants to the statement that the zero set of the corresponding moment map is nonempty.
  This is a general fact of Mumford's geometric invariant theory.
\end{remark*}

\begin{definition}
\label{def:kirwan cone}
  The \emph{Kirwan cone} $\Kirwan(r,s)$ is defined as the set of $\vec\xi = (\xi_1,\dots,\xi_s) \in C_+(r)^s$ such that there exist $X_k\in\CO_{\xi_k}$ with $\sum_{k=1}^s X_k=0$.
\end{definition}

Using this language, \cref{prp:kempf-ness} asserts that if the generalized Littlewood-Richardson coefficient $c(\vec\lambda) := \dim (L(\lambda_1) \ot \cdots \ot L(\lambda_s))^{\GL(r)} > 0$ is nonzero then $\vec\lambda$ is a point in the Kirwan cone $\Kirwan(r,s)$.

\begin{remark*}
We will see in \cref{sec:invariants} that, conversely, if $\vec\lambda\in\Kirwan(r,s)$, then $c(\vec\lambda)>0$ (by constructing an explicit nonzero invariant).
As a consequence, it will follow that $c(\vec\lambda) > 0$ if and only if $c(N\vec\lambda) > 0$ for some integer $N > 0$.
This is the remarkable saturation property of the Littlewood-Richardson coefficients.
In fact, we will show that the Horn inequalities give a complete set of conditions for nonvanishing $c(\vec\lambda)$ as well as for $\vec\xi\in\Kirwan(r,s)$, which in particular establishes that $\Kirwan(r,s)$ is indeed a convex polyhedral cone.
We will come back to these points at the end of this section.
\end{remark*}

If there exist permutations $w_k$ such that $\sum_{k=1}^s w_k \cdot \xi_k = 0$ then $\vec\xi \in \Kirwan(r,s)$
(choose each $X_k$ as the diagonal matrix $w_k \cdot \xi_k$).
This suffices to characterize the Kirwan cone for $s\leq2$:

\begin{example*}
  For $s=1$, it is clear that $\Kirwan(r,1) = \{0\}$.
  When $s=2$, then $\Kirwan(r,2) = \{(\xi,\xi^*)\}$.
  Indeed, if $X_1 \in \CO_{\xi_1}$ and $X_2 \in \CO_{\xi_2}$ with $X_1 + X_2 = 0$, then $X_2 = -X_1 \in \CO_{\xi_1^*}$.
\end{example*}

In general, however, it is quite delicate to determine if a given $\vec\xi \in C_+(r)^s$ is in $\Kirwan(r,s)$ or not.
Clearly, one necessary condition is that $\sum_{k=1}^s \lvert \xi_k \rvert = 0$, where we have defined $\lvert \mu \rvert := \sum_{j=1}^r \mu(j)$ for an arbitrary $\mu\in\h(r)$.
This follows by taking the trace of the equation $\sum_{k=1}^s X_k = 0$.
In fact, it is clear that by adding or subtracting appropriate multiples of the identity matrix we can always reduce to the case where each $\lvert \xi_k \rvert = 0$.

\begin{example*}
  Let $X_k \in \CO_{\xi_k}$ such that $\sum_{k=1}^s X_k = 0$.
  For each $k$, let $v_k$ denote a unit eigenvector of $X_k$ with eigenvalue $\xi_k(1)$.
  Then we have
  \begin{equation}
  \label{eq:triangle ieqs}
    0 = \braket{v_k, (\sum_{l=1}^s X_l) v_k} = \xi_k(1) + \sum_{l\neq k} \braket{v_k, X_l v_k} \geq \xi_k(1) + \sum_{l\neq k} \xi_l(r)
  \end{equation}
  since $\xi_l(r) = \min_{\lVert v \rVert = 1} \braket{v, X_l v}$ by the variational principle for the minimal eigenvalue of a Hermitian matrix $X_l$.
  These inequalities, together with $\sum_{k=1}^s \lvert\xi_k\rvert = 0$, characterize the Kirwan cone for $r=2$, as can be verified by brute force.

  There is also a pleasant geometric way of understanding these inequalities in the case $r=2$.
  As discussed above, we may assume that the $X_k$ are traceless, i.e., that $\xi_k = (j_k,-j_k)$ for some $j_k\geq0$.
  Recall that the traceless Hermitian matrices form a three-dimensional real vector space, spanned by the Pauli matrices.
  Thus each $X_k$ identifies with a vector $x_k \in \R^3$, and the condition that $X_k \in \CO_{\xi_k}$ translates into $\lVert x_k \rVert = j_k$.
  Thus we seek to characterize necessary and sufficient conditions on the lengths $j_k$ of vectors $x_k$ that sum to zero, $\sum_{k=1}^s x_k = 0$.
  By the triangle inequality, $j_k = \lVert x_k \rVert \leq \sum_{l\neq k} \lVert x_l \rVert = \sum_{l\neq k} j_l$, which is equivalent to the above.
  It is instructive to observe that $j_k \leq \sum_{l\neq k} j_l$ is precisely the Clebsch-Gordan rule for $\SL(2)$ when the $j_k$ are half-integers.
\end{example*}

The proof of \cref{eq:triangle ieqs}, which was valid for any $s$ and $r$, suggests that a more general variational principle for eigenvalues might be useful to produce linear inequalities for the Kirwan cones.

\begin{definition}
\label{def:flag and adapted basis}
  A \emph{(complete) flag} $F$ on a vector space $V$, $\dim V = r$, is a chain of subspaces
  \[ \{0\} = F(0) \subset F(1) \subset \cdots \subset F(j) \subset F({j+1}) \subset \dots \subset F(r) = V, \]
  such that $\dim F(j) = j$ for all $j=0,\dots,r$.
  Any ordered basis $f = (f(1),\dots,f(r))$ of $V$ determines a flag by $F(j) = \linspan \{ f(1), \dots, f(j) \}$. We say that $f$ is \emph{adapted} to $F$.
\end{definition}

Now let $X \in \CO_\xi$ be a Hermitian matrix with eigenvalues $\xi(1)\geq\dots\geq\xi(r)$.
Let $(f_X(1),\dots,f_X(r))$ denote an orthonormal eigenbasis, ordered correspondingly, and denote by $F_X$ the corresponding \emph{eigenflag} of $X$, defined as above.
Note that $F_X$ is uniquely defined if the eigenvalues $\xi(j)$ are all distinct.
We can quantify the position of a subspace with respect to a flag in the following way:

\begin{definition}
\label{def:schubert pos}
  The \emph{Schubert position} of an $d$-dimensional subspace $S \subseteq V$ with respect to a flag $F$ on $V$ is the strictly increasing sequence $J$ of integers defined by
  \[ J(b) := \min \{ j \in [r], \,\dim F(j) \cap S = b \} \]
  for $b \in [d]$.
  We write $\Pos(S,F) = J$ and freely identify $J$ with the subset $\{ J(1) < \dots < J(d) \}$ of $[r]$.
  In particular, $\Pos(S,F)=\emptyset$ for $S=\{0\}$ the zero-dimensional subspace.
\end{definition}

The upshot of these definitions is the following variational principle:

\begin{lemma}
\label{lem:variational}
  Let $\xi\in C_+(r)$, $X \in \CO_\xi$ with eigenflag $F_X$, and $J \subseteq [r]$ a subset of cardinality $d$.
  Then,
  \[ \min_{S : \Pos(S, F_X) = J} \tr(P_S X) = \sum_{j \in J} \xi(j) = (T_J, \xi), \]
  where $P_S$ denotes the orthogonal projector onto an $d$-dimensional subspace $S \subseteq \C^r$.
\end{lemma}
\begin{proof}
  Recall that $F_X(j) = \linspan \{ f_X(1), \dots, f_X(j) \}$, where $(f_X(1),\dots,f_X(r))$ is an orthonormal eigenbasis of $X$, ordered according to $\xi(1)\geq\dots\geq\xi(r)$.
  Given a subspace $S$ with $\Pos(S, F_X) = J$, we can find an ordered orthonormal basis $(s(1),\dots,s(d))$ of $S$ where each $s(a) \in F_X(J(a))$.
  Therefore,
  \[ \tr(P_S X) = \sum_{a=1}^d \braket{s(a), X s(a)} \geq \sum_{a=1}^d \xi(J(a)) = \sum_{j \in J} \xi(j). \]
  The inequality holds term by term, as the Hermitian matrix obtained by restricting $X$ to the subspace $F_X(J(a))$ has smallest eigenvalue $\xi(J(a))$. 
  Since $\tr(P_S X) = \sum_{j \in J} \xi(j)$ for $S = \linspan \{ f_X(j) : j \in J \}$, this establishes the lemma.
\end{proof}

Recall that the Grassmannian $\grass(d,V)$ is the space of $d$-dimensional subspaces of $V$.
We may partition $\grass(d,V)$ according to the Schubert position with respect to a fixed flag:

\begin{definition}
\label{def:schubert cell and variety}
  Let $F$ be a flag on $V$, $\dim V = r$, and $J \subseteq [r]$ a subset of cardinality $d$.
  The \emph{Schubert cell} is
  \[ \Omega^0_J(F) = \{ S \subseteq V : \dim S = d, \, \Pos(S, F) = J \}. \]
  The \emph{Schubert variety} $\Omega_J(F)$ is defined as the closure of $\Omega^0_J(F)$ in the Grassmannian $\grass(d,V)$. 
\end{definition}

The closures in the Euclidean and Zariski topology coincide; the $\Omega_J(F)$ are indeed algebraic varieties.
Using these definitions, \cref{lem:variational} asserts that
$\min_{S \in \Omega^0_J(F_X)} \tr(P_S X) = \sum_{j \in J} \xi(j)$
for any $X \in \CO_\xi$.
Since the orthogonal projector $P_S$ is a continuous function of $S \in \grass(d,V)$ (in fact, the Grassmannian is homeomorphic to the space of orthogonal projectors of rank $d$), it follows at once that
\begin{equation}
\label{eq:variational closure}
  \min_{S \in \Omega_J(F_X)} \tr(P_S X) = \sum_{j \in J} \xi(j) = (T_J,\xi).
\end{equation}
As a consequence, intersections of Schubert varieties imply linear inequalities of eigenvalues of matrices summing to zero:

\begin{lemma}
\label{lem:eigenflag ieqs}
  Let $X_k \in \CO_{\xi_k}$ be Hermitian matrices with $\sum_{k=1}^s X_k = 0$.
  If $J_1,\dots,J_s \subseteq [r]$ are subsets of cardinality $d$ such that $\bigcap_{k=1}^s \Omega_{J_k}(F_{X_k}) \neq \emptyset$, then $\sum_{k=1}^s (T_{J_k},\xi_k) 
  \leq 0$.
\end{lemma}
\begin{proof}
  Let $S \in \bigcap_{k=1}^s \Omega_{J_k}(F_{X_k})$. Then, $0 = \sum_{k=1}^s \tr(P_S X_k) \geq 
  \sum_{k=1}^s (T_{J_k}, \xi_k)$ by~\eqref{eq:variational closure}.
\end{proof}

Remarkably, we will find that it suffices to consider only those $J_1,\dots,J_s$ such that $\bigcap_{k=1}^s \Omega_{J_k}(F_k)\neq\emptyset$ for all flags $F_1,\dots,F_s$.
We record the corresponding eigenvalue inequalities, together with the trace condition, in \cref{cor:klyachko kirwan} below.
Following~\cite{belkale2006geometric}, we denote $s$-tuples by calligraphic letters, e.g., $\CJ = (J_1,\dots,J_s)$, $\CF=(F_1,\dots,F_s)$, etc.
In the case of Greek letters we continue to write $\vec\lambda=(\lambda_1,\dots,\lambda_s)$, etc., as above.

\begin{definition}
\label{def:intersecting}
  We denote by $\Subsets(d,r,s)$ the set of $s$-tuples $\CJ$, where each $J_k$ is a subset of $[r]$ of cardinality~$d$.
  Given such a $\CJ$, let $\CF$ be an $s$-tuple of flags on $V$, with $\dim V = r$. 
  Then we define
  \[ \Omega^0_\CJ(\CF) := \bigcap_{k=1}^s \Omega^0_{J_k}(F_k), \qquad \Omega_\CJ(\CF) := \bigcap_{k=1}^s \Omega_{J_k}(F_k). \]
  We shall say that $\CJ$ is \emph{intersecting} if $\Omega_\CJ(\CF)\neq\emptyset$ for every $s$-tuple of flags $\CF$, and we denote
  denote the set of such~$\CJ$ by $\Intersecting(d,r,s) \subseteq \Subsets(d,r,s)$.
\end{definition}

\begin{corollary}[Klyachko,~\cite{klyachko1998stable}]
\label{cor:klyachko kirwan}
  If $\vec\xi \in \Kirwan(r,s)$ then $\sum_{k=1}^s \lvert \xi_k \rvert = 0$, and for any $0<d<r$ and any $s$-tuple $\CJ \in \Intersecting(d,r,s)$ we have that
  $\sum_{k=1}^s (T_{J_k}, \xi_k) \leq 0$.
\end{corollary}

\begin{example*}
  If $J = \{1,\dots,d\} \subseteq [r]$ then $\Omega^0_J(F) = \{ F(d) \}$ is a single point.
  On the other end, if $J = \{r-d+1,\dots,r\}$ then $\Omega^0_J(F)$ is dense in $\grass(r, V)$, so that $\Omega_J(F) = \grass(r, V)$.
  It follows that $\CJ = (J_1, \{r-d+1,\dots,r\},\dots,\{r-d+1,\dots,r\}) \in \Intersecting(d, r, s)$ is intersecting for any $J_1$ (and likewise for permutations of the $s$ factors).

  For $d=1$, this means that $\Omega_{\{r\}}(F) = \P(V)$, so that~\eqref{eq:variational closure} reduces to the variational principle for the minimal eigenvalue, $\xi(r) = \min_{\lVert v \rVert = 1} \braket{v, X v}$, which we used to derive~\eqref{eq:triangle ieqs} above.
  Indeed, since $(\{a\},\{r\},\dots,\{r\})$ is intersecting for any $a$, we find that~\eqref{eq:triangle ieqs} is but a special case of \cref{cor:klyachko kirwan}.
\end{example*}

In order to understand the linear inequalities in \cref{cor:klyachko kirwan}, we need to understand the sets of intersecting tuples.
In the remainder of this section we thus motivate Belkale's inductive system of conditions for an $s$-tuple to be intersecting.
For reasons that will become clear shortly, we slightly change notation:
$E$ will be a complete flag on some $n$-dimensional vector space $W$, $I$ will be a subset of $[n]$ of cardinality $r$, and hence $\Omega^0_I(E)$ will be a Schubert cell in the Grassmannian $\grass(r,W)$.
We will describe $\grass(r,W)$ and $\Omega^0_I(E)$ in detail in \cref{sec:linalg}.
For now, we note that the dimension of $\grass(r,W)$ is $r(n-r)$.
In fact, $\grass(r,W)$ is covered by affine charts isomorphic to $\C^{r(n-r)}$.
The dimension of a Schubert cell and the corresponding Schubert variety (its Zariski closure) is given by
\newcommand{\printdimschubertcell}[3]{
\begin{#1}
#2
  \dim\Omega^0_I(E) = \dim\Omega_I(E) = #3 \sum_{a=1}^r \bigl( I(a) - a \bigr) =: \dim I.
\end{#1}}
\printdimschubertcell{equation*}{\tag{\ref{eq:dim schubert cell}}}{}
Indeed, $\Omega^0_I(E)$ is contained in an affine chart $\C^{r(n-r)}$ and is isomorphic to a vector subspace of dimension $\dim I$.
So locally $\Omega^0_I(E)$ is defined by $r(n-r)-\dim I$ equations.
This is easy to see and we give a proof in \cref{sec:linalg}.

\begin{definition}
\label{def:edim}
  Let $\CI\in\Subsets(r,n,s)$.
  The \emph{expected dimension} associated with $\CI$ is 
  \[ \edim \CI := r(n-r) - \sum_{k=1}^s \bigl( r(n-r) - \dim I_k \bigr). \]
\end{definition}

This definition is natural in terms of intersections, as the following lemma shows:

\begin{lemma}
\label{lem:dim intersection}
  Let $\CE$ be an $s$-tuple of flags on $W$, $\dim W = n$, and $\CI\in\Subsets(r,n,s)$.
  If $\Omega^0_{\CI}(\CE) \neq \emptyset$ then its irreducible components (in the sense of algebraic geometry) are all of dimension at least $\edim \CI$.
\end{lemma}
\begin{proof}
  Each Schubert cell $\Omega^0_{I_k}(E_k)$ is locally defined by $r(n-r) - \dim I_k$ equations.
  It follows that any irreducible component $\CZ \subseteq \Omega^0_{\CI}(\CE) = \bigcap_{k=1}^s \Omega^0_{I_k}(E_k)$ is locally defined by $\sum_{k=1}^s (r(n-r) - \dim I_k)$ equations.
  These equations, however, are not necessarily independent. Thus the codimension of $\CZ$ is at most that number, and we conclude that $\dim \CZ \geq \edim \CI$.
\end{proof}

Belkale's first observation is that the expected dimension of an intersecting tuple $\CI \in \Intersecting(r, n, s)$ is necessarily nonnegative,
\newcommand{\printedimnonnegative}[2]{
\begin{#1}
#2
  \edim \CI = r(n-r) - \sum_{k=1}^s (r(n-r) - \dim I_k)\geq0.
\end{#1}}
\printedimnonnegative{equation*}{\tag{\ref{eq:edim nonnegative}}}
This inequality, as well as some others, will be proved in detail in \cref{sec:horn necessary}.
For now, we remark that the condition is rather natural from the perspective of Kleiman's moving lemma.
Given $\CI \in \Intersecting(r,n,s)$, it not only implies that the intersection of the Schubert \emph{cells}, $\Omega^0_{\CI}(\CE) = \bigcap_{k=1}^s \Omega^0_{I_k}(E_k)\neq \emptyset$, is nonempty for generic flags, but in fact transverse, so that the dimensions of its irreducible components are exactly equal to the expected dimension; hence, $\edim\CI\geq0$.

We now show that~\eqref{eq:edim nonnegative} gives rise to an inductive system of conditions.
Given a flag $E$ on $W$ and a subspace $V\subseteq W$, we denote by $E^V$ the flag obtained from the distinct subspaces in the sequence $E(i)\cap V$, $i=0,\dots,n$.
Given subsets $I \subseteq [n]$ of cardinality $r$ and $J \subseteq [r]$ of cardinality $d$, we also define their \emph{composition} $IJ$ as the subset $IJ = \{ I(J(1)) < \dots < I(J(d)) \} \subseteq [n]$.
(For $s$-tuples $\CI$ and $\CJ$ we define $\CI\CJ$ componentwise.)
Then we have the following `chain rule' for positions:
If $S \subseteq V \subseteq W$ are subspaces and $E$ is a flag on $W$ then
\newcommand{\printchainrulepos}[3]{
\begin{#1}
#2
  \Pos(S, E) = #3 \Pos(V, E) \Pos(S, E^V).
\end{#1}}
\printchainrulepos{equation*}{\tag{\ref{eq:chain rule pos}}}{}
We also have the following description of Schubert varieties in terms of Schubert cells:
\newcommand{\printschubertvarietycharacterization}[2]{
\begin{#1}
#2
  \Omega_I(E) = \bigcup_{I' \leq I} \Omega^0_{I'}(E),
\end{#1}}
\printschubertvarietycharacterization{equation*}{\tag{\ref{eq:schubert variety characterization}}}
where the union is over all subsets $I' \subseteq [n]$ of cardinality $r$ such that $I'(a) \leq I(a)$ for $a\in[r]$.
Both statements are not hard to see; we will give careful proofs in \cref{sec:linalg} below.
We thus obtain a corresponding chain rule for intersecting tuples:

\begin{lemma}
\label{lem:chain rule intersecting}
  If $\CI \in \Intersecting(r,n,s)$ and $\CJ \in \Intersecting(d,r,s)$, then we have $\CI\CJ \in \Intersecting(d,n,s)$.
\end{lemma}
\begin{proof}
  Let $\CE$ be an $s$-tuple of flags on $W=\C^n$.
  Since $\CI$ is intersecting, there exists $V \in \Omega_\CI(\CE)$.
  Let $\CE^V$ denote the $s$-tuple of induced flags on $V$.
  Likewise, since $\CJ$ is intersecting, we can find $S \in \Omega_\CJ(\CE^V)$.
  In particular, $\Pos(V, E_k)(a) \leq I_k(a)$ for $a\in[r]$ and $\Pos(S, E_k^V) \leq J_k(b)$ for $b\in[d]$ by~\eqref{eq:schubert variety characterization}.
  Thus~\eqref{eq:chain rule pos} shows that $\Pos(S, E_k)(b) = \Pos(V, E_k)\bigl(\Pos(S, E_k^V)(b)\bigr) \leq \Pos(V, E_k)(J_k(b)) \leq I_k(J_k(b))$.
  Using~\eqref{eq:schubert variety characterization} one last time, we conclude that $S \in \Omega_{\CI\CJ}(\CE)$.
\end{proof}

As an immediate consequence of Inequality~\eqref{eq:edim nonnegative} and \cref{lem:chain rule intersecting} we obtain the following set of necessary conditions for an $s$-tuple $\CI$ to be intersecting:

\begin{corollary}
\label{cor:belkale inductive}
  If $\CI \in \Intersecting(r,n,s)$ then for any $0<d<r$ and any $s$-tuple $\CJ \in \Intersecting(d,r,s)$ we have that $\edim \CI\CJ \geq 0$.
\end{corollary}

Belkale's theorem asserts that these conditions are also sufficient. In fact, it suffices to restrict to intersecting $\CJ$ with $\edim\CJ = 0$:

\begin{definition}
\label{def:horn}
  Let $\Horn(r,n,s)$ denote the set of $s$-tuples $\CI\in\Subsets(r,n,s)$ defined by the conditions that $\edim\CI\geq0$ and, if $r>1$, that
  \[ \edim\CI\CJ \geq 0 \]
  for all $\CJ \in \Horn(d, r, s)$ with $0<d<r$ and $\edim\CJ=0$.
\end{definition}

\newcommand{\printbelkaletheorem}[3]{
\begin{#1}[Belkale,~\cite{belkale2006geometric}#3]
#2
  For $r\in[n]$ and $s\geq2$,
  $\Intersecting(r, n, s) = \Horn(r, n, s)$.
\end{#1}}
\printbelkaletheorem{belkaletheorem}{}{}

We will prove \cref{thm:belkale} in \cref{sec:horn sufficient}.
The inequalities defining $\Horn(r,n,s)$ are in fact tightly related to those constraining the Kirwan cone $\Kirwan(r,s)$ and the existence of nonzero invariant vectors.
To any $s$-tuple of dominant weights $\vec\lambda$ for $\GL(r)$ such that $\sum_{k=1}^s \lvert\lambda_k\rvert=0$, we will associate an $s$-tuple $\CI\in\Subsets(r,n,s)$ for some $[n]$ such that $\edim\CI=0$.
Furthermore, if $\vec\lambda$ satisfies the inequalities in \cref{cor:klyachko kirwan} then $\CI\in\Horn(r,n,s)$.
In \cref{sec:invariants} we will explain this more carefully and show how Belkale's considerations allow us to construct a corresponding nonzero $\GL(r)$-invariant in $L(\lambda_1)\ot\cdots\ot L(\lambda_s)$.
By \cref{prp:kempf-ness}, 
we will thus obtain at once a characterization of the Kirwan cone as well as of the existence of nonzero invariants in terms of Horn's inequalities:

\newcommand{\printhorncorollary}[3]{
\begin{#1}[Knutson-Tao,~\cite{knutson1999honeycomb}#3]
#2
(a) \emph{Horn inequalities:} The Kirwan cone $\Kirwan(r,s)$ is the convex polyhedral cone of $\vec\xi\in C_+(r)^s$ such that $\sum_{k=1}^s \lvert \xi_k \rvert = 0$, and for any $0<d<r$ and any $s$-tuple $\CJ \in \Horn(d,r,s)$ with $\edim\CJ=0$ we have that
$\sum_{k=1}^s (T_{J_k},\xi_k) \leq 0$.

(b) \emph{Saturation property:} For a dominant weight $\vec\lambda\in \Lambda_+(r)^s$, the space of invariants $(L(\lambda_1) \ot \cdots \ot L(\lambda_s))^{\GL(r)}$ is nonzero if and only if $\vec\lambda \in \Kirwan(r,s)$.

In particular, $c(\vec\lambda) := \dim (L(\lambda_1) \ot \cdots \ot L(\lambda_s))^{\GL(r)} > 0$ if and only if $c(N \vec\lambda) > 0$ for some integer $N > 0$.
\end{#1}}
\printhorncorollary{horncorollary}{}{}

The proof of \cref{cor:horn and saturation} will be given in \cref{sec:invariants}.
In \cref{app:examples horn,app:examples kirwan}, we list the Horn triples as well as the Horn inequalities for the Kirwan cones up to $r=4$.

\counterwithin{equation}{subsection}
\section{Subspaces, flags, positions}\label{sec:linalg}

In this section, we study the geometry of subspaces and flags in more detail and supply proofs of some linear algebra facts used previously in \cref{sec:easy}.

\subsection{Schubert positions}

We start with some remarks on the Grassmannian $\grass(r,W)$, which is an irreducible algebraic variety on which the general linear group $\GL(W)$ acts transitively.
The stabilizer of a subspace $V \in \grass(r, W)$ is equal to the parabolic subgroup $P(V, W) = \{ \gamma \in \GL(W) : \gamma V \subseteq V \}$, with Lie algebra $\p(V, W) = \{ x \in \gl(W) : x V \subseteq V \}$.
Thus we obtain that
\[ \grass(r,W) = \GL(W) \cdot V \cong \GL(W)/P(V,W), \]
and we can identify the tangent space at $V$ with
\[ T_V \grass(r,W) = \gl(W) \cdot V \cong \gl(W)/\p(V,W) \cong \Hom(V,W/V). \]
If we choose a complement $Q$ of $V$ in $W$ then
\begin{equation}\label{eq:affine chart of grassmannian}
  \Hom(V,Q)\to\grass(r,W), \quad \phi\mapsto(\id+\phi)(V)
\end{equation}
parametrizes a neighborhood of $V$. 
This gives a system of affine charts in $\grass(r,W)$ isomorphic to $\C^{r(n-r)}$.
In particular, $\dim\grass(r,W)=r(n-r)$, a fact we use repeatedly in this article.

We now consider Schubert positions and the associated Schubert cells and varieties in more detail (\cref{def:schubert pos,def:schubert cell and variety}).
For all $\gamma \in \GL(W)$, we have the following equivariance property:
\begin{equation}
\label{eq:position equivariance}
  \Pos(\gamma^{-1} V, E) = \Pos(V, \gamma E),
\end{equation}
which in particular implies that
\begin{equation}
\label{eq:schubert cell equivariance}
  \gamma \Omega^0_I(E) = \Omega^0_I(\gamma E).
\end{equation}
Thus $\Omega^0_I(E)$ is preserved by the Borel subgroup $B(E) = \{ \gamma\in\GL(W) : \gamma E(i)\subseteq E(i) \; (\forall i) \}$, which is the stabilizer of the flag $E$.
We will see momentarily that $\Omega^0_I(E)$ is in fact a single $B(E)$-orbit.
We first state the following basic lemma, which shows that adapted bases (\cref{def:flag and adapted basis}) provide a convenient way of computing Schubert positions:

\begin{lemma}
\label{lem:adapted basis}
  Let $E$ be a flag on $W$, $\dim W = n$, $V \subseteq W$ an $r$-dimensional subspace, and $I \subseteq [n]$ a subset of cardinality $r$, with complement $I^c$.
  The following are equivalent:
  \begin{enumerate}
  \item \label{item:adapted basis i} $\Pos(V, E) = I$. 
  \item \label{item:adapted basis ii} For any ordered basis $(f(1),\dots,f(n))$ adapted to $E$, there exists a (unique) basis $(v(1),\dots,v(r))$ of $V$ of the form 
  \[ v(a) \in f(I(a)) + \linspan \{ f(i) : i \in I^c, i < I(a) \}. \]
  \item \label{item:adapted basis iii} There exists an ordered basis $(f(1),\dots,f(n))$ adapted to $E$ such that $\{f(I(1)),\dots,f(I(r))\}$ is a basis of $V$.
  \end{enumerate}
\end{lemma}

The proof of \cref{lem:adapted basis} is left as an exercise to the reader.
%
Clearly, $B(E)$ acts transitively on the set of ordered bases adapted to $E$.
Thus, \cref{lem:adapted basis},~\ref{item:adapted basis iii} shows that $\Omega^0_I(E)$ is a single $B(E)$-orbit.
That is, just like Grassmannian itself, each Schubert cell is a homogeneous space.
In particular, $\Omega^0_I(E)$ and its closure $\Omega_I(E)$ (\cref{def:schubert cell and variety}) are both irreducible algebraic varieties.

\begin{example*}
  Consider the flag $E$ on $W=\C^4$ with adapted basis $(f(1),\dots,f(4))$, where
  $f(1) = e(1) + e(2) + e(3)$,
  $f(2) = e(2) + e(3)$,
  $f(3) = e(3) + e(4)$,
  $f(4) = e(4)$.
  If $V = \linspan \{ e(1), e(2) \}$ then $\Pos(V, E) = \{2,4\}$, while $\Pos(V, E_0) = \{1,2\}$ for the standard flag $E_0$ with adapted basis $(e(1),e(2),e(3),e(4))$.

  Note that the basis $(v(1),v(2))$ of $V$ given by
  $v(1) = f(2) - f(1) = e(1)$ and
  $v(2) = f(4) - f(3) + f(1) = e(1) + e(2)$
  satisfies the conditions in \cref{lem:adapted basis},~\ref{item:adapted basis ii}.
  It follows that $(f(1),v(1),f(3),v(2))$
  is an adapted basis of $E$ that satisfies the conditions in~\ref{item:adapted basis iii}.
\end{example*}

The following lemma characterizes each Schubert variety explicitly as a union of Schubert cells:

\begin{lemma}
\label{lem:schubert variety characterization}
  Let $E$ be a flag on $W$, $\dim W = n$, and $I \subseteq [n]$ a subset of cardinality $r$.
  Then,
  \printschubertvarietycharacterization{equation}{\label{eq:schubert variety characterization}}
  where the union is over all subsets $I' \subseteq [n]$ of cardinality $r$ such that $I'(a) \leq I(a)$ for $a\in[r]$.
\end{lemma}
\begin{proof}
  Recall that $\Omega_I(E)$ can be defined as the Euclidean closure of $\Omega^0_I(E)$.
  Thus let $(V_k)$ denote a convergent sequence of subspaces in $\Omega^0_I(E)$ with limit some $V \in \grass(r, W)$.
  Then $\dim E(I(a)) \cap V \geq \dim E(I(a)) \cap V_k$ for sufficiently large $k$, since intersections can only become larger in the limit, but $\dim E(I(a)) \cap V_k = a$ for all $k$.
  It follows that $\Pos(V, E)(a) \leq I(a)$.

  Conversely, suppose that $V' \in \Omega^0_{I'}(E)$, where $I'(a) \leq I(a)$ for all $a$.
  Let $a'$ denote the minimal integer such that $I'(a) = I(a)$ for $a=a'+1,\dots,r$.
  We will show that $V' \in \Omega_I(E)$ by induction on $a'$.
  If $a' = 0$ then $I' = I$ and there is nothing to show. 
  Otherwise, let $(f'(1),\dots,f'(n))$ denote an adapted basis for $E$ such that $v'(a) = f'(I'(a))$ is a basis of $V'$ (as in~\ref{item:adapted basis iii} of \cref{lem:adapted basis}).
  For each $\varepsilon>0$, consider the subspace $V_\varepsilon$ with basis vectors $v_\varepsilon(a) = v'(a)$ for all $a\neq a'$ together with $v_\varepsilon(a') := v'(a') + \varepsilon f'(I(a'))$.
  Then the space $V_\varepsilon$ is of dimension $r$ and in position $\{ I'(1),\dots,I'(a'-1),I(a'),\dots,I(r) \}$ with respect to $E$.
  By the induction hypothesis, $V_\varepsilon \in \Omega_I(E)$ for any $\varepsilon > 0$, and thus $V' \in \Omega_I(E)$ as $V_\varepsilon\to V'$ for $\varepsilon\to0$.
\end{proof}

We now compute the dimensions of Schubert cells and varieties.
This is straightforward from \cref{lem:adapted basis}, however it will be useful to make a slight detour and introduce some notation.
This will allow us to show that we can exactly parametrize $\Omega^0_I(E)$ by a unipotent subgroup of $B(E)$, which in particular shows that it is an affine space.

Choose an ordered basis $(f(1),\dots,f(n))$ that is adapted to $E$.
Then $V := \linspan \{ f(i) : i \in I \} \in \Omega^0_I(E)$.
By \Cref{lem:adapted basis},~\ref{item:adapted basis ii} any $V\in\Omega^0_I(E)$ is of this form.
Now define
\begin{align*}
  &\quad\; \Hom_E(V,W/V) := \{ \phi\in\Hom(V,W/V) : \phi(E(i)\cap V) \subseteq (E(i)+V)/V \text{~for~}i\in[n] \} \\
  &= \{ \phi\in\Hom(V,W/V) : \phi(f(I(a))) \subseteq \linspan \{ f(I^c(b)) + V : b \in [I(a)-a] \} \text{~for~}a\in[r] \}
\end{align*}
where the $f(j) + V$ for $j\in I^c$ form a basis of $W/V$.
In particular, $\Hom_E(V,W/V)$ is of dimension $\sum_{a=1}^r (I(a) - a)$.
Using this basis, we can identify $W/V$ with $Q := \linspan \{ f(j) : j\in I^c \}$.
Then $W = V \op Q$ and we can identify $\Hom_E(V,W/V)$ with
\[ H_E(V,Q) := \{ \phi\in\Hom(V,Q) : \phi(f(I(a))) \subseteq \linspan \{ f(I^c(b)) : b \in [I(a)-a] \} \text{~for~}a\in[r] \}. \]
\Cref{lem:adapted basis},~\ref{item:adapted basis ii} shows that for any $\phi\in H_E(V,Q)$, we obtain a distinct subspace $(\id+\phi)(V)$ in $\Omega^0_I(E)$, and that all subspaces in $\Omega^0_I(E)$ can obtained in this way.
Thus, $\Omega^0_I(E)$ is contained in the affine chart $\Hom(V,Q)$ of the Grassmannian described in~\eqref{eq:affine chart of grassmannian} and isomorphic to the linear subspace $H_E(V,Q)$ of dimension $\dim I$.
We define a corresponding unipotent subgroup,
\[ U_E(V,Q) := \{ u_\phi = \id + \phi = \begin{pmatrix}\id_V & 0 \\ \phi & \id_Q \end{pmatrix} \in \GL(W) : \phi\in H_E(V,Q) \}. \]
Thus we obtain the following lemma:

\begin{lemma}
\label{lem:dim schubert cell}
  Let $E$ be a flag on $W$, $\dim W = n$, $I \subseteq [n]$ a subset of cardinality $r$, $V\in\Omega^0_I(E)$, and $Q$ as above.
  Then we can parametrize $H_E(V,Q) \cong U_E(V,Q) \cong \Omega^0_I(E) = U_E(V,Q)V$, hence $H_E(V,Q)\cong T_V \Omega^0_I(E)$ and
\printdimschubertcell{equation}{\label{eq:dim schubert cell}}{\dim H_E(V,Q)=}
\end{lemma}

It will be useful to rephrase the above to obtain a parametrization of $\Omega^0_I(E)$ in terms of the fixed subspaces
\begin{equation}\label{eq:V_0 and Q_0}
\begin{aligned}
  V_0 &:= \linspan \{ f(1), \dots, f(r) \} = E(r), \\
  Q_0 &:= \linspan \{ \bar f(1), \dots, \bar f(n-r) \},
\end{aligned}
\end{equation}
where the $\bar f(i) := f(r+i)$ for $i\in[n-r]$ form a basis of $Q_0$.
Then $W = V_0 \op Q_0$.

\begin{definition}
\label{def:shuffle permutation}
  Let $I \subseteq [n]$ be a subset of cardinality $r$.
  The \emph{shuffle permutation} $\sigma_I \in S_n$ is defined by
  \[
    \sigma_I(a) = \begin{cases}
     I(a) & \text{for $a=1,\dots,r$}, \\
     I^c(a-r) & \text{for $a=r+1,\dots,n$}.
  \end{cases}
  \]
  and $w_I \in \GL(W)$ is the corresponding permutation operator with respect to the adapted basis $(f(1),\dots,f(n))$, defined as $w_I \, f(i) := f(\sigma_I^{-1}(i))$ for $i\in[n]$.
\end{definition}

Then $V_0 = w_I V$, where $V = \linspan \{ f(i) : i\in I \} \in \Omega^0_I(E)$ as before, and so
\begin{equation*}
  V_0 \in w_I \Omega^0_I(E) = \Omega^0_I(w_I E)
\end{equation*}
using~\eqref{eq:schubert cell equivariance}.
The translated Schubert cell can be parametrized by
\begin{align*}
  H_{w_I E}(V_0, Q_0) = \{ \phi\in\Hom(V_0,Q_0) : \phi(f(a)) \subseteq \linspan \{ \bar f(1), \dots, \bar f(I(a) - a)\} \text{~for~} a\in[r]\},
\end{align*}
where we identify $Q_0\cong W/V_0$.
We thus obtain the following consequence of \cref{lem:dim schubert cell}:

\begin{corollary}
\label{cor:schubert cell parametrization}
  Let $E$ be a flag on $W$, $\dim W = n$, $I \subseteq [n]$ of cardinality $r$, and $V\in\Omega^0_I(E)$.
  Moreover, define $w_I$ as above for an adapted basis.
  Then,
  \[ \Omega^0_I(E) = w_I^{-1} \Omega^0_I(w_I E) = w_I^{-1} U_{w_I E}(V_0, Q_0) V_0. \]
\end{corollary}

\begin{example*}[$r=3$,$n=4$]
  Let $I=\{1,3,4\}$ and $E_0$ the standard flag on $W=\C^4$, with its adapted basis $(e(1),\dots,e(4))$.
  Then $\sigma_I = \bigl(\begin{smallmatrix}1 & 2 & 3 & 4 \\ 1 & 3 & 4 & 2\end{smallmatrix}\bigr)$,
  \begin{align*}
    w_I^{-1} &= \begin{pmatrix}1 & 0 & 0 & 0 \\ 0 & 0 & 0 & 1 \\ 0 & 1 & 0 & 0 \\ 0 & 0 & 1 & 0\end{pmatrix} \\
  \end{align*}
  and $V = w_I^{-1} V_0 = \linspan \{ e(1), e(3), e(4) \}$ is indeed in position $I$ with respect to $E_0$, in agreement with the preceding discussion.
  Moreover,
  \begin{align*}
    H_{w_I E_0}(V_0,Q_0) &= \{ \begin{pmatrix}
      0 & * & *
    \end{pmatrix} \} \subseteq \Hom(\C^3, \C^1), \\
    U_{w_I E_0}(V_0,Q_0) &= \{ \begin{pmatrix}
        1 & 0 & 0 & 0 \\
        0 & 1 & 0 & 0 \\
        0 & 0 & 1 & 0 \\
        0 & * & * & 1
    \end{pmatrix} \} \subseteq \GL(4),
  \end{align*}
  and so \cref{cor:schubert cell parametrization} asserts that
  \[
     \Omega^0_I(E_0)
     = w_I^{-1} U_{w_I E_0}(V_0,Q_0) \linspan \{ e(1), e(2), e(3) \}
     = \linspan \{
       \begin{pmatrix}1 \\ 0 \\ 0 \\ 0\end{pmatrix},
       \begin{pmatrix}0 \\ * \\ 1 \\ 0\end{pmatrix},
       \begin{pmatrix}0 \\ * \\ 0 \\ 1\end{pmatrix}
      \},
  \]
  which agrees with \cref{lem:adapted basis}.
\end{example*}


\subsection{Induced flags and positions}

The space $\Hom_E(V,W/V)$ can be understood more conceptually as the space of homomorphisms that respect the filtrations $E(i) \cap V$ and $(E(i) + V)/V$ induced by the flag $E$.
Here we have used the following concept:

\begin{definition}
  A \emph{(complete) filtration} $F$ on a vector space $V$ is a chain of subspaces
  \[ \{0\} = F(0) \subseteq F(1) \subseteq \cdots \subseteq F(i) \subseteq F({i+1}) \subseteq \dots \subseteq F(l) = V, \]
  such that the dimensions increase by no more than one, i.e., $\dim F(i+1)\leq\dim F(i)+1$ for all $i=0,\dots,l-1$.
  Thus distinct subspaces in a filtration determine a flag.
\end{definition}

Given a flag $E$ on $W$ and a subspace $V\subseteq W$, we thus obtain an induced flag $E^V$ on $V$ from the distinct subspaces in the sequence $E(i) \cap V$, $i=0,\dots,n$.
We may also induce a flag $E_{W/V}$ on the quotient $W/V$ from the distinct subspaces in the sequence $(E(i) + V)/V$.
These flags can be readily computed from the Schubert position of $V$:

\begin{lemma}
\label{lem:induced flags}
  Let $E$ be a flag on $W$, $\dim W = n$, and $V \subseteq W$ an $r$-dimensional subspace in position $I = \Pos(V, E)$.
  Then the induced flags $E^V$ on $V$ and $E_{W/V}$ on $W/V$ are given by
  \begin{align*}
    E^V(a) &= E(I(a)) \cap V, \\
    E_{W/V}(b) &= (E(I^c(b)) + V)/V
  \end{align*}
  for $a\in[r]$ and $b\in[n-r]$, where $I^c$ denotes the complement of $I$ in $[n]$.
\end{lemma}
\begin{proof}
  Using an adapted basis as in \cref{lem:adapted basis},~\ref{item:adapted basis iii}, it is easy to see that $\dim E(i) \cap V = \lvert [i]\cap I \rvert$ and therefore that $\dim (E(i) + V)/V = \lvert [i]\cap I^c \rvert$.
  Now observe that $\lvert [i]\cap I\rvert = a$ if and only if $I(a) \leq i < I(a+1)$, while $\lvert [i]\cap I^c \rvert = b$ if and only if $I^c(b) \leq i < I^c(b+1)$.
  Thus we obtain the two assertions.
\end{proof}

We can use the preceding result to describe $\Hom_E(V,W/V)$ in terms of flags rather than filtrations and without any reference to the ambient space $W$.

\begin{definition}\label{def:H_I for flags}
  Let $V$ and $Q$ be vector spaces of dimension $r$ and $n-r$, respectively, $I \subseteq[n]$ a subset of cardinality $r$, $F$ a flag on $V$ and $G$ a flag on $Q$.
  We define
  \[ H_I(F,G) := \{\phi\in\Hom(V,Q) : \phi(F(a)) \subseteq G(I(a)-a) \}, \]
  which we note is well-defined by
  \begin{equation}
  \label{eq:nondecreasing}
    0 \leq I(a) - a \leq I(a+1) - (a+1) \leq n-r
    \qquad (a=1,\dots,r-1).
  \end{equation}
\end{definition}
It now easily follows from \cref{lem:dim schubert cell,lem:induced flags} that  
\begin{equation}
\label{eq:tangent space schubert cell}
  T_V \Omega^0_I(E) \cong \Hom_E(V, W/V) = H_I(E^V, E_{W/V}).
\end{equation}
As a consequence: 
\begin{equation}\label{eq:H_wIE}
  H_{w_I E}(V_0, Q_0) = H_I((w_I E)^{V_0}, (w_I E)_{Q_0}) = H_I(E^{V_0}, E_{Q_0})
\end{equation}
We record the following equivariance property:
\begin{lemma}
\label{lem:H_I borel invariance}
  Let $F$ be a flag on $V$ and $G$ a flag on $Q$.
  If $\phi\in H_I(F,G)$, $a\in GL(V)$ and $d\in GL(Q)$, then $d\phi a^{-1}\in H_I(aF,dG)$.
  In particular, $H_I(F, G)$ is stable under right multiplication by the Borel subgroup $B(F)$ and left multiplication by the Borel subgroup $B(G)$.
\end{lemma}

We now compute the position of subspaces and subquotients with respect to induced flags.
Given subsets $I \subseteq [n]$ of cardinality $r$ and $J \subseteq [r]$ of cardinality $d$, we recall that we had defined their \emph{composition} $IJ$ in \cref{sec:easy} as the subset
\[ IJ = \bigl\{ I(J(1)) < \dots < I(J(d)) \bigr\} \subseteq [n]. \]
We also define their \emph{quotient} to be the subset
\[ I/J = \bigl\{ I(J^c(b)) - J^c(b) + b \;:\; b \in [r-d] \bigr\} \subseteq [n-d], \]
where $J^c$ denotes the complement of $J$ in $[r]$.
It follows from~\eqref{eq:nondecreasing} that $I/J$ is indeed a subset of $[n-d]$.


The following lemma establishes the `chain rule' for positions:

\begin{lemma}
\label{lem:chain rule}
  Let $E$ be a flag on $W$, $S \subseteq V \subseteq W$ subspaces, and $I = \Pos(V, E)$, $J = \Pos(S, E^V)$ their relative positions.
  Then there exists an adapted basis $(f(1),\dots,f(n))$ for $E$ such that $\{f(I(a))\}$ is a basis of $V$ and $\{f(IJ(b))\}$ a basis of $S$.
  In particular,
\printchainrulepos{equation}{\label{eq:chain rule pos}}{IJ=}
\end{lemma}
\begin{proof}
  According to \cref{lem:adapted basis},~\ref{item:adapted basis iii}, there exists an adapted basis $(f(1),\dots,f(n))$ for $E$ such that $(f(I(1)), \dots, f(I(r)))$ is a basis of $V$, where $r=\dim V$.
  By \cref{lem:induced flags}, this ordered basis is in fact adapted to the induced flag $E^V$.
  Thus we can apply \cref{lem:adapted basis},~\ref{item:adapted basis ii} to $E^V$ and the subspace $S \subseteq V$ to obtain a basis $(v(1),\dots,v(s))$ of $S$ of the form
  \[ v(b) \in f(IJ(b)) + \linspan \{ f(I(a)) : a \in J^c, a < J(b) \}. \]
  It follows that the ordered basis $(f'(1),\dots,f'(n))$ obtained from $(f(1),\dots,f(n))$ by replacing $f(IJ(b))$ with $v(b)$ has all desired properties.
  We now obtain the chain rule, $\Pos(S,E)=IJ$, as a consequence of \cref{lem:adapted basis},~\ref{item:adapted basis iii} applied to $f'$ and $S\subseteq W$.
\end{proof}

We can visualize the subsets $IJ, IJ^c \subseteq [n]$ and $I/J \subseteq [n-d]$ as follows.
Let $L$ denote the string of length $n$ defined by putting the symbol $s$ at the positions in $IJ$, $v$ at those in $I \setminus IJ = IJ^c$, and $w$ at all other positions.
This mirrors the situation in the preceding \cref{lem:chain rule}, where the adapted basis $(f(1),\dots,f(n))$ can be partitioned into three sets according to membership in $S$, $V \setminus S$, and $W \setminus V$.
Now let $L'$ denote the string of length $n-d$ obtained by deleting all occurrences of the symbol $s$.
Thus the remaining symbols are either $v$ or $w$, i.e., those that were at locations $(IJ)^c$ in $L$.
We observe that the $b$-th occurrence of $v$ in $L$ was at location $IJ^c(b)$, where it was preceded by $J^c(b) - b$ occurrences of $s$.
Thus the occurrences of $v$ in $L'$ are given precisely by the quotient position, $(I/J)(b) = IJ^c(b) - (J^c(b) - b)$.

\begin{example*}
  If $n=6$, $I = \{1,3,5,6\}$ and $J = \{2,4\}$, then $IJ = \{3,6\}$ and $L = (v,w,s,w,v,s)$.
  It follows that $L' = (v,w,w,v)$ and hence the symbols $v$ appear indeed at positions $I/J = \{1,4\}$.
\end{example*}

We thus obtain the following recipe for computing positions of subquotients:

\begin{lemma}
\label{lem:quotient rule}
  Let $E$ be a flag on $W$ and $S\subseteq V\subseteq W$ subspaces.
  Then,
  \[ \Pos(V/S, E_{W/S}) = \Pos(V, E) / \Pos(S, E^V). \]
\end{lemma}
\begin{proof}
  Let $I = \Pos(V, E)$ and $J = \Pos(S, E^V)$.
  According to \cref{lem:chain rule}, there exists an adapted basis $(f(1),\dots,f(n))$ of $E$ such that $\{f(I(a))\}$ is a basis of $V$ and $\{f(IJ(b))\}$ a basis of $S$.
  This shows not only that $\{f(IJ^c(b))\}$ is a basis of $V/S$, but also, by \cref{lem:induced flags}, that $(f((IJ)^c(b)))$ is an adapted basis for $E_{W/S}$.
  Clearly, $IJ^c \subseteq (IJ)^c$, and the preceding discussion showed that the location of the $IJ^c$ in $(IJ)^c$ is exactly equal to the quotient position $I/J$.
  Thus we conclude from \cref{lem:adapted basis},~\ref{item:adapted basis iii} that $\Pos(V/S, E_{W/S}) = I/J$.
\end{proof}

One last consequence of the preceding discussion is the following lemma:

\begin{lemma}
\label{lem:IJ^c}
  Let $E$ be a flag on $W$, $\dim W = n$, $S \subseteq V \subseteq W$ subspaces, and $I = \Pos(V, E)$, $J = \Pos(S, E^V)$.
  Then $F(i) := \bigl((E(i) \cap V) + S\bigr)/S$ is a filtration on $V/S$, and
  \[ IJ^c(b) = \min \{ i \in [n] : \dim F(i) = b \} \]
  for $b=1,\dots,\dim V/S$.
\end{lemma}
\begin{proof}
  As in the preceding proof, we use the adapted basis $(f(1),\dots,f(n))$ from \cref{lem:chain rule}.
  Then $\{f(IJ^c(b))\}$ is a basis of $V/S$ and $F(i) = \linspan \{ f(IJ^c(b)) : b\in[q], IJ^c(b) \leq i \}$, and this implies the claim.
\end{proof}

The following corollary uses \cref{lem:IJ^c} to compare filtrations for a space that is isomorphic to a subquotient in two different ways, $(S_1+S_2)/S_2 \cong S_1/(S_1\cap S_2)$.

\begin{corollary}
\label{cor:slope preliminary}
  Let $E$ be a flag on $W$, $\dim W = n$, and $S_1,S_2\subseteq W$ subspaces.
  Furthermore, let $J=\Pos(S_1,E)$, $K=\Pos(S_1\cap S_2, E^{S_1})$, $L=\Pos(S_1+S_2,E)$, and $M=\Pos(S_2, E^{S_1+S_2})$.
  Then both $JK^c$ and $LM^c$ are subsets of $[n]$ of cardinality $q := \dim S_1/(S_1\cap S_2) = \dim (S_1+S_2)/S_2$, and
  \[ JK^c(b) \leq LM^c(b) \]
  for $b\in[q]$.
\end{corollary}
\begin{proof}
  Consider the filtration $F(j) := \bigl( (E(j)\cap S_1) + (S_1\cap S_2) \bigr)/(S_1\cap S_2)$ of $S_1/(S_1\cap S_2)$ and the filtration $F'(j) := \bigl( (E(j)\cap (S_1+S_2)) + S_2 \bigr)/S_2$ of $(S_1+S_2)/S_2$.
  If we identify $S_1/(S_1\cap S_2) \cong (S_1+S_2)/S_2$, then $F(j)$ gets identified with the subspace $\bigl((E(j)\cap S_1) + S_2\bigr)/S_2$ of $F'(j)$.
  It follows that
  \[ JK^c(b) = \min \{ j\in[n]: \dim F(j) = b \} \geq \min \{ j\in[n]: \dim F'(j) = b \} = LM^c(b), \]
  where we have used \cref{lem:IJ^c} twice.
\end{proof}

We now compute the dimension of quotient positions:

\begin{lemma}
\label{lem:quotient dim}
  Let $I \subseteq [n]$ be a subset of cardinality $r$ and $J \subseteq [r]$ a subset of cardinality $d$.
  Then:
  \[ \dim I/J = \dim I + \dim J - \dim IJ \]
\end{lemma}
\begin{proof}
  Straight from the definition of dimension and quotient position,
  \begin{align*}
  &\quad \dim I/J
  = \sum_{b=1}^{r-d} I(J^c(b)) - \sum_{b=1}^{r-d} J^c(b) \\
  &= \bigl( \sum_{a=1}^r I(a) - \sum_{b=1}^d I(J(b)) \bigr) - \bigl( \sum_{a=1}^r a - \sum_{b=1}^d J(b) \bigr) \\
  &= \sum_{a=1}^r (I(a) - a) + \sum_{b=1}^d (J(b) - b) - \sum_{b=1}^d (I(J(b)) - b) \\
  &= \dim I + \dim J - \dim IJ. \qedhere
  \end{align*}
\end{proof}

Lastly, given subsets $I\subseteq[n]$ of cardinality $r$ and $J\subseteq[r]$ of cardinality $d$, we define
\[ I^J = \bigl\{ I(J(b)) - J(b) + b \;:\; b \in [d] \bigr\} \subseteq [n-(r-d)]. \]
Clearly, $I^J = I/J^c$, but we prefer to introduce a new notation to avoid confusion, since the role of $I^J$ will be quite different.
Indeed, $I^J$ is related to composition, as is indicated by the following lemmas:

\begin{lemma}
\label{lem:exp dim chain rule}
  Let $I\subseteq[n]$ be a subset of cardinality $r$, $J\subseteq[r]$ a subset of cardinality $d$.
  Then,
  \[ \dim I^J K - \dim K = \dim I(JK) - \dim JK \]
  for any subset $K\subseteq[d]$.
  In particular, $\dim I^J = \dim IJ - \dim J$.
\end{lemma}
\begin{proof}
  Let $m$ denote the cardinality of $K$. Then:
  \begin{align*}
  &\quad \dim I^J K - \dim K
  = \sum_{c=1}^m \bigl( I^J(K(c)) - K(c) \bigr) \\
  &= \sum_{c=1}^m \bigl( I(J(K(c)) - J(K(c)) \bigr)
  = \dim I(JK) - \dim JK. \qedhere
  \end{align*}
\end{proof}



\begin{lemma}
\label{lem:exp vs composition}
  Let $I\subseteq[n]$ be a subset of cardinality $r$, $\phi \in \Hom(V,Q)$, and $F$ a flag on $V$.
  Let $S=\ker\phi$ denote the kernel, $J:=\Pos(S,F)$ its position with respect to $F$, and $\bar\phi\in\Hom(V/S,Q)$ the corresponding injection.

  Then $\phi\in H_I(F, G)$ if and only if $\bar\phi\in H_{I/J}(F_{V/S}, G)$.
  In this case, we have for all $\psi\in H_J(F^S, F_{V/S})$ that $\bar\phi\psi\in H_{I^J}(F^S, G)$.
\end{lemma}
\begin{proof}
  For the first claim, note that if $\phi\in H_I(F,G)$ then
  \[ \bar\phi(F_{V/S}(b)) = \phi(F(J^c(b))) \subseteq G(I(J^c(b)) - J^c(b)) = G((I/J)(b) - b). \]
  Conversely, if $\bar\phi\in H_{I/J}(F_{V/S}, G)$, then this shows that
  \[ \phi(F(a)) \subseteq G(I(a) - a) \]
  for all $a = J^c(b)$, and hence for all $a$, since $\phi(F(J^c(b))) = \dots = \phi(F(J^c(b+1)-1))$.

  For the second, we use $H_J(F^S, F_{V/S}) = \Hom_F(S, V/S)$ (\cref{eq:tangent space schubert cell}) and compute
  \begin{align*}
  &\quad \bar\phi\psi(F^S(a))
  = \bar\phi\psi(F(J(a)) \cap S)
  \subseteq \bar\phi((F(J(a)) + S)/S) \\
  &= \phi(F(J(a))) \subseteq G(I(J(a)) - J(a)) = G(I^J(a) - a). \qedhere
  \end{align*}

\end{proof}

\subsection{The flag variety}

The Schubert cells of the Grassmannian were defined by fixing a flag and classifying subspaces according to their Schubert position.
As we will later be interested in intersections of Schubert cells for different flags, it will be useful to also consider variations of the flag for a fixed subspace.

Let $\flag(W)$ denote the \emph{(complete) flag variety}, defined as the space of (complete) flags on $W$.
It is a homogeneous space with respect to the transitive $\GL(W)$-action, so indeed an irreducible variety.

\begin{definition}
  Let $V \subseteq W$ be a subspace, $\dim V=r$, $\dim W=n$, and $I \subseteq [n]$ a subset of cardinality $r$.
  We define
  \[ \flag^0_I(V,W) = \{ E \in \flag(W) : \Pos(V, E) = I \}, \]
  and $\flag_I(V,W)$ as its closure in $\flag(W)$ (in either the Euclidean or the Zariski topology).
\end{definition}

We have the following equivariance property as a consequence of~\eqref{eq:position equivariance}:
For all $\gamma \in \GL(W)$,
\begin{equation}
\label{eq:flag equivariance}
  \gamma \flag^0_I(V, W)
  = \flag^0_I(\gamma V,W).
\end{equation}
In particular, $\flag^0_I(V,W)$ and $\flag_I(V,W)$ are stable under the action of the parabolic subgroup $P(V, W) = \{ \gamma \in \GL(W) : \gamma V \subseteq V \}$, which is the stabilizer of $V$.

We will now show that $\flag^0_I(V,W)$ is in fact a single $P(V,W)$-orbit.
This implies that both $\flag^0_I(V,W)$ and $\flag_I(V,W)$ are irreducible algebraic varieties.

\begin{definition}
\label{def:G_I}
  Let $E$ be a flag on $W$, $\dim W=n$, $V_0 = E(r)$, and $I \subseteq [n]$ a subset of cardinality $r$.
  We define
  \[ G_I(V_0,E) := \{ \gamma\in\GL(W) : \gamma E \in \flag^0_I(V_0,W) \}, \]
  so that $\flag^0_I(V_0,W) \cong G_I(V_0,E)/B(E)$.
\end{definition}

\begin{lemma}
\label{lem:G_I coarse}
  Let $E$ be a flag on $W$, $\dim W=n$, $V_0 = E(r)$, and $I \subseteq [n]$ a subset of cardinality $r$.
  Then, $G_I(V_0,E) = P(V_0,W) w_I B(E)$.
  In particular, $\flag^0_I(V_0,W) = P(V_0,W) w_I E$.
\end{lemma}
\begin{proof}
  Let $\gamma\in\GL(W)$.
  Then,
  \begin{align*}
    \gamma\in G_I(V_0,E)
  \Leftrightarrow V_0\in\Omega^0_I(\gamma E)=\gamma\Omega^0_I(E)=\gamma B(E)w_I^{-1}V_0
  \Leftrightarrow \gamma \in P(V_0,W)w_I B(E),
  \end{align*}
  where we have used that $\Omega^0_I(E) = B(E) w_I^{-1}V_0$.
\end{proof}

We now derive a more precise parametrization of $\flag^0_I(V_0,W)$.

\begin{lemma}
\label{lem:G_I precise}
  Let $E$ be a flag on $W$, $\dim W=n$, $V_0$ and $Q_0$ as in~\eqref{eq:V_0 and Q_0}, and $I \subseteq [n]$ a subset of cardinality $r$.
  Then we have that $G_I(V_0,E) = P(V_0,W) U_{w_I E}(V_0,Q_0) w_I$.
\end{lemma}
\begin{proof}
  Let $\gamma\in\GL(W)$. Then,
  \begin{align*}
  &\qquad \gamma\in G_I(V_0,E)
  \;\Leftrightarrow\; V_0\in\Omega^0_I(\gamma E)=\gamma\Omega^0_I(E) = \gamma w_I^{-1} U_{w_I E}(V_0, Q_0) V_0 \\
  &\Leftrightarrow \gamma\in P(V_0,W)U_{w_I E}(V_0,Q_0)w_I,
  \end{align*}
  since $\Omega^0_I(E) = w_I^{-1} U_{w_I E}(V_0, Q_0) V_0$ (\cref{cor:schubert cell parametrization}).
\end{proof}

In coordinates, using $W=V_0\op Q_0$ and~\eqref{eq:H_wIE}, we obtain that
\begin{align*}
  G_I(V_0,E)
&= \{ \begin{pmatrix}a & b\\0 & d\end{pmatrix} \begin{pmatrix}1 & 0\\\phi & 1\end{pmatrix} : a\in\GL(V_0), d\in\GL(Q_0), \phi\in H_I(E^{V_0}, E_{Q_0}) \} w_I \\
&= \{ \begin{pmatrix}a & b\\c & d\end{pmatrix} : a-bd^{-1}c\in\GL(V_0), d\in\GL(Q_0), d^{-1}c\in H_I(E^{V_0}, E_{Q_0}) \} w_I.
\end{align*}
In particular, $\dim G_I(V_0,E) = \dim P(V_0,W) + \dim I$.
This allows us to compute the dimension of the subvarieties $\flag^0_I(V,W)$ and to relate their codimension to the codimension of the Schubert cells of the Grassmannian:

\begin{corollary}
  Let $V\subseteq W$ be a subspace, $\dim W=n$, $\dim V=r$, and $I \subseteq [n]$ a subset of cardinality $r$.
  Then,
  \begin{equation}
  \label{eq:dim flag^0_I}
    \dim \flag^0_I(V,W)
  = \dim \flag_I(V,W)
  = \dim \flag(V) + \dim \flag(Q) + \dim I
  \end{equation}
  and
  \begin{equation}
   \label{eq:codim flag vs grass}
   \dim \flag(W) - \dim \flag^0_I(V,W) = \dim\grass(r,W) - \dim I.
  \end{equation}
\end{corollary}
\begin{proof}
  Without loss of generality, we may assume that $V=V_0=E(r)$ for some flag $E$ on $W$.
  Then, $\flag_I^0(V_0,W)\cong G_I(V_0,E)/B(E)$ and hence
  \begin{align*}
  &\quad  \dim \flag_I^0(V_0,E)
  = \dim P(V_0,W) + \dim I - \dim B(E)\\
  &= \dim\GL(W) - \dim\grass(r,W) + \dim I - \dim B(E) \\
  &= \dim\flag(W) - \dim\grass(r,W) + \dim I
  \end{align*}
  since $\grass(r,W)\cong\GL(W)/P(V_0,W)$ and $\flag(W)=\GL(W)/B(E)$.
  This establishes~\eqref{eq:codim flag vs grass}.
  On the other hand, a direct calculation shows that
  \[
    \dim\flag(W)-\dim\grass(r,W) = \dim\flag(V)+\dim\flag(Q),
  \]
  so we also obtain~\eqref{eq:dim flag^0_I}.
\end{proof}

At last, we study the following set of flags on the target space of a given homomorphism:

\begin{definition}
\label{def:flags for injective morphism}
  Let $V$, $Q$ be vector spaces of dimension $r$ and $n-r$, respectively, and $I\subseteq[n]$.
  Moreover, let $F$ be a flag on $V$ and $\phi\in\Hom(V,Q)$ an \emph{injective} homomorphism.
  We define
  \[ \flag^0_I(F,\phi) := \{ G \in \flag(Q) : \phi\in H_I(F,G) \} \]
  where we recall that $H_I(F,G)$ was defined in~\cref{def:H_I for flags}.
\end{definition}

It is clear that $I(a)\geq 2a$ is necessary and sufficient for $\flag_I^0(F,\phi)$ to be nonempty.

\begin{example*}[r=3,n=8]
  Let $V_0\cong\C^3$, with basis $e(1),\dots,e(3)$, and $Q_0\cong\C^5$, with basis $\bar e(1),\dots,\bar e(5)$.
  Take $\phi\colon V_0\to Q_0$ to be the canonical injection and let $F_0$ denote the standard flag on $V_0$.
  For $I=\{3,4,7\}$, $G\in\flag^0_I(F_0,\phi)$ if and only if
  \[ \C\bar e(1) \subseteq G(2),  \quad \C \bar e(1) \op \C\bar e(2) \subseteq G(2), \quad \C \bar e(1)\op\C \bar e(2)\op\C \bar e(3) \subseteq G(4). \]
  For example, the standard flag $G_0$ on $Q_0$ is a point in $\flag^0_I(F_0,\phi)$.

  On the other hand, if $I=\{2,3,7\}$ then we obtain the condition $\C \bar e(1) \op \C\bar e(2) \subseteq G(1)$ which can never be satisfied.
  Thus in this case $\flag^0_I(F_0,\phi)=\emptyset$.
\end{example*}

In the following lemma we show that $\flag_I^0(F,\phi)$ is a smooth variety and compute its dimension.

\begin{lemma}
\label{lem:image flags}
  Let $V$, $Q$ be vector spaces of dimension $r$ and $n-r$, respectively, and $I\subseteq[n]$ a subset of cardinality $r$.
  Moreover, let $F$ be a flag on $V$ and $\phi\in\Hom(V,Q)$ an \emph{injective} homomorphism.
  If $\flag_I^0(F,\phi)$ is nonempty, that is, if $I(a)\geq2a$ for all $a\in[r]$, then it is a smooth irreducible subvariety of $\flag(Q)$ of dimension
  \[ \dim \flag_I^0(F,\phi) = \dim \flag(Q_0) + \dim I - r(n-r). \]
\end{lemma}
\begin{proof}
  Without loss of generality, we may assume that $V=V_0\cong\C^r$, $Q=Q_0\cong\C^{n-r}$, that $F=F_0$ is the standard flag on $V_0$ and $\phi$ the canonical injection $\C^r\to\C^{n-r}$.
  Then the standard flag $G_0$ on $Q_0$ is an element of $\flag^0_I(F_0,\phi)$.
  We will show that
  \[ M_I := \{ h\in\GL(Q_0) : h \, G_0 \in \flag^0_I(F_0,\phi) \} \]
  is a subvariety of $\GL(Q_0)$ and compute its dimension.
  Note that $h \in M_I$ if and only if $h^{-1}\phi\in H_I(F_0, G_0)$.
  We now identify $V_0$ with its image $\phi(V_0)$ and denote by $R_0\cong\C^{n-2r}$ its standard complement in $Q_0$.
  Thus $Q_0 \cong V_0 \op R_0$ and we can think of $h^{-1}\in\GL(Q_0)$ as a block matrix
  \[ h^{-1} = \begin{pmatrix} A & B \end{pmatrix} \]
  where $A\in\Hom(V_0,Q_0)$ and $B\in\Hom(R_0,Q_0)$.
  The condition $h^{-1}\phi\in H_I(F_0, G_0)$ amounts to demanding that $A\in H_I(F_0, G_0)$, while $B$ is unconstrained.
  Thus we can identify $M_I$ via $h\mapsto h^{-1}$ with the invertible elements in
  \[ H_I(F_0, G_0) \times \Hom(R_0,Q_0), \]
  which form a nonempty Zariski-open subset, and hence a smooth irreducible subvariety of $\GL(Q_0)$.
  It follows that $\flag_I^0(F_0,\phi) = M_I / B(G_0)$ is likewise a smooth irreducible subvariety, and
  \begin{align*}
  &\quad \flag_I^0(F_0,\phi)
  = \dim M_I - \dim B(G_0)
  = \dim I + (n-r)(n-2r) - \dim B(G_0)\\
  &= \dim\flag(Q_0) + \dim I - (n-r)r,
  \end{align*}
  where we have used \cref{eq:dim schubert cell} and that $\flag(Q_0) \cong \GL(Q_0)/B(G_0)$.
\end{proof}

\counterwithin{equation}{subsection}
\section{Intersections and Horn inequalities}\label{sec:horn necessary}

In this section, we study intersections of Schubert varieties.
Recall from \cref{def:intersecting} that given an $s$-tuple $\CE$ of flags on $W$, $\dim W = n$, and $\CI\in\Subsets(r,n,s)$, we had defined
\[ \Omega^0_\CI(\CE) = \bigcap_{k=1}^s \Omega_{I_k}^0(E_k) \quad\text{and}\quad \Omega_\CI(\CE) = \bigcap_{k=1}^s \Omega_{I_k}(E_k). \]
We are particularly interested in the intersecting $\CI$, denoted $\CI \in \Intersecting(r,n,s)$, for which $\Omega_\CI(\CE)\neq\emptyset$ for every $\CE$.

\subsection{Coordinates}\label{subsec:coordinates}

Without loss of generality, we may assume that $W=\C^n$, and we shall do so for the remainder of this article.
As before, we denote by $(e(1),\dots,e(n))$ the ordered standard basis of $\C^n$ and by $E_0$ the corresponding \emph{standard flag}.
Let $V_0 = E_0(r)$ be the standard $r$-dimensional subspace, with ordered basis $(e(1),\dots,e(r))$, and $Q_0$ the subspace with ordered basis $(\bar e(1),\dots,\bar e(n-r))$, where $\bar e(b) := e(r+b)$.
Thus $W=V_0\op Q_0$.
We denote the corresponding standard flags on $V_0$ and $Q_0$ by $F_0$ and $G_0$, respectively.
Note that $F_0=E_0^{V_0}$ and, if we identify $Q_0\cong W/V_0$, then $G_0=(E_0)_{W/V_0}$.
We further abbreviate the Grassmannian by $\grass(r,n):=\grass(r,\C^n)$, the parabolic by $P(r,n):=P(V_0,\C^n)$ and the Borel by $B(n):=B(E_0)$.
We write $\flag(n) := \flag(W)$ and $\flag^0_I(r,n) := \flag^0_I(V_0, W)$ for the set of flags with respect to which $V_0$ has position $I$; $\flag_I(r,n) := \flag_I(V_0,W)$ is its closure.
We recall from \cref{def:H_I for flags} that
\begin{align*}
  H_I(F_0, G_0) = \{ \phi\in\Hom(V_0,Q_0) : \phi(e(a))\subseteq\linspan\{\bar e(1),\dots,\bar e(I(a)-a)\} \},
\end{align*}
and \cref{lem:G_I precise} reads
\begin{equation}
\label{eq:G_I precise concrete}
  G_I(r,n) = P(r,n) \bigl\{ \begin{pmatrix}\id_{V_0} & 0 \\ \phi & \id_{Q_0}\end{pmatrix} : \phi \in H_I(F_0, G_0) \bigr\} w_I,
\end{equation}
where we have introduced $G_I(r,n) := G_I(V_0,E_0)$.

\subsection{Intersections and dominance}
We start by reformulating the intersecting property in terms of the dominance of certain morphisms of algebraic varieties.
This allows us to give a simple proof of \cref{lem:edim nonnegative}, which states that the expected dimension of an intersecting tuple is necessarily nonnegative. 

We caution that while $\Omega^0_\CI(\CE) \subseteq \Omega_\CI(\CE)$, the latter is \emph{not} necessarily the closure of the former:

\begin{example*}
  Let $W = \C^2$, $I_1 = \{1\}$, $I_2 = \{2\}$, and $E_1 = E_2$ the same flag on $W$.
  Since the Schubert cells $\Omega^0_{I_k}(\CE)$ partition the projective space $\P(W) = \grass(1, W)$, $\Omega^0_\CI(\CE) = \emptyset$ is empty, but $\Omega_\CI(\CE) = \{ E_1(1) \}$ is a point.
\end{example*}

It is also possible that $\Omega^0_\CI(\CE)$ or $\Omega_\CI(\CE)$ are nonempty for some $\CE$ but empty for generic $s$-tuples $\CE$:

\begin{example*}
  Let $W = \C^2$, $I_1 = I_2 = \{1\}$. 
  Then $\Omega^0_\CI(\CE) = \Omega_\CI(\CE) = E_1(1) \cap E_2(1)$, so the intersection is nonempty if and only if $E_1 = E_2$.
\end{example*}

We will later show the existence of a `good set' of sufficiently generic $\CE$ such that $\CI$ is intersecting if and only if $\Omega^0_\CI(\CE)\neq\emptyset$ for any single `good' $\CE$ (\cref{lem:good set}).
Here is a more interesting example:

\begin{example}\label{ex:interesting}
  Let $W = \C^6$, $s = 3$, and $\CI = (I_1, I_2, I_3)$ where all $I_k = \{2,4,6\}$.
  The triple $\CI$ is intersecting.
  Let
  \[ f(t) := e_1 + t e_2 + \frac {t^2} {2!} e_3 + \frac {t^3} {3!} e_4 + \frac {t^4} {4!} e_5 + \frac {t^5} {5!} e_6 \]
  and consider the one-parameter family of flags $E(t)$ with adapted basis $(f(t), \frac d{dt}f(t),\dots,\frac {d^5}{dt^5} f(t))$.
  We consider the 3-tuple $\CE = (E_1, E_2, E_3)$, where $E_1 := E(0)$ is the standard flag, $E_2 := E(1)$, and $E_3 := E(-1)$.
  Then the intersection $\Omega^0_\CI(\CE)$ consists of precisely two points:
  \begin{align*}
    V_1 &= \linspan \{ e_2 + \sqrt 5 e_1,   e_4 - 24 \sqrt 5 e_1 - 3 \sqrt 5 e_3,   e_6 - 24 \sqrt 5 e_3 + \sqrt 5 e_5 \}, \\
    V_2 &= \linspan \{ e_2 - \sqrt 5 e_1,   e_4 + 24 \sqrt 5 e_1 + 3 \sqrt 5 e_3,   e_6 + 24 \sqrt 5 e_3 - \sqrt 5 e_5 \},
  \end{align*}
  and coincides with $\Omega_{\mathcal I}(\mathcal E)$.
\end{example}

To study generic intersections of Schubert cells, it is useful to introduce the following maps:
Let $\CI\in\Subsets(r,n,s)$.
We define
\[
  \omega^0_\CI\colon \begin{cases}
    \GL(n) \times \flag^0_{I_1}(r,n)\times\dots\times\flag^0_{I_s}(r,n) \to \flag(n)^s \\
    (\gamma, E_1, \dots, E_s) \mapsto (\gamma E_1, \dots, \gamma E_s)
  \end{cases}
\]
and its extension
\begin{equation}
\label{eq:delta_I}
  \omega_\CI\colon \begin{cases}
    \GL(n) \times \flag_{I_1}(r,n)\times\dots\times\flag_{I_s}(r,n) \to \flag(n)^s \\
    (\gamma, E_1, \dots, E_s) \mapsto (\gamma E_1, \dots, \gamma E_s).
  \end{cases}
\end{equation}


The following lemma shows that the images of $\omega^0_\CI$ and $\omega_\CI$, respectively, characterize the $s$-tuples $\CE$ of flags for which the intersections $\Omega^0_\CI(\CE)$ and $\Omega_\CI(\CE)$ are nonempty:

\begin{lemma}
\label{lem:images}
  Let $\CI\in\Subsets(r,n,s)$.
  Then,
  \begin{align*}
    \im \omega^0_\CI &= \{ \CE \in \flag(n)^s : \Omega^0_\CI(\CE)\neq\emptyset \}, \\
    \im \omega_\CI &= \{ \CE \in \flag(n)^s : \Omega_\CI(\CE)\neq\emptyset \}.
  \end{align*}
  In particular, $\CI \in \Intersecting(r,n,s)$ if and only if $\omega_\CI$ is surjective.
\end{lemma}
\begin{proof}
  If $\CE \in \im \omega^0_\CI$ then there exists $\gamma \in \GL(n)$ such that $E_k \in \gamma \flag^0_{I_k}(r,n)$ for $k\in[s]$.
  But
  \[
    E_k \in \gamma \flag^0_{I_k}(r,n)
    \Leftrightarrow \gamma^{-1} E_k \in \flag^0_{I_k}(r,n)
    \Leftrightarrow V_0 \in \Omega^0_{I_k}(\gamma^{-1} E_k)
    \Leftrightarrow \gamma V_0 \in \Omega^0_{I_k}(E_k),
  \]
  and therefore $\gamma V_0 \in \Omega^0_\CI(\CE)$.
  Conversely, if $V \in \Omega^0_\CI(\CE)$, then we write $V = \gamma V_0$ and obtain that $E_k \in \gamma \flag^0_{I_k}(r,n)$ for all $k$, and hence that $\CE\in\im\omega^0_\CI$.
  The result for $\im\omega_\CI$ is proved in the same way.
\end{proof}

We now use some basic algebraic geometry (see, e.g., \cite{perrin2007algebraic}).
Recall that a morphism $f\colon\CX\to\CY$ of irreducible algebraic varieties is called \emph{dominant} if its image is Zariski dense.
In this case, the image contains a nonempty Zariski-open subset $\CY_0$ such that the dimension of any irreducible component of the fibers $f^{-1}(y)$ for $y \in \CY_0$ is equal to $\dim \CX - \dim \CY$.
Furthermore, if $\CX_0 \subseteq \CX$ is a nonempty Zariski-open subset then $f$ is dominant if and only if its restriction $f$ to $\CX_0$ is dominant.

We also recall for future reference the following results:
If $\CX$ and $\CY$ are smooth (irreducible algebraic) varieties and $f\colon\CX\to\CY$ is dominant then the set of regular values (i.e., the points $y$ such that $df_x$ is surjective for all preimages $x\in f^{-1}(y)$) contains a Zariski-open set.
Also, if $df_x$ is surjective for every $x$ then the image by $f$ of any Zariski-open set in $\CX$ is a Zariski-open set in $\CY$.
In particular this is the case when $f\colon\mathcal V\to \mathcal B$ is a vector bundle.

In the present context, the maps $\omega^0_\CI$ and $\omega_\CI$ are morphisms of irreducible algebraic varieties and so the preceding discussion applies.
Furthermore, the domain of $\omega_\CI$ is the closure of the domain of $\omega^0_\CI$ in $\GL(n) \times \flag(n)^s$.
Therefore, $\omega_\CI$ is dominant if and only if $\omega^0_\CI$ is dominant.

\begin{lemma}
\label{lem:dominant}
  Let $\CI\in\Subsets(r,n,s)$.
  Then $\CI \in \Intersecting(r,n,s)$ if and only if $\omega_\CI$ or $\omega_\CI^0$ is dominant.
\end{lemma}
\begin{proof}
  On the one hand, \cref{lem:images} shows that $\CI$ is intersecting if and only if $\omega_\CI$ is surjective.
  On the other hand, we have just observed that $\omega_\CI^0$ is dominant if and only if $\omega_\CI$ is dominant.
  Thus it remains to show that $\omega_\CI$ is automatically surjective if it is dominant.
  For this, we observe that the space $\flag_{I_1}(r,n)\times\dots\times\flag_{I_s}(r,n)$ is left invariant by the diagonal action of the parabolic $P(r,n)$, as can be seen from~\eqref{eq:flag equivariance}.
  Thus $\omega_\CI$ factors over a map
  \begin{equation}
  \label{eq:delta I factorized}
    \bar\omega_\CI\colon \begin{cases}
      \GL(n) \times_{P(r,n)} \flag_{I_1}(r,n)\times\dots\times\flag_{I_s}(r,n) \to \flag(n)^s \\
      [\gamma, E_1, \dots, E_s] \mapsto (\gamma E_1, \dots, \gamma E_s).
    \end{cases}
  \end{equation}
  Clearly, $\omega_\CI$ and $\bar\omega_\CI$ have the same image.
  If $\bar\omega_\CI$ is dominant, then its image contains a nonempty Zariski-open set and therefore is dense in the Euclidean topology.
  But the domain of $\bar\omega_\CI$ is compact in the Euclidean topology and hence the image is also closed in the Euclidean topology.
  It follows that $\bar\omega_\CI$ is automatically surjective if $\bar\omega_\CI$ is dominant.
\end{proof}

A first, obvious condition for $\CI$ to be intersecting is therefore that the dimension of the domain of $\omega_\CI$ is no smaller than the dimension of the target space.
If we apply this argument to the factored map~\eqref{eq:delta I factorized}, which has the same image, we obtain that the expected dimension introduced in \cref{def:edim} is nonnegative:

\begin{lemma}
\label{lem:edim nonnegative}
  If $\CI\in\Intersecting(r,n,s)$ then
\printedimnonnegative{equation}{\label{eq:edim nonnegative}}
\end{lemma}
\begin{proof}
  Let $\CX := \GL(n) \times_{P(r,n)} \flag_{I_1}(r,n)\times\dots\times\flag_{I_s}(r,n)$ and $\CY := \flag(n)^s$.
  If $\CI$ is intersecting then the map $\bar\omega_\CI\colon\CX\to\CY$ in~\eqref{eq:delta I factorized} is dominant, hence $\dim\CX\geq\dim\CY$.
  But
  \begin{equation}
  \label{eq:edim as dim diff}
  \begin{aligned}
    &\quad \dim\CX-\dim\CY
    = (\dim \GL(n)/P(r,n)) + \sum_{k=1}^s (\dim\flag_{I_k}(r,n) - \dim \flag(n)) \\
    &= \dim \grass(r,n) - \sum_{k=1}^s (\dim \grass(r,n) - \dim I_k) = \edim\CI
  \end{aligned}
  \end{equation}
  where the first equality is obvious and the second is \cref{eq:codim flag vs grass}.
\end{proof}

At this point, we have established all facts that we used in \cref{sec:easy} to prove \cref{cor:belkale inductive}.
That is, the proof of \cref{cor:belkale inductive} is now complete.

We conclude this section by recording the following rules for the expected dimension,
\begin{align}
\label{eq:quotient edim}
  \edim\CI/\CJ &= \edim\CI + \edim\CJ - \edim\CI\CJ, \\
\label{eq:exp edim induction}
  \edim\CI^\CJ\CK - \edim \CK &= \edim\CI(\CJ\CK) - \edim \CJ\CK, \\
\label{eq:exp edim}
  \edim\CI^\CJ &= \edim\CI\CJ - \edim \CJ,
\end{align}
which hold for all
$\CI\in\Subsets(r,n,s)$,
$\CJ\in\Subsets(d,r,s)$, and
$\CK\in\Subsets(m,d,s)$.
\Cref{eq:quotient edim,eq:exp edim induction,eq:exp edim} are direct consequences of \cref{lem:quotient dim,lem:exp dim chain rule}.
\Cref{eq:exp edim induction} in particular will play a crucial role in \cref{subsec:kernel recurrence}, as we will use it to show that if $\CI$ satisfies the Horn inequalities and $\CJ$ is intersecting then so does $\CI^\CJ$.
This will be key to establishing Belkale's theorem on the sufficiency of the Horn inequalities by induction (\cref{thm:belkale}).

\subsection{Slopes and Horn inequalities}
We are now interested in proving a strengthened version of \cref{cor:belkale inductive} (see \cref{cor:belkale inductive stronger} below).
As a first step, we introduce the promised `good set' of $s$-tuples of flags which are sufficiently generic to detect when an $s$-tuple $\CI$ is intersecting:
Define in analogy to~\eqref{eq:delta I factorized} the map
\begin{equation*}
  \bar\omega^0_\CI \colon \begin{cases}
    \GL(n) \times_{P(r,n)}\flag^0_{I_1}(r,n)\times\cdots\times\flag^0_{I_s}(r,n) \to \flag(n)^s \\
    [\gamma, E_1, \dots, E_s] \mapsto (\gamma E_1, \dots, \gamma E_s)
  \end{cases}.
\end{equation*}

\begin{lemma}\label{lem:good set}
  There exists a nonempty Zariski-open subset $\Good(n, s) \subseteq \flag(n)^s$ that satisfies the following three properties for all $r\in[n]$:
  \begin{enumerate}[label=(\alph*)]
  \item\label{item:good set a} $\Good(n,s)$ consists of regular values (in the image) of~$\bar\omega^0_\CI$ for every $\CI\in\Intersecting(r,n,s)$.
  \item\label{item:good set b} For every $\CI\in\Subsets(r,n,s)$, the following are equivalent:
    \begin{enumerate}
    \item \label{item:good set i} $\CI \in \Intersecting(r, n, s)$.
    \item \label{item:good set ii} For all $\CE \in \Good(n, s)$, $\Omega^0_\CI(\CE) \neq \emptyset$.
    \item \label{item:good set iii} There exists $\CE \in \Good(n, s)$ such that $\Omega^0_{\CI}(\CE) \neq \emptyset$.
    \end{enumerate}
  \item\label{item:good set c} If $\CI\in\Intersecting(r,n,s)$, then for every $\CE\in\Good(n,s)$ the variety $\Omega^0_\CI(\CE)$ has the same number of irreducible components, each connected component is of dimension $\edim \CI$, and $\Omega^0_{\CI}(\CE)$ is dense in $\Omega_{\CI}(\CE)$.
  \end{enumerate}
\end{lemma}
\begin{proof}
  Let us construct $\Good(n,s)$ satisfying the properties above.
  Let $\CI\in\Subsets(r,n,s)$, where $r\in[n]$.
  If $\CI\not\in\Intersecting(r,n,s)$ then by \cref{lem:dominant} the map $\omega^0_\CI$ is not dominant, and we define $U_\CI$ as the complement of the Zariski-closure of $\im \omega^0_\CI$.
  Thus $U_\CI$ is a nonempty Zariski-open subset of $\flag(n)^s$.
  Otherwise, if $\CI\in\Intersecting(r,n,s)$ then $\omega^0_\CI$ is dominant by \cref{lem:dominant}. The map $\bar\omega^0_\CI$ has the same image as $\omega^0_\CI$ and is therefore also a dominant map between smooth irreducible varieties.
  Thus its image contains a nonempty Zariski-open subset $U_\CI$ of $\flag(n)^s$ consisting of regular values, such that the fibers $(\bar\omega^0_\CI)^{-1}(\CE)$ for $\CE \in U_{\CI}$ all have the same number of irreducible components, each of dimension equal to $\edim \CI$, by the calculation in~\eqref{eq:edim as dim diff}.
  We now define the good set as
  \[ \Good(n,s) := \bigcap_\CI U_\CI, \]
  where the intersection is over all $s$-tuples $\CI$, intersecting or not.
  As a finite intersection of nonempty Zariski-open subsets, $\Good(n,s)$ is again nonempty and Zariski-open.
  By construction, it satisfies property~\ref{item:good set a}.

  We now show that $\Good(n,s)$ satisfies~\ref{item:good set b}.
  To see that~\ref{item:good set i} implies~\ref{item:good set ii}, note that for any $\CI \in \Intersecting(r,n,s)$ and $\CE\in\Good(n,s)$,
  $\CE \in U_\CI \subseteq \im\bar\omega_\CI^0 = \im\omega_\CI^0$.
  Thus \cref{lem:images} shows that $\Omega^0_\CI(\CE)\neq\emptyset$.
  Clearly,~\ref{item:good set ii} implies~\ref{item:good set iii} since $\Good(n,s)$ is nonempty.
  Lastly, suppose that~\ref{item:good set iii} holds. By \cref{lem:images}, $\Omega^0_\CI(\CE)\neq\emptyset$ implies that $\CE\in\im\omega_\CI^0$.
  But $\CE\in\Good(n,s) \subseteq (\im\omega_\CI)^c \subseteq (\im\omega_\CI^0)^c$ unless $\CI$ is intersecting; this establishes~\ref{item:good set i}.
  

  Lastly, we verify ~\ref{item:good set c}.
  Observe that, for any $\CE\in\flag(n)^s$, the fiber $(\bar\omega^0_\CI)^{-1}(\CE)$ is equal to the set of $[\gamma, \gamma^{-1}E_1, \dots, \gamma^{-1}E_s]$ such that $\gamma^{-1} E_k \in \flag^0_{I_k}(r,n)$ for all $k\in[s]$.
  It can therefore by $\gamma \mapsto \gamma V_0$ be identified with $\Omega^0_\CI(\CE)$.
  Now assume that $\CI\in\Intersecting(r,n,s)$. As we vary $\CE\in\Good(n,s)$, $\CE\in U_\CI$ and so $(\bar\omega^0_\CI)^{-1}(\CE) \cong \Omega^0_\CI(\CE)$ has the same number of irreducible components, each of dimension $\edim\CI$.
  We still need to show that $\Omega_\CI^0(\CE)$ is dense in $\Omega_\CI(\CE)$.
  This will follow if we can show that $\Omega_\CI^0(\CE)$ meets any irreducible component $\CZ$ of $\Omega_\CI(\CE)$.
  Let us assume that this is not the case, so that $\CZ \subseteq \Omega_\CI(\CE) \setminus \Omega_\CI^0(\CE)$.
  But
  \[ \Omega_\CI(\CE) \setminus \Omega_\CI^0(\CE)
   = \bigcup_{k=1}^s \biggl( (\Omega_{I_k}(E_k) \setminus \Omega^0_{I_k}(E_k)) \cap \bigcap_{l\neq k} \Omega_{I_l}(E_l) \biggr)
   = \bigcup_{\substack{I'_1\leq I_1,\ \dots,\ I'_s\leq I_s \\ \exists k \in [s]: I'_k \neq I_k}} \Omega^0_{\CI'}
  \]
  by \cref{lem:schubert variety characterization}.
  That is, $\Omega_\CI(\CE) \setminus \Omega_\CI^0(\CE)$ is a union of varieties $\Omega^0_{\CI'}(\CE)$ with $\edim\CI'<\edim\CI$.
  If $\CI'$ is intersecting then any irreducible component of $\Omega^0_{\CI'}(\CE)$ has dimension equal to $\edim\CI'$.
  Otherwise, if $\CI'$ is not intersecting, then $\Omega^0_{\CI'}(\CE) = \emptyset$.
  It follows that any irreducible component of $\Omega_\CI(\CE) \setminus \Omega_\CI^0(\CE)$ has dimension strictly smaller than $\edim\CI$.
  But this is a contradiction, since the dimension of $\CZ$ is equal to at least $\edim\CI$.
\end{proof}

The following is a direct consequence of the equivalence between~\ref{item:good set i} and~\ref{item:good set iii} in \cref{lem:good set}:
\begin{equation}
\label{eq:intersecting via good}
  \Intersecting(r,n,s) = \{ \Pos(V,\CE) : V \subseteq \C^n, \dim V = r \}
\end{equation}
for every $\CE\in\Good(n,s)$.

We now study the numerical inequalities satisfied by intersecting $s$-tuples more carefully.
Recall that a weight $\theta$ for $\GL(r)$ is \emph{antidominant} if $\theta(1)\leq\dots\leq\theta(r)$.
For example, given a subset $I\subseteq[n]$ of cardinality $r$, the weight $\theta(a) := I(a)-a$ is antidominant.
It is convenient to introduce the following definition:

\begin{definition}\label{def:slope}
  Given an $s$-tuple $\vec\theta=(\theta_1,\dots,\theta_s)$ of antidominant weights for $\GL(r)$, we define the \emph{slope} of a tuple $\CJ\in\Subsets(d,r,s)$ as
  \[ \mu_{\vec\theta}(\CJ) := \frac1d \sum_{k=1}^s \sum_{a \in J_k} \theta_k(a) = \frac1d \sum_{k=1}^s (T_{J_k}, \theta_k). \]
  For any nonzero subspace $\{0\}\neq S\subseteq\C^r$ and $s$-tuple of flags $\CF$ on $\C^r$, we further define
  \[ \mu_{\vec\theta}(S,\CF) := \mu_{\vec\theta}(\Pos(S, \CF)). \]
  Here and in the following, we write $\Pos(S,\CF)$ for the $s$-tuple of positions $(\Pos(S,F_k))_{k\in[s]}$.
\end{definition}
Note that we can interpret $\mu_{\vec\theta}(\CJ)$ as a sum of averages of the nowhere decreasing functions $\theta_k$ for uniform choice of $a \in J_k$.

The following lemma asserts that there is a unique slope-minimizing subspace of maximal dimension:

\begin{lemma}[Harder--Narasimhan,~\cite{belkale2001local}]
\label{lem:slope}
  Let $\vec\theta$ be an $s$-tuple of antidominant weights for $\GL(r)$, and $\CF\in\flag(r)^s$.
  Let $m_* := \min_{\{0\}\neq S\subseteq\C^r} \mu_{\vec\theta}(S,\CF)$ and $d_* := \max \{ \dim S : \mu_{\vec\theta}(S,\CF) = m_* \}$.
  Then there exists a unique subspace $S_* \subseteq \C^r$ such that $\mu_{\vec\theta}(S_*,\CF) = m_*$ and $\dim S_* = d_*>0$.
\end{lemma}
\begin{proof}
  Existence is immediate, so it remains to show uniqueness.
  Thus suppose for sake of finding a contradiction that there are two such subspaces, $S_1\neq S_2$, such that $\mu_{\vec\theta}(S_j,\CF) = m_*$ and $\dim S_j = d_*$ for $j=1,2$.
  We note that $d_*>0$ and that the inclusions $S_1\cap S_2\subsetneq S_1$ and $S_2 \subsetneq S_1+S_2$ are strict.

  Let $\CJ = \Pos(S_1,\CF)$ and $\CK = \Pos(S_1\cap S_2, \CF^{S_1})$.
  Then $\Pos(S_1\cap S_2, \CF) = \CJ\CK$ by the chain rule (\cref{lem:chain rule}).
  Let us first assume that $S_1\cap S_2\neq\{0\}$, so that $\mu_{\vec\theta}(\CJ\CK)$ is well-defined.
  Then,
  \[ \mu_{\vec\theta}(\CJ\CK) = \mu_{\vec\theta}(S_1\cap S_2,\CF) \geq m_* = \mu_{\vec\theta}(S_1,\CF) = \mu_{\vec\theta}(\CJ), \]
  where the equalities hold by definition, and the inequality holds as $m_*$ is the minimal slope.
  On the other hand, note that $J_k = J_k K_k \cup J_k K_k^c$ for each $k\in[s]$, hence we can write
  \[ \mu_{\vec\theta}(\CJ) = \frac d {d_*} \mu_{\vec\theta}(\CJ\CK) + \frac {d_*-d} {d_*} \mu_{\vec\theta}(\CJ\CK^c), \]
  where $d := \dim S_1\cap S_2 < \dim S_1 = d_*$.
  It follows that
  \begin{equation}
  \label{eq:slope first ieq}
    m_* = \mu_{\vec\theta}(\CJ) \geq \mu_{\vec\theta}(\CJ\CK^c).
  \end{equation}
  If $S_1\cap S_2=\{0\}$ then $\CJ=\CJ\CK^c$ and so~\eqref{eq:slope first ieq} holds with equality.

  Likewise, let $\CL = \Pos(S_1 + S_2, \CF)$ and $\CM = \Pos(S_2, \CF^{S_1+S_2})$.
  Since $S_2\subsetneq S_1+S_2$, but $S_2$ was assumed to be a maximal-dimensional subspace with minimal slope, it follows that the slope of $S_1+S_2$ is strictly larger than $m_*$:
  \[ \mu_{\vec\theta}(\CL\CM) = \mu_{\vec\theta}(S_2, \CF) = m_* < \mu_{\vec\theta}(S_1 + S_2, \CF) = \mu_{\vec\theta}(\CL). \]
  Just as before, we decompose
  \[ \mu_{\vec\theta}(\CL) = \frac {d_*} {d'} \mu_{\vec\theta}(\CL\CM) + \frac {d' - d_*} {d'} \mu_{\vec\theta}(\CL\CM^c), \]
  where now $d' := \dim S_1+S_2 > \dim S_2 = d_* > 0$.
  Thus we obtain the strict inequality
  \begin{equation}
  \label{eq:slope second ieq}
    \mu_{\vec\theta}(\CL\CM^c) > \mu_{\vec\theta}(\CL\CM) = m_*.
  \end{equation}
  At last, we apply \cref{cor:slope preliminary}, which shows that $J_k K_k^c(b) \geq L_k M_k^c(b)$ for all $b$ and $k$, and hence
  \[ \mu_{\vec\theta}(\CJ\CK^c) \geq \mu_{\vec\theta}(\CL\CM^c). \]
  Together with~\eqref{eq:slope first ieq} and~\eqref{eq:slope second ieq}, we obtain the desired contradiction:
  \[ m_* \geq \mu_{\vec\theta}(\CJ\CK^c) \geq \mu_{\vec\theta}(\CL\CM^c) > m_*. \qedhere \]
\end{proof}

We will now use \cref{lem:good set,lem:slope} to show that the conditions in \cref{cor:belkale inductive} with $\edim\CJ = 0$ imply those for general intersecting $\CJ$.

\begin{definition}
\label{def:lambda_I}
Let $I \subseteq [n]$ be a subset of cardinality $r$. We define $\lambda_I\in\Lambda_+(r)$ by
\[ \lambda_I(a) := a - I(a) \qquad (a\in[r]). \]
Any highest weight $\lambda$ with $\lambda(1)\leq 0$, $\lambda(r)\geq r-n$ can be written in this form.
Moreover, if $I^c$ denotes the complement of $I$ in $[n]$ then the dominant weight $\lambda_{I^c}\in\Lambda_+(n-r)$ can be written as
\begin{equation}
\label{eq:complement}
  \lambda_{I^c}(b) = b - I^c(b)
= -\#\{ a \in [r] : I(a) < I^c(b) \}
= -\#\{ a \in [r] : I(a) - a < b \}.
\end{equation}
\end{definition}

\begin{remark*}
This equation has a pleasant interpretation in terms of Young diagrams.
Consider the Young diagram $Y_I$ corresponding to $\lambda_I^*$, which has $I(r+1-a) - (r+1-a)$ boxes in its $a$-th row.
By definition, its transpose $Y_I^t$ is the Young diagram such that the number of boxes in the $b$-th row is equal to the number of boxes in the $b$-th column of $T_I$.
Thus~\eqref{eq:complement} asserts that $Y_I^t = r\one_{n-r} + \lambda_{I^c}$, i.e., the two Young diagrams $Y_I^t$ and $Y_{I^c}$ (the latter with rows in reverse order) make up a rectangle of size $r\times(n-r)$.
\end{remark*}

\begin{lemma}
\label{lem:I to lambdaprime}
  Let $\CI\in\Subsets(r,n,s)$.
  Set $\lambda_k = \lambda_{I_k} + (n-r)\one_r$ for $k\in[s-1]$ and $\lambda_s = \lambda_{I_s}$.
  Then we have that $\edim\CI=-\sum_{k=1}^s \lvert\lambda_k\rvert$.
  More generally, for every $\CJ\in\Subsets(d,r,s)$,
  \[ \edim\CI\CJ-\edim\CJ
  = -\sum_{k=1}^s (T_{J_k}, \lambda_k)
  = d \mu_{-\vec\lambda}(\CJ), \]
  where we recall that $(T_J,\xi)=\sum_{j\in J} \xi(j)$ for any $J\subseteq[r]$ and $\xi\in i\mathfrak t(r)$.
\end{lemma}
\begin{proof}
  It suffices to prove the second statement, which follows from
  \begin{align*}
  &\quad \edim\CI\CJ-\edim\CJ
  = d(n-r)(1-s) + \sum_{k=1}^s \sum_{a\in J_k} \bigl( I_k(a) - a \bigr) \\
  &= d(n-r)(1-s) - \sum_{k=1}^s (T_{J_k}, \lambda_{I_k})
  = - \sum_{k=1}^s (T_{J_k}, \lambda_k)
  = d \mu_{-\vec\lambda}(\CJ). \qedhere
  \end{align*}
\end{proof}

It follows that minimizing $\mu_{-\vec\lambda}(\CJ)$ and $\frac1d\bigl(\edim\CI\CJ - \edim\CJ\bigr)$ as a function of $\CJ$ are equivalent.
We then have the following result:

\begin{proposition}
\label{prp:belkale strong weak}
  Let $\CI\in\Subsets(r,n,s)$ such that $\edim\CI\geq0$ and, for any $0<d<r$ and $\CJ\in\Intersecting(d,r,s)$ with $\edim\CJ=0$ we have that $\edim\CI\CJ\geq0$.
  Then we have for any $0<d<r$ and $\CJ\in\Intersecting(d,r,s)$ that
  \[ \edim\CI\CJ\geq\edim\CJ. \]
\end{proposition}
\begin{proof}
  Suppose for sake of finding a contradiction that there exists $\CJ\in\Intersecting(d,r,s)$ with $0<d<r$ and $\edim\CI\CJ < \edim\CJ$, so that $\mu_{-\vec\lambda}(\CJ) < 0$ according to \cref{lem:I to lambdaprime}.
  Fix some $\CF\in\Good(r,s)$.
  Then $\Omega^0_\CJ(\CF)\neq\emptyset$ by \cref{lem:good set},~\ref{item:good set ii}.
  Thus there exists a subspace $\{0\}\neq S\subseteq\C^r$ such that $\mu_{-\vec\lambda}(S,\CF) = \mu_{-\vec\lambda}(\CJ) < 0$.

  Now let $S_*$ be the unique subspace of minimal slope $m_*<0$ and maximal dimension $d_*>0$ from \cref{lem:slope} and denote by $\CJ_* := \Pos(S_*,\CF)$ its $s$-tuple of positions.
  The uniqueness statement implies that $\Omega^0_{\CJ_*}(\CF) = \{S_*\}$, since slope and dimension are fully determined by the position.
  Moreover, $\CJ_*$ is intersecting by~\eqref{eq:intersecting via good}, and therefore $\edim \CJ_* = \dim \Omega^0_{\CJ_*}(\CF) = 0$ by \cref{lem:good set}.
  Thus we have found an $s$-tuple $\CJ_* \in \Intersecting(d_*,r,s)$ with $d_*>0$, $\edim\CJ_* = 0$, and
  \[ \edim\CI\CJ_* = \edim\CI\CJ_* - \edim\CJ_* = d_* m_* < 0, \]
  where we have used \cref{lem:I to lambdaprime} once again in the last equality.
  Since $\edim\CI\geq 0$, this also implies that $d_*<r$.
  This is the desired contradiction.
\end{proof}

\Cref{prp:belkale strong weak} will be useful to prove Belkale's \cref{thm:belkale} in \cref{sec:horn sufficient} below, since it allows us to work with a larger set of inequalities.

\begin{remark*}
  The proof of \cref{prp:belkale strong weak} shows that we may in fact restrict to $\CJ$ such that $\Omega^0_\CJ(\CF)$ is a point for all $s$-tuples of good flags $\CF\in\Good(r,s)$ -- or also to those for which $\Omega_\CJ(\CF)$ is a point, which is equivalent by the last statement in \cref{lem:good set}.
  See the remark after \cref{cor:horn and saturation} for the implications of this on the description of the Kirwan cone.
\end{remark*}

We also record the following corollary which follows together with and improves over \cref{cor:belkale inductive}.

\begin{corollary}
\label{cor:belkale inductive stronger}
  If $\CI \in \Intersecting(r,n,s)$ then for any $0<d<r$ and any $s$-tuple $\CJ \in \Intersecting(d,r,s)$ we have that $\edim\CI\CJ \geq \edim\CJ$.
\end{corollary}

We remark that for $d=r$ there is only one $s$-tuple, $\CJ = ([r],\dots,[r])$, and it is intersecting and satisfies $\edim\CJ=0$.
In this case, $\edim\CI\CJ-\edim\CJ=\edim\CI$, and so we may safely allow for $d=r$ in \cref{cor:belkale inductive,prp:belkale strong weak,cor:belkale inductive stronger}.

We conclude this section with some simple examples of the Horn inequalities of \cref{cor:belkale inductive stronger}.
We refer to \cref{app:examples horn} for lists of all Horn triples $\CI=(I_1,I_2,I_3)$ up to $n=4$.

\begin{example}[$r=1$]
\label{ex:base case}
  The only condition for $\CI\in\Intersecting(1,n,s)$ is the dimension condition, $\edim\CI\geq0$.
  Indeed, the Grassmannian $\grass(1,n)$ is the projective space $\P(\C^n)$, whose Schubert varieties are given by $\Omega_{\{i\}}(E) = \{ [v] \in \P(\C^n) : v \in E(i) \}$.
  Thus $\CI = (\{i_1\},\dots,\{i_s\})$ is intersecting if and only if for any $s$-tuple of flags $\CE$, $E_1(i_1)\cap\dots\cap E_s(i_s)\neq\{0\}$.
  By linear algebra, it is certainly sufficient that $\sum_{k=1}^s (n - i_k) \leq n-1$, which is equivalent to $\edim\CI\geq0$.
  This also establishes \cref{thm:belkale} in the case $r=1$.
\end{example}

\begin{example*}[$s=2$, $r=2$]
  Let $\CI = (I_1,I_2)$. Then the condition $\edim\CI\geq0$ is $I_1(1) + I_1(2) + I_2(1) + I_2(2) \geq 2n+2$.
  However, there are two additional conditions coming from the $\CJ\in\Intersecting(1,2,2)$ with $\edim\CJ=0$.
  By the preceding example, there are two such pairs, $(\{1\},\{2\})$ and $(\{2\},\{1\})$.
  The corresponding conditions are $I_1(1) + I_2(2) \geq n+1$ and $I_1(2) + I_2(1) \geq n+1$.

  For example, if $n=4$ then $\CI = (\{1,4\},\{2,4\})$ satisfies all Horn inequalities.
  On the other hand, $\CI = (\{1,4\},\{2,3\})$ fails one the Horn inequalities.
  Indeed, if we consider $\CJ=(\{1\},\{2\})$ then $\CI\CJ=(\{1\},\{3\})$ is such that $\edim\CI\CJ = -1 < 0$.
\end{example*}

\counterwithin{equation}{subsection}
\section{Sufficiency of Horn inequalities}\label{sec:horn sufficient}

In this section we prove that the Horn inequalities are also sufficient to characterize intersections of Schubert varieties.

\subsection{Tangent maps}
In \cref{lem:dominant}, we established that an $s$-tuple $\CI$ is intersecting if and only if the corresponding morphism $\omega_\CI$ defined in~\eqref{eq:delta_I} is dominant.
Now it is a general fact that a morphism $f\colon\CX\to\CY$ between \emph{smooth} and irreducible varieties is dominant if and only if there exists a point $p\in\CX$ where the differential $T_pf$ is surjective.
This will presently allow us to reduce the intersecting of Schubert varieties to an infinitesimal question about tangent maps.
Later, in \cref{sec:invariants}, we will also use the determinant of the tangent map to construct explicit nonzero tensor product invariants and establish the saturation property.

\begin{lemma}\label{lem:differential concrete}
  Let $\CI\in\Subsets(r,n,s)$.
  Then $\CI\in\Intersecting(r,n,s)$ if and only if there exist $\vec g=(g_1,\dots,g_s)\in\GL(V_0)^s$ and $\vec h=(h_1,\dots,h_s)\in\GL(Q_0)^s$ such that the linear map
  \begin{equation}
  \label{eq:differential concrete}
    \Delta_{\CI,\vec g,\vec h}\colon\begin{cases}
    \Hom(V_0, Q_0)\times H_{I_1}(F_0, G_0)\times\dots\times H_{I_s}(F_0, G_0)\to \Hom(V_0,Q_0)^s \\
    (\zeta,\phi_1,\dots,\phi_s)\mapsto(\zeta+h_1\phi_1g_1^{-1},\dots,\zeta+h_s\phi_s g_s^{-1})
    \end{cases}
  \end{equation}
  is surjective.
\end{lemma}
\begin{proof}
  Using the isomorphisms $\flag^0_{I_k}(r,n) = G_{I_k}(r,n) E_0 \cong G_{I_k}(r,n)/B(n)$ (\cref{def:G_I,subsec:coordinates}) and $\flag(n) \cong \GL(n)/B(n)$, we find that $\omega^0_\CI$ is dominant if and only if
  \begin{equation}
  \label{eq:dominant cover}
    \GL(n) \times G_{I_1}(r,n) \times \dots \times G_{I_s}(r,n) \to \GL(n)^s, (\gamma, \gamma_1, \dots, \gamma_s) \mapsto (\gamma \gamma_1, \dots, \gamma \gamma_s)
  \end{equation}
  is dominant.
  This is again a morphism between smooth and irreducible varieties and thus dominance is equivalent to surjectivity of the differential at some point $(\gamma,\gamma_1,\dots,\gamma_s)$.
  The map~\eqref{eq:dominant cover} is $\GL(n)$-equivariant on the left and $B(n)^s$-equivariant on the right.
  By the former, we may assume that $\gamma=1$, and by the latter that $\gamma_k = p_k w_{I_k}$ for some $p_k = \bigl(\begin{smallmatrix}g_k & b_k \\ 0 & h_k\end{smallmatrix}\bigr)$, since $G_{I_k}(r,n) = P(r, n) w_{I_k} B(n)$ according \cref{lem:G_I coarse}.

  We now compute the differential. 
  Thus we consider an arbitrary curve $1 + \varepsilon X$ tangent to $\gamma=1$, where $X \in \gl(n)$, and curves $(1 + \varepsilon Y_k) p_k w_{I_k}$ through the $\gamma_k = p_k w_k$, where $Y_k \in \gl(n)$. If we write
  $Y_k = \bigl(\begin{smallmatrix}A_k & B_k \\ C_k & D_k \end{smallmatrix}\bigr)$
  with $A_k \in \gl(r)$ etc., then we see from~\eqref{eq:G_I precise concrete} that $(1 + \varepsilon Y_k) p_k w_{I_k}$ is tangent to $G_{I_k}(r,n)$ precisely if $h^{-1}_k C_k g_k \in H_{I_k}(F_0, G_0)$, that is, if $C_k \in h_k H_{I_k}(F_0,G_0) g_k^{-1}$.
  Lastly, the calculation $(1 + \varepsilon X)(1 + \varepsilon Y_k) \gamma_k = \gamma_k + \varepsilon (X + Y_k) \gamma_k + O(\varepsilon^2)$ shows that the differential of~\eqref{eq:dominant cover} at $(g,\gamma_1,\dots,\gamma_s)$ can be identified with $(X, Y_1, \dots, Y_s) \mapsto (X + Y_1, \dots, X + Y_s)$.

  We may check for surjectivity block by block.
  Since there are no constraints on the $A_k$, $B_k$, and $D_k$, it is clear that the differential is surjective on the three blocks corresponding to $\p(r,n)$.
  Thus we only need to check surjectivity on the last block of the linear map, corresponding to $\Hom(V_0, Q_0)$.
  This block can plainly be identified with~\eqref{eq:differential concrete}, since the $C_k$ are constrained to be elements of $h_k H_{I_k}(F_0,G_0) g_k^{-1}$.
  Thus we obtain that $\omega_\CI$ is dominant if and only if~\eqref{eq:differential concrete} is surjective.
\end{proof}

\begin{remark}\label{rem:delta^0_I factorized differential}
  The map $\Delta_{\CI,\vec g,\vec h}$ can be identified with the differential of $\bar\delta^0_\CI$ at the point $[1, \CE]$, where $E_k\in\flag^0_{I_k}(r,n)$ is such that $(E_k)^{V_0} = g_k \cdot F_0$ and $(E_k)_{Q_0} = h_k \cdot G_0$ for $k\in[s]$.
  This follows from the proof of \cref{lem:differential concrete} and justifies calling $\Delta_{\CI,\vec g,\vec h}$ a \emph{tangent map}.
\end{remark}

By the rank-nullity theorem and using \cref{lem:dim schubert cell}, the kernel of the linear map $\Delta_{\CI,\vec g,\vec h}$ defined in~\eqref{eq:differential concrete} is of dimension at least
\begin{equation}
\label{eq:rank nullity}
\begin{aligned}
  &\quad \dim \bigl( \Hom(V_0,Q_0)\times H_{I_1}(F_0, G_0)\times\dots\times H_{I_s}(F_0, G_0) \bigr) - \dim \Hom(V_0,Q_0)^s \\
  &= r(n-r)(1-s) + \sum_{k=1}^s \dim I_k = \edim\CI,
\end{aligned}
\end{equation}
and $\Delta_{\CI,\vec g,\vec h}$ is surjective if and only if equality holds.
On the other hand, it is immediate that
\begin{equation}
\label{eq:ker Delta}
  \ker \Delta_{\CI,\vec g,\vec h} = \bigcap_{k=1}^s h_k H_{I_k}(F_0, G_0) g_k = \bigcap_{k=1}^s H_{I_k}(F_k, G_k),
\end{equation}
where $F_k = g_k F_0$ and $G_k = h_k G_0$.
As we vary $g_k$ and $h_k$, the $F_k$ and $G_k$ are arbitrary flags on $V_0$ and $Q_0$, respectively.
Thus we obtain the following characterization:

\begin{definition}
  Let $\CI\in\Subsets(r,n,s)$.
  We define the \emph{true dimension} of $\CI$ as
  \begin{equation}
  \label{eq:tdim}
    \tdim\CI := \min_{\CF,\CG} \dim H_\CI(\CF,\CG) = \min_{\vec g, \vec h} \dim\ker\Delta_{\CI,\vec g,\vec h},
  \end{equation}
  where the first side minimization is over all $s$-tuples of flags $\CF$ on $V_0$ and $\CG$ on $Q_0$,
  the second one over $\vec g\in\GL(r)^s$, $\vec h\in\GL(n-r)^s$,
  and where
  \[ H_\CI(\CF,\CG) := \bigcap_{k=1}^s H_{I_k}(F_k, G_k) \subseteq \Hom(V_0,Q_0). \]
\end{definition}

\begin{corollary}
\label{cor:tdim edim}
  Let $\CI\in\Subsets(r,n,s)$.
  Then we have $\tdim\CI\geq\edim\CI$, with equality if and only if $\CI\in\Intersecting(r,n,s)$.
\end{corollary}

We note that for the purpose of computing true dimensions we may always assume that $F_1$ and $G_1$ are the standard flags on $V_0$ and $Q_0$, respectively (by equivariance).

\begin{example*}[s=2,r=2,n=4]
  We verify the example at the end of \cref{sec:horn necessary} by using \cref{cor:tdim edim}.
  We first consider $\CI=(\{1,4\},\{2,4\})$.
  Then $\edim\CI=1$. 
  To bound $\tdim\CI$, we let $\CF=(F_1,F_2)$ and $\CG=(G_1,G_2)$, where $F_1$ is the standard flag on $V_0$, $F_2$ the flag with adapted basis $(e(1)+e(2), e(2))$, and $G_1=G_2$ the standard flags on $Q_0$.
  Then
  \[
      H_\CI(\CF,\CG)
    = \{ \begin{pmatrix} 0 & * \\ 0 & * \end{pmatrix} \} \cap H_{I_2}(F_2, G_2)
    = \{ \begin{pmatrix} 0 & * \\ 0 & 0 \end{pmatrix} \}
  \]
  is one-dimensional, which shows that $\tdim\CI\leq1$.
  Since always $\tdim\CI\geq\edim\CI$, it follows that, in fact, $\tdim\CI=\edim\CI$ and so $\CI$ is intersecting.

  We now consider $\CI=(\{1,4\}, \{2,3\})$.
  Then $\edim\CI=0$. 
  Let $\CF$ and $\CG$ be pairs of flags on $V_0$ and $Q_0$, respectively.
  Without loss of generality, we shall assume that $F_1$ and $G_1$ are the standard flags. Then
  \[
      H_\CI(\CF,\CG)
    = \{ \begin{pmatrix} 0 & * \\ 0 & * \end{pmatrix} \} \cap H_{I_2}(F_2,G_2)
    = \C \begin{pmatrix} 0 & x \\ 0 & y \end{pmatrix},
  \]
  where $\C \bigl(\begin{smallmatrix}x \\ y\end{smallmatrix}\bigr) := G_2(1)$.
  Indeed, $H_{I_2}(F_2,G_2)$ consists of those linear maps that map any vector in $V_0$ into $G_2(1)$.
  In particular, $H_\CI(\CF,\CG)$ is one-dimensional for any choice of $F_2$ and $G_2$.
  Thus $\tdim\CI=1>0=\edim\CI$, and we conclude that $\CI$ is not intersecting.
\end{example*}

\begin{example*}[s=2,r=3,n=6]
  Let $\CI=(\{3,4,6\},\{2,4,5\})$.
  Then $\edim\CI=3$. 
  We now establish that $\CI$ is intersecting by verifying that $\tdim\CI=3$.
  Again we choose $F_1$ and $G_1$ to be the standard flags on $V_0$ and $Q_0$, respectively, while $F_2$ and $G_2$ are defined as follows in terms of adapted bases:
  \begin{align*}
    F_2\colon \quad &e(1) + z_{21} e(2) + z_{31} e(3),\quad e(2) + z_{32} e(3),\quad e(3),\\
    G_2\colon \quad &\bar e(1) + u_{21} \bar e(2) + u_{31} \bar e(3),\quad\bar e(2) + u_{32} \bar e(3),\quad\bar e(3).
  \end{align*}
  Then a basis for $H_\CI(\CF,\CG)$ is on the open set where $u_{31}u_{32}\neq0$ given by
  \begin{align*}
    \phi_1 &= \begin{pmatrix} -z_{21} u_{32} & u_{32} & 0 \\
      -z_{21} ( u_{32} u_{21} - u_{31} ) & u_{32} u_{21} - u_{31} & 0 \\
      0 & 0 & 0
    \end{pmatrix}, \\
    \phi_2 &= \begin{pmatrix} z_{31} u_{32} & 0 & 0 \\
      z_{31} ( u_{32} u_{21} - u_{31} ) & 0 & u_{31} \\
      0 & 0 & u_{32} u_{31}
    \end{pmatrix},
    \qquad
    \phi_3 = \begin{pmatrix}
      0 & 0 & 1 \\
      0 & 0 & u_{21} \\
      0 & 0 & u_{31}
    \end{pmatrix}
  \end{align*}
  as can be checked by manual inspection.
\end{example*}

\subsection{Kernel dimension and position}
\label{subsec:kdim kpos}

Let us consider a tuple $\CI\in\Subsets(r,n,s)$, where we always assume that $r\in[n]$.
To prove sufficiency of the Horn inequalities, we aim to use \cref{cor:tdim edim}, which states that $\tdim\CI\geq\edim\CI$, with equality if and only if $\CI$ is intersecting.

If $\tdim\CI=0$ then, necessarily, $\tdim\CI=\edim\CI=0$, since $\edim\CI$ is nonnegative by assumption (part of the Horn inequalities). Hence in this case $\CI$ is intersecting.

Thus the interesting case is when $\tdim\CI>0$.
To study the spaces $H_\CI(\CF,\CG)$ in a unified fashion, we consider the space
\[ \Pa(\CI) := \{ (\CF,\CG,\phi) \in \flag(V_0)^s \times \flag(Q_0)^s \times \Hom(V_0,Q_0) : \phi\in H_\CI(\CF,\CG) \}. \]
We caution that $\Pa(\CI)$ is not in general irreducible, as the following example shows:

\begin{example*}
  Let $s=2$, $n=3$, $r=1$, and consider $I_1 = I_2 = \{2\}$.
  There is only a single flag on $V_0 \cong \C$, while any flag $G$ on $Q_0 \cong \C^2$ is determined by a line $L = G(1)\in\P(\C^2)$.
  Thus we can identify $\Pa(\CI) \cong \{ (L_1, L_2, \phi) \in \P(\C^2)^2 \times \Hom(\C^1,\C^2) : \phi(e(1)) \in L_1 \cap L_2 \}$.
  If we consider the map $(L_1,L_2,\phi)\mapsto(L_1,L_2)$, then the fiber for any $L_1=L_2$ is a one-dimensional line, while for any $L_1\neq L_2$ the fiber is just $\phi=0$.
  In particular, we note that $\Pa(\CI)$ is not irreducible.
\end{example*}

We now restrict to those $(\CF,\CG)$ such that the intersection $H_\CI(\CF,\CG)$ is of dimension $\tdim\CI$.
Thus we introduce
\begin{align*}
  \Pt(\CI) &:= \{ (\CF,\CG,\phi) \in \Pa(\CI) : \dim H_\CI(\CF,\CG) = \tdim\CI \}, \\
  \Bt(\CI) &:= \{ (\CF,\CG) \in \flag(V_0)^s \times \flag(Q_0)^s : \dim H_\CI(\CF,\CG) = \tdim \CI \}.
\end{align*}
The subscripts in $\Pt(\CI)$ and $\Bt(\CI)$ stands for the true dimension, $\tdim\CI$.
We use similar subscripts throughout this section when we fix various other dimensions and positions.

Since $\tdim\CI$ is the minimal possible dimension, this is the generic case.
Moreover, this restriction makes $\Pt(\CI)$ irreducible, as it is a vector bundle over $\Bt(\CI)$.
We record this in the following lemma:

\begin{lemma}
\label{lem:P_t zopen}
  The space $\Pa(\CI)$ is a closed subvariety of $\flag(V_0)^s\times\flag(Q_0)^s\times\Hom(V_0,Q_0)$, and $\Pt(\CI)$ is a nonempty Zariski-open subset of $\Pa(\CI)$.
  Moreover, $\Bt(\CI)$ is a nonempty Zariski-open subset of $\flag(V_0)^s\times\flag(Q_0)^s$, and the map $(\CF,\CG,\phi)\mapsto(\CF,\CG)$ turns $\Pt(\CI)$ into a vector bundle over $\Bt(\CI)$.
  In particular, $\Pt(\CI)$ is an irreducible and smooth variety.
\end{lemma}

In particular:
\begin{equation}
\label{eq:dim Pt}
  \dim \Pt(\CI) = s \bigl( \dim \flag(V_0) + \dim \flag(Q_0) \bigr) + \tdim\CI
\end{equation}

Belkale's insight is now to consider the behavior of generic kernels of maps $\phi\in H_\CI(\CF,\CG)$, where $(\CF,\CG)\in \Bt(\CI)$.
We start with the following definition:

\begin{definition}
  Let $\CI\in\Subsets(r,n,s)$.
  We define the \emph{kernel dimension} of $\CI$ as
  \[ \kdim\CI := \min \{ \dim\ker\phi : \phi\in H_\CI(\CF, \CG) \text{ where } (\CF,\CG) \in \Bt(\CI) \} \]
\end{definition}

There are two special cases that we can treat right away.
If $\kdim\CI=r$ then any morphism in $H_\CI(\CF,\CG)$ for $(\CF,\CG)\in \Bt(\CI)$ is zero, and hence $\tdim\CI=0$.
This is the case that we had discussed initially and we record this observation for future reference:

\begin{lemma}
\label{lem:kdim r intersecting}
  Let $\CI\in\Subsets(r,n,s)$ such that $\edim\CI\geq0$.
  If $\kdim\CI=r$ then $\tdim\CI=\edim\CI=0$, and hence $\CI\in\Intersecting(r,n,s)$.
\end{lemma}

Likewise, the case where $\kdim\CI=0$ can easily be treated directly.
The idea is to compute the dimension of $\Pt(\CI)$ in a second way and compare the result with~\eqref{eq:dim Pt}.

\begin{lemma}
\label{lem:dim P_t zero}
  Let $\CI\in\Subsets(r,n,s)$.
  If $\kdim\CI=0$ then
  \[ \dim \Pt(\CI) = s \bigl( \dim \flag(V_0) + \dim \flag(Q_0) \bigr) + \edim\CI. \]
\end{lemma}
\begin{proof}
  We first note that $\kdim\CI=0$ implies that there exists an injective map $\phi\in H_\CI(\CF,\CG)$ for some $(\CF,\CG) \in \Bt(\CI)$.
  In particular, $I_k(a)-a\geq a$ for all $k\in[s]$ and $a\in[r]$ (a fact that we use further below in the proof).
  Now define
  \[ \Pk := \{ (\CF,\CG,\phi) \in \Pa(\CI) : \dim\ker\phi=0 \} \]
  Then $\Pk$ is a nonempty Zariski-open subset of $\Pa(\CI)$ that intersects $\Pt(\CI)$.
  By \cref{lem:P_t zopen}, the latter is irreducible.
  Thus it suffices to show that $\Pk$ is likewise irreducible and to compute its dimension.

  For this, we consider the map
  \[ \pi\colon \Pk\to \Mk:=\flag(V_0)^s\times\Hom^\times(V_0,Q_0), \quad (\CF,\CG,\phi)\mapsto(\CF,\phi) \]
  where we write $\Hom^\times(V_0,Q_0)$ for the Zariski-open subset of injective linear maps in $\Hom(V_0,Q_0)$.
  The fibers of $\pi$ are given by
  \[ \pi^{-1}(\CF,\phi) \cong \prod_{k=1}^s \flag_{I_k}^0(F_k, \phi) \]
  which according to \cref{lem:image flags} are smooth irreducible varieties of dimension
  $s \dim \flag(Q_0) - s r(n-r) + \sum_{k=1}^s \dim I_k$.
  It is not hard to see that $\pi$ gives $\Pk$ the structure of a fiber bundle over $\Mk$.
  Therefore, $\Pk$ is irreducible.
  Moreover, the space $\Mk$ has dimension $s \dim \flag(V_0) + r(n-r)$.
  By adding the dimension of the fibers, we obtain that the dimension of $\Pk$, and hence of $\Pt(\CI)$, is indeed the one claimed in the lemma.
\end{proof}

\begin{corollary}
\label{cor:kdim zero intersecting}
  Let $\CI\in\Subsets(r,n,s)$.
  If $\kdim\CI=0$ then $\tdim\CI=\edim\CI$, and hence $\CI\in\Intersecting(r,n,s)$.
\end{corollary}
\begin{proof}
  This follows directly by comparing \cref{eq:dim Pt,lem:dim P_t zero}.
\end{proof}

We now consider the general case, where $0<d:=\kdim\CI<r$.
We first note that the kernel dimension is attained generically.
Thus we define
\begin{align*}
  \Pkt(\CI) &:= \{ (\CF,\CG,\phi) \in \Pt(\CI) : \dim\ker\phi = \kdim\CI \}, \\
  \Bkt(\CI) &:= \{ (\CF,\CG) : \exists\phi \text{~s.th.~} (\CF,\CG,\phi) \in \Pkt(\CI) \} \subseteq \Bt(\CI),
\end{align*}
where the subscripts denote that we fix both the true dimension as well as the kernel dimension.
We have the following lemma:

\begin{lemma}
\label{lem:P_kt zopen}
  The set $\Pkt(\CI)$ is a nonempty Zariski-open subset of $\Pt(\CI)$, hence also irreducible.
  Moreover, $\Bkt(\CI)$ is a nonempty Zariski-open subset of $\flag(V_0)^s\times\flag(Q_0)^s$.
\end{lemma}
\begin{proof}
  The first claim holds since $\Pkt(\CI)$ can be defined by the nonvanishing of certain minors.
  The second claim now follows as $\Bkt(\CI)$ is the image of the Zariski-open subset $\Pkt(\CI)$ of the vector bundle $\Pt(\CI)\to\Bt(\CI)$.
\end{proof}


Belkale's insight is to consider the positions of generic kernels for an induction:

\begin{definition}\label{def:kerpos}
  Let $\CI\in\Subsets(r,n,s)$. 
  Then we define the \emph{kernel position} of $\CI$ as the tuple $\CJ\in\Subsets(d,r,s)$ defined by
  \[ J_k(b) := \min \{ \Pos(\ker\phi,F_k)(b) \;:\; (\CF,\CG,\phi) \in \Pkt(\CI) \} \]
  for $b \in [d]$ and $k \in [s]$.
  We write $\kPos(\CI) = \CJ$.
\end{definition}

The goal in the remainder of this subsection is to prove the following equality:
\[ \tdim\CI = \edim\CJ + \edim\CI/\CJ, \]
where $\CJ=\kPos(\CI)$.
This will again be accomplished by computing the dimension of $\Pt(\CI)$ in a second way and comparing the result with~\eqref{eq:dim Pt}.
Specifically, we consider the spaces
\begin{align*}
  \Pkpt(\CI) &:= \{ (\CF,\CG,\phi) \in \Pkt(\CI) : \Pos(\ker\phi, \CF) = \kPos(\CI) \}, \\
  \Bkpt(\CI) &:= \{ (\CF,\CG) : \exists\phi \text{~s.th.~} (\CF,\CG,\phi) \in \Pkpt(\CI) \} \subseteq \Bkt(\CI).
\end{align*}
Then $\Pkpt(\CI)$ is Zariski-open in $\Pkt(\CI)$, since it can again be defined by demanding that certain minors are nonzero. 
We obtain the following lemma, the second claim in which is proved as before:

\begin{lemma}
\label{lem:P_kpt zopen}
  Let $\CI\in\Subsets(r,n,s)$ such that $0<\kdim\CI<r$.
  Then $\Pkpt(\CI)$ is a nonempty Zariski-open subset of $\Pkt(\CI)$, hence also irreducible.
  Moreover, $\Bkpt(\CI)$ is a nonempty Zariski-open subset of $\flag(V_0)^s\times\flag(Q_0)^s$.
\end{lemma}

\begin{corollary}
\label{cor:kerpos intersecting}
  Let $\CI\in\Subsets(r,n,s)$ such that $0<\kdim\CI<r$.
  Then $\kPos(\CI)\in\Intersecting(d,r,s)$.
\end{corollary}
\begin{proof}
  According to \cref{lem:P_kpt zopen}, $\Bkpt(\CI)$ is a nonempty Zariski-open subset of $\flag(V_0)^s\times\flag(Q_0)^s$,
  hence Zariski-dense. 
  It follows that its image under the projection $(\CF,\CG)\mapsto\CF$ is likewise Zariski-dense.
  For any such $\CF$, there exists a $\CG$ and $\phi$ such that $(\CF,\CG,\phi) \in \Pkpt(\CI)$, and hence $\ker\phi \in \Omega^0_{\kPos(\CI)}(\CF)$; in particular, $\Omega^0_{\kPos(\CI)}(\CF)$ is nonempty.
  Thus \cref{lem:images,lem:dominant} show that $\kPos(\CI)$ is intersecting.
\end{proof}

We now compute the dimension of $\Pkpt(\CI)$.
As in the proof of \cref{lem:dim P_t zero}, it will be useful to consider an auxiliary space where we do \emph{not} enforce the true dimension:
\begin{equation*}
  \Pkp(\CI) := \{ (\CF,\CG,\phi) \in \Pa(\CI) : \Pos(\ker\phi,\CF) = \kPos(\CI) \}
\end{equation*}
Note that constraint on the position of the kernel implies that its dimension is $\kdim\CI$.

\begin{lemma}
\label{lem:dim P_kpt nonzero}
  Let $\CI\in\Subsets(r,n,s)$ such that $0<\kdim\CI<r$.
  Then $\Pkp(\CI)$ is nonempty, smooth, irreducible, and satisfies
  \[ \dim \Pkp(\CI) = s \bigl( \dim \flag(V_0) + \dim \flag(Q_0) \bigr) + \edim\CJ + \edim\CI/\CJ, \]
  where $\CJ := \kPos(\CI)$.
\end{lemma}
\begin{proof}
  Clearly, $\Pkp(\CI)$ is nonempty since it contains $\Pkpt(\CI)$.
  We now introduce
  \[ \Mkp := \{ (\CF,\phi) \in \flag(V_0)^s \times \Hom(V_0,Q_0) : \Pos(\ker\phi,\CF) = \kPos(\CI) \} \]
  and consider the map
  \[ \pi\colon \Pkp(\CI)\to\Mkp, \quad (\CF,\CG,\phi)\mapsto(\CF,\phi). \]
  Its fibers are given by
  \[ \pi^{-1}(\CF,\phi) \cong \prod_{k=1}^s \{ G_k \in \flag(Q_0) : \phi \in H_{I_k}(F_k, G_k) \} \]
  To understand the right-hand side, define $S:=\ker\phi$ and let $\bar\phi\colon V_0/S\to Q_0$ the corresponding injective map.
  By \cref{lem:exp vs composition}, $\phi \in H_{I_k}(F_k, G_k)$ if and only if $\bar\phi \in H_{I_k/J_k}((F_k)_{V_0/S}, G_k)$, that is, $G_k \in \flag_{I_k/J_k}^0((F_k)_{V_0/S}, \bar\phi)$ as introduced in \cref{def:flags for injective morphism}.
  Thus we find that the fibers of $\pi$ can be identified as
  \[ \pi^{-1}(\CF,\phi) \cong \prod_{k=1}^s \flag_{I_k/J_k}^0((F_k)_{V_0/S}, \bar\phi). \]
  By \cref{lem:image flags}, the $k$-th factor on the right-hand side is a smooth irreducible variety of dimension $\dim \flag(Q_0) - (r-d)(n-r) + \dim I_k/J_k$, where $d := \dim\ker\phi = \kdim\CI$.
  It is not hard to see that $\pi$ is a fiber bundle, and we will show momentarily that $\Mkp$ is irreducible.
  Hence
  \begin{equation}
  \label{eq:dim P_kpt half}
    \dim \Pkp(\CI) = \dim \Mkp + s \dim \flag(Q_0) - s(r-d)(n-r) + \sum_{k=1}^s \dim I_k/J_k.
  \end{equation}
  It remains to show that $\Mkp$ is smooth and irreducible and to compute its dimension.
  For this, we consider the map
  \[ \tau \colon \Mkp\to \grass(d,V_0), (\CF,\phi) \mapsto \ker\phi. \]
  Since $\phi$ can be specified in terms of the kernel $S:=\ker\phi$ and the injection $\bar\phi\colon V_0/S\to Q_0$, it is clear that the fibers of $\tau$ are given by
  \[ \tau^{-1}(S) = \Hom^\times(V_0/S, Q_0) \times \prod_{k=1}^s \flag^0_{J_k}(S, V_0). \]
  Since $\tau$ is likewise a fiber bundle, we obtain that $\Mkp$ is smooth and irreducible and, using~\eqref{eq:dim flag^0_I}, that
  \begin{align*}
    &\quad \dim \Mkp = \dim \grass(d,V_0) + (r-d)(n-r) + \sum_{k=1}^s \dim\flag^0_{J_k}(S, V_0) \\
    &= d(r-d) + (r-d)(n-r) + s \bigl( \dim \flag(S) + \dim \flag(V_0/S) \bigr) + \sum_{k=1}^s \dim J_k \\
    &= d(r-d)(1-s) + (r-d)(n-r) + s \dim \flag(V_0) + \sum_{k=1}^s \dim J_k.
  \end{align*}
  By plugging this result into~\eqref{eq:dim P_kpt half} and simplifying, we obtain the desired result.
\end{proof}

\begin{corollary}
  Let $\CI\in\Subsets(r,n,s)$ such that $0<\kdim\CI<r$, and $\CJ=\kPos(\CI)$.
  Then,
  \begin{equation}
  \label{eq:tdim via kpos}
    \tdim\CI = \edim\CJ + \edim\CI/\CJ
  \end{equation}
\end{corollary}
\begin{proof}
  Recall that $\Pkp(\CI) \subseteq \Pa(\CI) \supseteq \Pt(\CI)$.
  Moreover,
  \[ \Pkpt(\CI) = \Pkp(\CI)\cap\Pt(\CI) \subseteq \Pa(\CI). \]
  All three varieties $\Pkpt(\CI)$, $\Pkp(\CI)$, $\Pt(\CI)$ are irreducible (\cref{lem:P_t zopen,lem:P_kpt zopen,lem:dim P_kpt nonzero}).
  Moreover, $\Pkpt(\CI)$ is nonempty and Zariski-open in $\Pa(\CI)$, hence in both $\Pkp(\CI)$ and $\Pt(\CI)$.
  It follows that
  \[ \dim\Pkp(\CI) = \dim\Pkpt(\CI) = \dim\Pt(\CI). \]
  We now obtain~\eqref{eq:tdim via kpos} via \cref{lem:dim P_kpt nonzero,eq:dim Pt}.
\end{proof}

\begin{remark*}
  Purbhoo~\cite{purbhoo2006two} asserts that if $\CJ$ denotes the kernel position of $\CI$ then $\CI/\CJ$ is intersecting.
  However, we believe that the proof given therein is incomplete, as it is not clear that the map $(\CF,\CG,\phi)\mapsto(\CF_{V/S},\CG)$ is dominant
  (cf.\ the remark at~\cite{purbhooweb}).
  The following argument suggests that the situation is somewhat more delicate.
\end{remark*}

\subsection{The kernel recurrence}
\label{subsec:kernel recurrence}

To conclude the proof in the case that $0<\kdim\CI<r$, we need to understand the right-hand side of~\eqref{eq:tdim via kpos} some more.
We start with the calculation
\begin{equation}
\label{eq:tdim minus edim first}
  \tdim\CI - \edim\CI
  = \edim\CJ - (\edim\CI\CJ - \edim\CJ)
  = \edim\CJ - \edim\CI^\CJ,
\end{equation}
where the first equality is due to \cref{eq:tdim via kpos,eq:quotient edim} and the second is \cref{eq:exp edim}.

The last missing ingredient is to understand the expected dimension of the kernel position, $\edim\CJ$.

\begin{lemma}
\label{lem:sherman upper bound}
  Let $\CI\in\Subsets(r,n,s)$ such that $0<\kdim\CI<r$, and let $\CJ:=\kPos(\CI)$.
  Then we have $\edim\CJ \leq \tdim\CI^\CJ$.
\end{lemma}
\begin{proof}
  For any $(\CF,\CG,\phi)\in \Pkp(\CI)$, the space $H_\CJ(\CF^{\ker\phi}, \CF_{V_0/\ker\phi})$ injects into $H_{\CI^\CJ}(\CF^{\ker\phi},\CG)$ by composition with the injective map $\bar\phi\colon V_0/\ker\phi\to Q_0$ induced by $\phi$ (\cref{lem:exp vs composition}). Thus,
  \[ \edim\CJ\leq\tdim\CJ\leq \dim H_\CJ(\CF^{\ker\phi},\CF_{V_0/\ker\phi})\leq\dim H_{\CI^\CJ}(\CF^{\ker\phi},\CG), \]
  where the first inequality is always true (\cref{cor:tdim edim}), the second holds by definition of the true dimension and the third follows from the injection.
  It thus suffices to prove that there exists $(\CF,\CG,\phi)\in\Pkp(\CI)$ such that $\dim H_{\CI^\CJ}(\CF^S,\CG)\leq\tdim\CI^\CJ$.

  For this, let $K(d,V_0)$ denote the fiber bundle over $\grass(d,V_0)$ with fiber over $S\in\grass(d,V_0)$ given by $\flag(S)^s\times\flag(Q_0)^s$.
  It is an irreducible algebraic variety and we denote its elements by $(S,\tilde\CF,\CG)$.
  We consider the morphism
  \[ \pi\colon \Pkp(\CI)\to K(d,V_0), \quad (\CF,\CG,\phi)\mapsto(\ker\phi,\CF^{\ker\phi},\CG). \]
  For any $(\CF,\CG,\phi)\in\Pkp(\CI)$, $\dim\ker\phi=d$ and $\Pos(\ker\phi,\CF)=\CJ$, hence $\pi$ is indeed a morphism.

  We first prove that $\pi$ is dominant.
  Note that, as a consequence of \cref{lem:P_kpt zopen}, the map $\Pkp(\CI)\to\flag(Q_0)^s, (\CF,\CG,\phi)\mapsto\CG$ contains a nonempty Zariski-open subset $U \subseteq \flag(Q_0)^s$.
  We now show that the image of $\pi$ contains all elements $(S,\tilde\CF,\CG)$ with $S\in\grass(d,V_0)$, $\tilde\CF\in\flag(S)^s$ and $\CG\in U$.
  For this, let $(\CF_0,\CG,\phi_0)\in\Pkp(\CI)$ be the preimage of some arbitrary $\CG\in U$.
  Let $S_0 :=\ker\phi_0$ and choose some $g \in \GL(V_0)$ such that $g \cdot S_0 = S$.
  Using the corresponding diagonal action, $\CF:=g\cdot\CF_0$ and $\phi:=g\cdot\phi_0$, we obtain that $(\CF,\CG,\phi)\in\Pkp(\CI)$ and $\ker\phi=S$.
  Given $\tilde\CF\in\flag(S)^s$, we now choose $\vec{h}\in\GL(V_0)^s$ such that
  $h_k S \subseteq S$,
  $h_k \cdot F_k^S = \tilde F_k$,
  and $h_k$ acts trivially on $V_0/S$ for all $k\in[s]$.
  Then $\Pos(S, \vec{h} \cdot \CF) = \Pos(S, \CF) = \CJ$, which shows that $(\vec{h} \cdot \CF)^S = \vec{h} \cdot \CF^S = \tilde\CF$.
  Moreover, $(\vec{h} \cdot \CF)_{V_0/S} = \CF_{V_0/S}$.
  Thus $\phi\in H_\CI(\CF, \CG)$ implies that $\phi\in H_\CI(\vec{h} \cdot \CF, \CG)$ by \cref{lem:exp vs composition}.
  Together, we find that the triple $(\vec{h}\cdot\CF,\CG,\phi)$ is in $\Pkp(\CI)$ and mapped by $\pi$ to $(S,\tilde\CF,\CG)$.
  We thus obtain that $\pi$ is dominant.

  To conclude the proof, we note that the subset $W\subseteq K(d,V_0)$ consisting of those $(S,\tilde\CF,\CG)$ with $\dim H_{\CI^\CJ}(\tilde\CF,\CG)=\tdim\CI^\CJ$ is a nonempty Zariski-open subset, and hence Zariski-dense since $K(d,V_0)$ is irreducible.
  For each fixed choice of $S$, this is the claim in \cref{lem:P_t zopen} for $B_t(\CI)$, with $\CI^\CJ$ instead of $\CI$.
  The `parametrized version' is proved in the same way.
  Since $\pi$ is dominant, the preimage $\pi^{-1}(W)$ is a nonempty Zariski-open subset of $\Pkp(\CI)$.
  In particular, any $(\CF,\CG,\pi)\in\pi^{-1}(W)\subseteq\Pkp(\CI)$ satisfies $\dim H_{\CI^\CJ}(\CF^S,\CG)\leq\tdim\CI^\CJ$.
\end{proof}

We thus obtain the following fundamental recurrence relation, due to Sherman~\cite{sherman2015geometric}, as a consequence of \cref{eq:tdim minus edim first,lem:sherman upper bound}:
\begin{equation}
\label{eq:sherman recurrence}
  \tdim\CI-\edim\CI \leq \tdim\CI^\CJ-\edim\CI^\CJ
\end{equation}

Now we have assembled all ingredients to prove Belkale's theorem:

\printbelkaletheorem{theorem}{\label{thm:belkale}}{, restated}
\begin{proof} 
  We proceed by induction on $r$.
  The base case, $r=1$, is \cref{ex:base case}.
  Thus we have $\Intersecting(1,n,s) = \Horn(1,n,s)$ for all $n\geq1$.

  Now let $r>1$.
  By the induction hypothesis, $\Horn(d,n',s) = \Intersecting(d,n',s)$ for all $0<d<r$ and $d\leq n'$.
  In particular, $\Horn(r,n,s)$ from \cref{def:horn} can be written in the following form:
  \begin{align*}
    &\quad \Horn(r,n,s) \\
    &= \{ \CI : \edim\CI\geq0, \;\;\forall \CJ\in\Intersecting(d,r,s),0<d<r, \;\edim\CJ=0: \edim\CI\CJ\geq0\} \\
    &= \{ \CI : \edim\CI\geq0, \;\;\forall \CJ\in\Intersecting(d,r,s),0<d<r, \;\edim\CI\CJ\geq\edim\CJ\}
  \end{align*}
  where the second equality is due to \cref{prp:belkale strong weak}.
  Hence it is a direct consequence of \cref{cor:belkale inductive} that $\Intersecting(r,n,s) \subseteq \Horn(r,n,s)$.
  We now prove the converse.

  Thus let $\CI\in\Horn(r,n,s)$.
  Let $d:=\kdim\CI$. If $d=0$ or $d=r$ then we know from \cref{lem:kdim r intersecting,cor:kdim zero intersecting}, respectively, that $\CI$ is intersecting.
  We now discuss the case where $0<d<r$.
  By \cref{eq:sherman recurrence}, we have that
  \[ \tdim\CI - \edim\CI \leq \tdim\CI^\CJ - \edim\CI^\CJ, \]
  where $\CJ := \kPos(\CI)$ denotes the kernel position of $\CI$.
  If we can show that $\CI^\CJ$ is intersecting then the right-hand side is zero by \cref{cor:tdim edim}, hence so is the left-hand side, since $\tdim\CI-\edim\CI\geq 0$, and thus $\CI$ is intersecting, which is what we set out to prove.

  To see that $\CI^\CJ$ is intersecting, we note that $\Intersecting(d,n-r+d,s) = \Horn(d,n-r+d,s)$ by the induction hypothesis, hence it remains to verify that $\CI^\CJ$ satisfies the Horn inequalities.
  Let $\CK\in\Horn(m,d,s)=\Intersecting(m,d,s)$ for any $0<m\leq d$, where we have used the induction hypothesis one last time.
  Thus $\CJ\CK\in\Intersecting(m,r,s)$ by \cref{cor:kerpos intersecting,lem:chain rule intersecting}.
  It follows that
  \begin{align*}
    \edim\CI^\CJ\CK - \edim\CK
  = \edim\CI(\CJ\CK) - \edim\CJ\CK
  \geq 0
  \end{align*}
  where the first step is~\eqref{eq:exp edim induction} and the second step holds because by assumption $\CI\in\Horn(r,n,s)$ and $\CJ\CK\in\Intersecting(m,r,s)=\Horn(m,r,s)$, as explained above.
  We remark that these inequalities include $\edim \CI^\CJ\geq0$ (corresponding to $m=d$).
  Thus we have shown that $\CI^\CJ$ satisfies the Horn inequalities.
  This is what remained to be proved.
\end{proof}

\counterwithin{equation}{subsection}
\section{Invariants and Horn inequalities}\label{sec:invariants}

In this section, we show that the Horn inequalities not only characterize intersections, but also the existence of corresponding nonzero invariants and, thereby, the Kirwan cone for the eigenvalues of sums of Hermitian matrices.

\subsection{Borel-Weil construction}\label{subsec:borel-weil}
For any dominant weight $\lambda\in\Lambda_+(r)$ there exists an irreducible representation $L(\lambda)$ of $\GL(r)$ with highest weight $\lambda$, unique up to isomorphism.
Following Borel and Weil, it can be constructed as follows:

For any weight $\mu\in\Lambda(r)$, let us denote by $\chi_\mu\colon B(r)\to\C^*$ the character of $B(r)$ such that $\chi_\mu(t) = t^\mu = t(1)^{\mu(1)}\cdots t(r)^{\mu(r)}$ for all $t\in H(r)\subseteq B(r)$.
Here, we recall that $B(r)$ is the group of upper-triangular invertible matrices and $H(r) \subseteq B(r)$ the Cartan subgroup, which consists of invertible matrices $t\in\GL(r)$ that are diagonal in the standard basis, with diagonal entries $t(1),\dots,t(r)$.
Lastly, we write $\one_r=(1,\dots,1)\in\Lambda(r)$ for the highest weight of the determinant representation of $\GL(r)$, denoted $\det_r$.
It is clear that $L(\lambda + k \one_r) = L(\lambda) \ot \det_r^k$ for any $\lambda\in\Lambda_+(r)$ and $k\in\Z$.

\begin{definition}
\label{def:borel weil}
  Let $\lambda\in\Lambda_+(r)$.
  Then we define the \emph{Borel-Weil realization} of $L(\lambda)$ as
  \[ L_{BW}(\lambda) = \{ s\colon \GL(r) \to \C \text{ holomorphic} \;:\; s(gb) = s(g)\chi_{\lambda^*}(b) \quad \forall g\in\GL(r), b\in B(r) \} \]
  with the action of $\GL(r)$ given by $(g \cdot s)(h) := s(g^{-1}h)$.
  We recall that $\lambda^*=(-\lambda(r),\dots,-\lambda(1))$.
\end{definition}

The Borel-Weil theorem asserts that $L_{BW}(\lambda)$ is an irreducible $\GL(r)$-representation of highest weight $\lambda$.
Note that, by definition, a holomorphic function is in $L_{BW}(\lambda)$ if it is a highest weight vector of weight $\lambda^*$ with respect to the \emph{right} multiplication representation, $(g \star s)(h) := s(h g)$.

The space $L_{BW}(\lambda)$ can also be interpreted as the space of holomorphic sections of the $\GL(r)$-equivariant line bundle $\CL_{BW}(\lambda) := \GL(r) \times_{B(r)} \C_{-\lambda^*}$ over $\flag(r)\cong\GL(r)/B(r)$, where we write $\C_\mu$ for the one-dimensional representation of $B(r)$ given by the character $\chi_\mu$.

It is useful to observe that we have a $\GL(r)$-equivariant isomorphism
\begin{equation}\label{eq:borel weil dual iso}
  L(\lambda)^* \to L_{BW}(\lambda^*), \quad f \mapsto (s_f \colon \GL(r)\to\C, \; g \mapsto f(g \cdot v_\lambda))
\end{equation}
where $v_\lambda$ denotes a fixed highest weight vector in $L(\lambda)$.

The tensor product of several Borel-Weil representations can again be identified with a space of functions. E.g., if $\lambda\in\Lambda_+(r)$ and $\lambda'\in\Lambda_+(r')$ then
\begin{align*}
&L_{BW}(\lambda) \ot L_{BW}(\lambda') \cong \{ s \colon \GL(r)\times\GL(r')\to\C \text{ holomorphic}, \\
&\qquad s(gb,g'b') = s(g,g')\chi_{\lambda^*}(b)\chi_{{\lambda'}^*}(b') \quad \forall g\in\GL(r),g'\in\GL(r'), b\in B(r),b'\in B(r') \}.
\end{align*}
We will use this below to obtain a nonzero vector in a tensor product space by exhibiting a corresponding holomorphic function with the appropriate equivariance properties.

\subsection{Invariants from intersecting tuples}\label{subsec:invariants from intersecting tuples}

Let us consider the tangent map~\eqref{eq:differential concrete},
\[ \Delta_{\CI,\vec g,\vec h}\colon\begin{cases}
    \Hom(V_0,Q_0)\times H_{I_1}(F_0,G_0)\times\dots\times H_{I_s}(F_0,G_0)\to \Hom(V_0,Q_0)^s \\
    (\zeta,\phi_1,\dots,\phi_s)\mapsto(\zeta+h_1\phi_1g_1^{-1},\dots,\zeta+h_s\phi_s g_s^{-1})
    \end{cases} \]
If $\edim\CI=0$ then~\eqref{eq:rank nullity} implies that the dimension of the domain and target space are the same.
Thus we may consider the determinant of $\Delta_{\CI,\vec g,\vec\ d}$, as in the following definition:

\begin{definition}
\label{def:determinant}
  Let $\CI\in\Subsets(r,n,s)$ such that $\edim\CI=0$.
  Then we define the \emph{determinant function} as the holomorphic function
  \[ \delta_\CI\colon\begin{cases}
    \GL(r)^s \times \GL(n-r)^s \rightarrow \C, \\
    (\vec g, \vec h) \mapsto \det\Delta_{\CI,\vec g,\vec h}
  \end{cases} \]
  where the determinant is evaluated with respect to two arbitrary bases.
\end{definition}

Here, and throughout the following, we identify $V_0\cong\C^r$ and $Q_0 \cong\C^{n-r}$, so that $\GL(V_0)\cong\GL(r)$ and $\GL(Q_0)\cong\GL(n-r)$ and the discussion in \cref{subsec:borel-weil} is applicable.

If $\CI$ is intersecting then also $\tdim\CI=\edim\CI$ by \cref{cor:tdim edim}.
Hence by~\eqref{eq:tdim} there exist $\vec g,\vec h$ such that $\delta_\CI(\vec g,\vec h)\neq 0$.
That is, $\delta_\CI$ is a nonvanishing holomorphic function of $\GL(r)^s \times \GL(n-r)^s$.
Our goal is to show that $\delta_\CI$ can be interpreted as an invariant in a tensor product of irreducible $\GL(r)\times\GL(n-r)$-representations.

We now consider the representation of $\GL(r)\times\GL(n-r)$ on $\Hom(V_0,Q_0)$ given by $(a,d)\cdot\phi := d\phi a^{-1}$.
Since $\Hom(V_0, Q_0) = V_0^* \ot Q_0$, it is clear that for $g\in\GL(V_0)$, $g'\in\GL(Q_0)$,
\begin{equation}
\label{eq:base change H}
  \det\Bigl( \Hom(V_0, Q_0)\ni\phi \mapsto g'\phi g^{-1}\in\Hom(V_0, Q_0) \Bigr) = \det(g)^{-(n-r)} \det(g')^r.
\end{equation}
We now restrict to the subspaces $H_I(F_0, G_0)$:

\begin{lemma}
\label{lem:base change H_I}
  Let $I \subseteq [n]$ be a subset of cardinality $r$.
  Then $H_I(F_0, G_0) \subseteq \Hom(V_0,Q_0)$ is $B(r) \times B(n-r)$-stable.
  Furthermore, for $b\in B(r)$ and $b'\in B(n-r)$ we have
  \[ \det \Bigl( H_I(F_0, G_0)\ni\phi \mapsto b'\phi b^{-1} \in H_I(F_0, G_0) \Bigr) = \chi_{\lambda_I}(b) \chi_{\lambda_{I^c} + r\one_{n-r}}(b'), \]
  where we recall that $\lambda_I$ was defined in \cref{def:lambda_I}.
\end{lemma}
\begin{proof}
  For the first claim, we use \cref{lem:H_I borel invariance}:
  Since the flag $F_0$ is stabilized by $B(r)$ and the flag $G_0$ is stabilized by $B(n-r)$, it is clear that $H_I(F_0, G_0)$ is stable under the action of $B(r)\times B(n-r)$.

  For the second claim, we note that unipotent elements always act by representation matrices of determinant one.
  Hence it suffices to verify the formula for the determinant for $t \in H(r)$ and $t' \in H(n-r)$.
  For this, we work in the weight basis of $H_I(F_0, G_0)$ given by the elementary matrices $E_{b,a}$ that send $e(a) \mapsto \bar e(b)$, where $a\in[r]$ and $b\in[I(a)-a]$, and all other basis vectors to zero.
  Then:
  \begin{align*}
  &\quad \det (H_I(F_0, G_0)\ni\phi \mapsto t'\phi t^{-1} \in H_I(F_0, G_0)) \\
  &= \prod_{a=1}^r \prod_{b=1}^{I(a)-a} t'(b) t(a)^{-1}
  = \left( \prod_{a=1}^r t(a)^{a-I(a)} \right) \left( \prod_{b=1}^{n-r} t'(b)^{r - \#\{ a : I(a)-a<b\}} \right) \\
  &= t^{\lambda_I} {t'}^{r\one_{n-r} + \lambda_{I^c}},
  \end{align*}
  where we have used~\eqref{eq:complement} in the last step.
\end{proof}

We now show that the $\delta_\CI$ can be interpreted as an invariant:

\begin{theorem}
\label{thm:invariant}
  Let $\CI\in\Subsets(r,n,s)$ such that $\edim\CI=0$, and let $\delta_\CI$ denote the corresponding determinant function (\cref{def:determinant}).
  Then $\delta_\CI$ belongs to $\bigotimes_{k=1}^s \bigl( L_{BW}(\lambda_{I_k}^*) \ot L_{BW}(\lambda_{I_k^c}^* - r\one_{n-r}) \bigr)$.
  Moreover, it transforms under the diagonal action of $\GL(r)\times\GL(n-r)$ by the character $\det_r^{(n-r)(s-1)} \ot \det_{n-r}^{r(1-s)}$.
\end{theorem}
\begin{proof}
  For the first claim, we note that if $\vec{g'}\in B(r)^s$, $\vec{h'}\in B(n-r)^s$ then we can write $\Delta_{\CI,\vec g\vec{g'},\vec h\vec{h'}}$ as a composition of $\Delta_{\CI,\vec g,\vec h}$ with the automorphisms on $H_{I_k}(F_0, G_0)$ that send $\phi_k \mapsto h'_k \phi_k (g'_k)^{-1}$.
  Using \cref{lem:base change H_I}, we obtain
  \[ \delta_\CI(\vec g\vec{g'}, \vec h\vec{h'}) = \delta_\CI(\vec g, \vec h) \prod_{k=1}^s \chi_{\lambda_I}(g'_k) \chi_{\lambda_{I_k^c} + r\one_{n-r}}(h'_k). \]
  In view of the discussion at the end of \cref{subsec:borel-weil} this establishes the first claim.

  For the second claim, let $g\in\GL(r)$ and $g'\in\GL(n-r)$.
  Thus $\Delta_{\CI,g^{-1} \vec g,{g'}^{-1} \vec h}$ maps $(\zeta, \phi_1, \dots, \phi_s)$ to
  \begin{align*}
  &\quad(\zeta + g'^{-1} h_1 \phi_1 g_1^{-1} g, \dots, \zeta + g'^{-1} h_s \phi_s g_s^{-1} g) \\
  &= g'^{-1} (g' \zeta g^{-1} + h_1 \phi_1 g_1^{-1}, \dots, g' \zeta g^{-1} + h_s \phi_s g_s^{-1}) g
  \end{align*}
  Thus we can write $\Delta_{\CI,g^{-1} \vec g,{g'}^{-1} \vec h}$ as a composition of three maps:
  The automorphism $\zeta \mapsto g' \zeta g^{-1}$ of $\Hom(V_0,Q_0)$, the map $\Delta_{\CI,\vec g,\vec h}$ and the automorphism $\vec\phi \mapsto g'^{-1}\vec\phi g$ on $\Hom(V_0,Q_0)^s$.
  Thus, using \cref{eq:base change H},
  \begin{align*}
    ((g,g') \cdot \delta_\CI)(\vec g, \vec h)
  = \delta_\CI(g^{-1} \vec g, {g'}^{-1} \vec h)
  = \det(g)^{-(n-r)(1-s)} \det(g')^{r(1-s)} \delta_\CI(\vec g, \vec h),
  \end{align*}
  which establishes the second claim.
\end{proof}


If $\CI$ is intersecting then we had argued before that $\delta_\CI$ is nonzero.
By dualizing and simplifying, we obtain the following corollary of \cref{thm:invariant}:

\begin{corollary}
\label{cor:invariant}
  Let $\CI\in\Intersecting(r,n,s)$ and $\edim\CI=0$. Then,
  \[ \bigl( {\det}_r^{(s-1)(n-r)} \ot \bigotimes_{k=1}^s L(\lambda_{I_k}) \bigr)^{\GL(r)} \neq 0 \quad\text{and}\quad \bigl( {\det}_{n-r}^r \ot \bigotimes_{k=1}^s L(\lambda_{I^c_k}) \bigr)^{\GL(n-r)} \neq 0. \]
\end{corollary}

Let us correspondingly define
\begin{align*}
  c(\CI) := \dim \bigl( {\det}_r^{(s-1)(n-r)} \ot \bigotimes_{k=1}^s L(\lambda_{I_k}) \bigr)^{\GL(r)}.
\end{align*}
Then \cref{cor:invariant} states that, if $\CI$ is intersecting and $\edim\CI=0$ then $c(\CI)>0$.
This relationship between generic intersections of Schubert cells and tensor product multiplicities can be made quantitative.
While we do not use this in the following \cref{subsec:kirwan} to describe the Kirwan cone and prove the saturation property for tensor product multiplicities, we will give a brief sketch later on in \cref{subsec:invariants and intersections} and use it to establish the Fulton conjecture.

\subsection{Kirwan cone and saturation}
\label{subsec:kirwan}

We now show that the existence of nonzero invariants is characterized by the Horn inequalities.
For this, recall that we defined $c(\vec\lambda)$ as the dimension of the space of $\GL(r)$-invariants in the tensor product $\bigotimes_{k=1}^s L(\lambda_k)$.
Thus, if we define $\lambda_k = \lambda_{I_k} + (n-r)\one_r$ for $k\in[s-1]$ and $\lambda_s = \lambda_{I_s}$, then \cref{cor:invariant} shows that
\begin{equation}
\label{eq:tensor product invariant}
  c(\vec\lambda) = c(\CI) > 0
\end{equation}
whenever $\CI$ is intersecting and $\edim\CI=0$.
Here, we have somewhat arbitrarily selected the first $s-1$ highest weights $\lambda_1,\dots,\lambda_{s-1}$ to have nonnegative entries no larger than $n-r$, while $\lambda_s$ has nonpositive entries no smaller than $r-n$.
Conversely, any $s$-tuple of highest weights $\vec{\lambda}$ with these properties can be obtained in this way from some $\CI\in\Subsets(r,n,s)$ (recall discussion below \cref{def:lambda_I}).

\begin{proposition}
\label{prp:horn implies invariant}
  Let $\vec\lambda\in\Lambda_+(r)^s$ be an $s$-tuple such that $\sum_{k=1}^s \lvert\lambda_k\rvert = 0$, and for any $0<d<r$ and any $s$-tuple $\CJ\in\Horn(d,r,s)$ with $\edim\CJ=0$ we have that $\sum_{k=1}^s (T_{J_k},\lambda_k) \leq 0$.
  Then $c(\vec\lambda)>0$.
\end{proposition}
\begin{proof}
  By adding/removing suitable multiples of $\one_r$, the highest weight of the determinant representation, we may assume that $\lambda_1(r),\dots,\lambda_{s-1}(r)\geq0$ and $\lambda_s(1)\leq0$.
  Let $n:=r+q$, where $q:=\max\{\lambda_1(1),\dots,\lambda_{s-1}(1),-\lambda_s(r)\}$.
  Then $\vec\lambda$ is associated to an $s$-tuple $\CI\in\Subsets(r,n,s)$ as in \cref{lem:I to lambdaprime}.

  We now show that $\edim\CI=0$ and that $\CI$ is intersecting.
  The former follows from the first statement in \cref{lem:I to lambdaprime}, which gives that $\edim\CI=-\sum_{k=1}^s \lvert\lambda_k\rvert=0$.
  To see that $\CI$ is intersecting, we may use \cref{thm:belkale} and show instead that $\CI$ satisfies the Horn inequalities $\edim\CI\CJ\geq0$ for any $\CJ\in\Horn(d,r,s)$ with $\edim\CJ=0$ and $0<d<r$.
  But the second statement in \cref{lem:I to lambdaprime} implies that these are equivalent to the linear inequalities $\sum_{k=1}^s (T_{J_k},\lambda_k)\leq0$, which hold by assumption.
  Thus $\CI$ is indeed intersecting and satisfies $\edim\CI=0$.
  Now~\eqref{eq:tensor product invariant} shows that $c(\vec\lambda)=c(\CI)>0$.
\end{proof}

At last we can prove the saturation property and characterize of the Kirwan cone in terms of Horn inequalities.

\printhorncorollary{corollary}{\label{cor:horn and saturation}}{, restated}
\begin{proof}
The two statements are closely interlinked. For clarity, we give separate proofs that do not refer to each other.

(a) Any $\vec\xi\in\Kirwan(r,s)$ satisfies the Horn inequalities (\cref{cor:klyachko kirwan}).
We now observe that $\Kirwan(r,s)$ is a closed subset of $C_+(r)^s$ which, moreover, is invariant under rescaling by nonnegative real numbers.
Thus it suffices to prove the converse only for $\vec\lambda\in\Lambda_+(r)^s$.
For this, we use that if $\vec\lambda$ satisfies the Horn inequalities then $c(\vec\lambda)>0$ by \cref{prp:horn implies invariant}, hence $\vec\lambda\in\Kirwan(r,s)$ by \cref{prp:kempf-ness}.

(b) Let $\vec\lambda\in\Lambda_+(r)^s$.
If $c(\vec\lambda)>0$ then $\vec\lambda\in\Kirwan(r,s)$ by \cref{prp:kempf-ness}.
Conversely, if $\vec\lambda\in\Kirwan(r,s)$ then it satisfies the Horn inequalities by \cref{cor:klyachko kirwan}, hence $c(\vec\lambda)>0$ by \cref{prp:horn implies invariant}.
\end{proof}

\begin{remark}\label{rem:redundant}
  As follows from the discussion below \cref{prp:belkale strong weak}, the Kirwan cone is in fact already defined by those $\CJ$ such that $\Omega_\CJ(\CG)$ is a point for all $\CG\in\Good(r,s)$.
  Ressayre has shown that the corresponding inequalities are irredundant and can be computed by an inductive algorithm~\cite{ressayre2011cohomology}.
  Demanding that $\Omega_\CJ(\CG)$ is a point for all good $\CG$ is a more stringent requirement than $\edim\CJ=0$, and indeed the set of inequalities $\edim\CI\CJ\geq0$ for $\CJ\in\Horn(d,r,s)$ with $\edim\CJ=0$ is in general still redundant.
  However, from a practical point of view we prefer the latter criterion since it is much easier to check numerically.
\end{remark}

\subsection{Invariants and intersection theory}\label{subsec:invariants and intersections}
We now explain how the relationship between generic intersections of Schubert cells and tensor product multiplicities can be made more quantitative.
Specifically, we shall relate the dimension $c(\CI)$ of the space of $\GL(r)$-invariants to the number of points in a generic intersection $\Omega_\CI(\CE)$, as in the following definition:

\begin{definition}
  Let $\CI\in\Subsets(r,n,s)$ such that $\edim\CI=0$.
  We define the corresponding \emph{intersection number} as
  \[ \cint(\CI) := \#\Omega^0_\CI(\CE) = \#\Omega_\CI(\CE), \]
  where $\CE$ is an arbitrary $s$-tuple of flags in $\Good(n,s)$.
  By \cref{lem:good set}, the right-hand side is finite and independent of the choice of $\CE$ in $\Good(n,s)$.
  Moreover, $\cint(\CI)>0$ if and only if $\CI$ is intersecting.
\end{definition}

In \cref{subsec:invariants from intersecting tuples} we showed that if $\CI$ is intersecting then $c(\CI)>0$.
Indeed, in this case the determinant function $\delta_\CI$ on $\GL(r)^s\times\GL(n-r)^s$ is nonzero, so that for some suitable $\vec h\in\GL(n-r)^s$ the function
\begin{equation}\label{eq:invariant in one argument}
  \delta_{\CI,\vec h} \colon \GL(r)^s \to \C, \quad \delta_{\CI,\vec h}(\vec g) := \delta_\CI(\vec g, \vec h)
\end{equation}
is a nonzero vector in $\bigotimes_{k=1}^s L_{BW}(\lambda^*_{I_k})$ that transforms as the character $\det_r^{(n-r)(s-1)}$ with respect to the diagonal action of $\GL(r)$.

In the following we show that, as we vary $\vec h$, the functions $\delta_{\CI,\vec h}$ span a vector space of dimension at least $\cint(\CI)$, which will imply that $c(\CI)\geq\cint(\CI)$.
More precisely, we shall construct elements $(\vec g_\alpha, \vec h_\alpha)\in\GL(r)^s\times\GL(n-r)^s$ for $\alpha\in[\cint(\CI)]$ such that $\delta_\CI(\vec g_\alpha, \vec h_\alpha)\neq0$ while $\delta_\CI(\vec g_\alpha, \vec h_\beta)=0$ if $\alpha\neq\beta$.
The construction, due to Belkale~\cite{belkale2004invariant}, depends on a choice of good flags $\CE$ and goes as follows.

Let $\CE$ be an $s$-tuple of good flags and consider the intersection
\[ \Omega^0_\CI(\CE) = \{ V_1, \dots, V_{\cint(\CI)} \}. \]
Let $\gamma_\alpha\in\GL(n)$ such that $V_\alpha = \gamma_\alpha \cdot V_0$ for each $\alpha\in[\cint(\CI)]$, and consider the $s$-tuple of flags $\CE_\alpha=(E_{\alpha,1},\dots,E_{\alpha,s})$ defined by $E_{\alpha,k} = \gamma_\alpha^{-1} \cdot E_k$.
Then $\bar\omega^0_\CI([\gamma_\alpha, \CE_\alpha]) = \CE$.
According to \cref{lem:good set}, $\CE$ is a regular value of $\bar\omega^0_\CI$, since $\CI$ is intersecting.
Since $\edim\CI=0$, this implies that the differential of $\bar\omega^0_\CI$ is bijective at $[\gamma_\alpha, \CE_\alpha]$, and, by equivariance, so is its differential at $[1, \CE_\alpha]$.
By \cref{rem:delta^0_I factorized differential}, its determinant is precisely $\delta_\CI(\vec g_\alpha, \vec h_\alpha)$, where $\vec g_\alpha=(g_{\alpha,1},\dots,g_{\alpha,s})\in\GL(r)$ and $\vec h_\alpha=(h_{\alpha,1},\dots,h_{\alpha,s})\in\GL(n-r)$ are such that $g_{\alpha,k} \cdot F_0 = (E_{\alpha,k})^{V_0}$ and $h_{\alpha,k} \cdot G_0 = (E_{\alpha,k})_{Q_0}$ for all $\alpha$ and $k$.
In particular, $\delta_\CI(\vec g_\alpha, \vec h_\alpha)\neq0$.

Using $\edim\CI=0$, \cref{eq:rank nullity,eq:ker Delta} imply that
\begin{equation}\label{eq:delta_I vs H_I}
  \delta_\CI(\vec g,\vec h)\neq0 \;\Leftrightarrow\; \dim H_\CI(\vec g \cdot F_0, \vec h \cdot G_0)=0.
\end{equation}
Then we have the following lemma:

\begin{lemma}
\label{lem:criss cross}
  Let $\CI$ be intersecting, $\edim\CI=0$, and $\CE\in\Good(n,s)$.
  As above, choose $\gamma_\alpha$, $\vec g_\alpha$ and $\vec h_\alpha$ for $\alpha\in[\cint(\CI)]$.
  Define $\delta_{\CI,\alpha}(\vec g) := \det\Delta_{\CI,\vec g,\vec h_\alpha}$.
  Then $\delta_{\CI,\alpha}(\vec g_\alpha)\neq0$ for all $\alpha$, while $\delta_{\CI,\beta}(\vec g_\alpha)=0$ for all $\alpha\neq\beta$.
\end{lemma}
\begin{proof}
  We only need to consider the case that $\alpha\neq \beta$.
  In view of~\eqref{eq:delta_I vs H_I}, it suffices to show that $H_\CI((\CE_\alpha)^{V_0}, (\CE_\beta)_{Q_0})\neq\{0\}$.
  For this, we define the map
  \[ \phi_{\alpha,\beta}\colon V_0\to \C^n/V_0\cong Q_0, \; v \mapsto (\gamma_\alpha)^{-1} \gamma_\beta v + V_0, \]
  which is nonzero since $\gamma_\alpha V_0=V_\alpha\neq V_\beta=\gamma_\beta V_0$.
  Then $\phi_{\alpha,\beta}$ is a nonzero element in $H_\CI((\CE_\alpha)^{V_0}, (\CE_\beta)_{Q_0})$, since
  \[ \phi_{\alpha,\beta}((E_{\beta,k})^{V_0}(a))
  = \phi_{\alpha,\beta}(E_{\beta,k}(I_k(a)))
  = E_{\alpha,k}(I_k(a)) + V_0
  = (E_{\alpha,k})_{Q_0}(I_k(a)-a)
  \]
  for all $a\in[r]$ and $k\in[s]$, using that $\CI=\Pos(V_0,\CE_\alpha)$.
\end{proof}

\Cref{lem:criss cross} shows that the functions $\delta_{\CI,1},\dots,\delta_{\CI,\cint(\CI)}$ are linearly independent. 
If we identify them with $\GL(r)$-invariants as before, 
we obtain the following corollary:

\begin{corollary}
\label{cor:quantitative invariants}
  Let $\edim\CI=0$. 
  Then, $c(\CI) \geq \cint(\CI)$.
\end{corollary}

In fact, it is a classical result that
\begin{equation}
\label{eq:c=cint}
  c(\CI) = \cint(\CI)
\end{equation}
(see, e.g.,~\cite{fulton1997young}).
Thus \cref{cor:quantitative invariants} shows that we can produce a basis of the tensor product invariants from Belkale's determinants $\delta_{\CI,\vec h}(\vec g) = \det\Delta_{\CI,\vec g,\vec h}$.
These invariants can be identified with the construction of Howe, Tan and Willenbring~\cite{howe2005basis}, as described in~\cite{vergne2014inequalities}.

\counterwithin{equation}{subsection}
\section{Proof of Fulton's conjecture}\label{sec:fulton}
We now revisit the conjecture by Fulton which states that if $c(\vec\lambda)=1$ for an $s$-tuple of highest weights then $c(N\vec\lambda)=1$ for all $N\geq1$.
We note that its converse is also true and holds as a direct consequence of the saturation property and the bound $c(N\vec\lambda)\geq c(\vec\lambda)$, which follows from the semigroup property of the Littlewood-Richardson coefficients.
Fulton's conjecture was first proved by Knutson, Tao and Woodward~\cite{knutson2004honeycomb}.
We closely follow Belkale's geometric proof~\cite{belkale2004invariant,belkale2006geometric,belkale2007geometric}, in its simplified form due to Sherman~\cite{sherman2015geometric}, which in turn was in part inspired by the technique of Schofield~\cite{schofield1992general}.


\subsection{Nonzero invariants and intersections}\label{subsec:nonzero}
Let $c(\vec\lambda)=1$.
Equivalently, $c(\vec\lambda^*)=1$ and so there exists a nonzero $\GL(r)$-invariant holomorphic function $f$ in $L_{BW}(\lambda^*_1)\ot\dots\ot L_{BW}(\lambda^*_s)$, which is unique up to rescaling.

Suppose for a moment that there exists a nowhere vanishing function $g$ in $L_{BW}(\lambda^*_1)\ot\dots\ot L_{BW}(\lambda^*_s)$ (not necessarily $\GL(r)$-invariant).
In this case, if $g'$ is any other holomorphic function in $L_{BW}(\lambda^*_1)\ot\dots\ot L_{BW}(\lambda^*_s)$, then $g'/g$ is right $B(r)^s$-invariant and therefore descends to a holomorphic function on $\flag(r)^s$.
But this is a compact space, hence any such function is constant.
It then follows that each $L_{BW}(\lambda^*_k)$ is one-dimensional and hence that the $\lambda_k$ are just characters, i.e., $\lambda_k = m_k \one_r$ and $L(\lambda_k)=\det_r^{m_k}$ for some $m_k\in\Z$.
In this case, Fulton's conjecture is certainly true.

We now consider the nontrivial case when $f$ has zeros.
For any function in $L_{BW}(\lambda^*_1)\ot\dots\ot L_{BW}(\lambda^*_s)$, the zero set is right $B(r)^s$-stable.
Accordingly, we shall write $f(\CF)=0$ for the condition that $f(\vec g)=0$, where $\CF=\vec g \cdot F_0$, and consider
\[ Z_f := \{ \CF \in \flag(r)^s : f(\CF) = 0 \}. \]
Without loss of generality, we may assume that there exists an $s$-tuple $\CI$ with $\edim\CI=0$ that is related to $\vec\lambda$ as in \cref{lem:I to lambdaprime}, i.e.,
\begin{equation}\label{eq:lambda vs CI}
  \lambda_k = \lambda_{I_k} + (n-r) \one_r \text{ for $k\in[s-1]$}, \quad \lambda_s = \lambda_{I_s}
\end{equation}
(otherwise we may add/remove suitable multiples of $\one_r$, as in the proof of \cref{prp:horn implies invariant}).
Now recall from~\eqref{eq:invariant in one argument} that the functions $\delta_{\CI,\vec h} = \delta_\CI(-, \vec h)$ are in $\bigotimes_{k=1}^s L_{BW}(\lambda^*_{I_k})$ and transform as the character $\det_r^{(n-r)(s-1)}$ with respect to the diagonal action of $\GL(r)$.
It follows that each $\tilde\delta_{\CI,\vec h}(\vec g) := \det_r^{-(n-r)}(g_1) \cdots \det_r^{-(n-r)}(g_{s-1}) \delta_{\CI,\vec h}(\vec g)$  must be proportional to $f$.
Hence,
\begin{equation}\label{eq:delta decomposed}
  \delta_\CI(\vec g,\vec h) = {\det}_r^{(n-r)}(g_1) \cdots {\det}_r^{(n-r)}(g_{s-1}) f(\vec g) \hat f(\vec h),
\end{equation}
for some function $\hat f\colon\GL(n-r)^s\to\C$, which is nonzero due to~\eqref{eq:c=cint}.
In view of~\eqref{eq:delta_I vs H_I}, we obtain the following lemma:

\begin{lemma}\label{lem:Z vs H_I}
  Let $f$, $\CI$ as above.
  If $\CF\in Z_f$ then $H_\CI(\CF,\CG)\neq\{0\}$ for all $\CG\in\flag(Q_0)^s$.
  Conversely, if $\hat f(\CG)\neq0$ then $H_\CI(\CF,\CG)\neq\{0\}$ implies that $\CF\in Z_f$.
\end{lemma}

For sake of finding a contradiction, let us assume that $c(N\vec\lambda)>1$ for some $N$.
Then there exists an invariant $f' \in L_{BW}(N\lambda_1^*)\ot\dots\ot L_{BW}(N\lambda_s^*)$ that is linearly independent from $f^N$.

\begin{lemma}\label{lem:algebraic independence}
  Let $\CL$ be a holomorphic line bundle over a smooth irreducible variety.
  Then two linearly independent holomorphic sections $f_1$, $f_2$ are automatically algebraically independent.
\end{lemma}
\begin{proof}
  Let us suppose that $f_1$ and $f_2$ satisfy a nontrivial relation $\sum_{i,j} c_{i,j} f_1^i f_2^j=0$.
  Each $f_1^i f_2^j$ is a section of the line bundle $\CL^{\ot (i+j)}$.
  The relation holds degree by degree, and so we may assume that $i+j$ is the same for each nonzero $c_{i,j}$.
  But any homogeneous polynomial in two variables is a product of linear factors.
  Thus we have $\prod_i (a_i f_1 + b_i f_2) = 0$ for some $a_i, b_i \in \C$, and one of the factors has to vanish identically.
  This shows that $f_1$ and $f_2$ are linearly dependent, in contradiction to our assumption.
\end{proof}

\Cref{lem:algebraic independence} implies that $f^N$ and $f'$, and therefore $f$ and $f'$ are algebraically independent.
As a consequence, there exists a nonempty Zariski-open subset of $\CF \in Z_f$ such that $f'(\CF) \neq 0$.

Our strategy in the below will be as follows.
As before, we consider the kernel position $\CJ$ of a generic map $0\neq\phi\in H_\CI(\CF,\CG)$, with now $\CF$ varying in $Z_f$.
Although $\CJ$ is not necessarily intersecting, the condition $f'(\CF)\neq0$ will be sufficient to show that the tuple $\CI^\CJ$ is intersecting.
In \cref{subsec:sherman fulton} we will then prove Sherman's refined version of his recurrence relation~\eqref{eq:sherman recurrence}, which will allow us to show that $H_\CI(\CF,\CG)=\{0\}$ for generic $\CF\in Z_f$.
In view of \cref{lem:Z vs H_I}, this will give a contradiction.


We first prove a general lemma relating semistable vectors and moment maps.
Let $M$ be a complex vector space equipped with a $\GL(r)$-representation and $U(r)$-invariant Hermitian inner product $\braket{\cdot,\cdot}$, complex linear in the second argument, and denote by $\rho_M\colon\gl(r)\to\gl(M)$ the Lie algebra representation.
We define the corresponding \emph{moment map} $\Phi_M\colon\P(M)\to i\mathfrak u(r)$ by
\[ \tr\bigl(\Phi_M([m]) A\bigr) = \frac {\braket{m, \rho_M(A) m}} {\lVert m \rVert^2} \]
for all $A\in\gl(r)$; cf.\ \cref{eq:moment map equivariance}.

\begin{lemma}
\label{lem:moma nonpositive}
  Let $A\in i\mathfrak u(r)$ and $0\neq m\in M$. If $\exp(At) \cdot m\not\to0$ as $t\to-\infty$ then
  \[ \lim_{t\to-\infty} \tr\bigl(\Phi_M([\exp(At) \cdot m]) A\bigr) \leq0. \]
\end{lemma}
\begin{proof}
  Write $m=\sum_{i=1}^k m_i$ where the $m_i$ are nonzero eigenvectors of $\rho_M(A)$, with eigenvalues $\theta_1<\dots<\theta_k$.
  Then,
  \[ \exp(At) \cdot m = \sum_{i=1}^k e^{\theta_it} m_i \not\to 0 \]
  as $t\to-\infty$ if and only if $\theta_1\leq0$.
  In this case,
  \[ \lim_{t\to-\infty} \tr\bigl(\Phi_M([\exp(At) \cdot m]) A\bigr)
  = \lim_{t\to-\infty} \frac {\sum_i \theta_i e^{2\theta_it} \lVert m_i\rVert^2} {\sum_i e^{2\theta_it} \lVert m_i\rVert^2}
  = \theta_1\leq0.
  \qedhere \]
\end{proof}

We now relate the position of subspaces to components of the moment map:

\begin{lemma}\label{lem:single slope}
  Let $\lambda\in\Lambda_+(r)$, $F = g \cdot F_0$ a flag on $V_0$, $S$ a nonzero subspace of $\C^r$, and $P_S$ the orthogonal projector. Then,
  \[ \lim_{t\to-\infty} \tr\bigl(\Phi_{L(\lambda)}([\exp(P_St) g \cdot v_\lambda]) P_S\bigr) = \braket{T_J, \lambda}, \]
  where $J=\Pos(S,F)$.
\end{lemma}
\begin{proof}
  Let $d=\#J$.
  We may assume that $S=S_0$ is generated by the first $d$ vectors $e(1),\dots,e(d)$ of the standard basis of $V_0$, and also that $g=u$ is unitary.
  Thus $P_{S_0}$ is the diagonal matrix with $d$ ones and $r-d$ zeros, and we need to show that
  \[ \lim_{t\to-\infty} \tr\bigl(\Phi_{L(\lambda)}([\exp(P_{S_0}t) u \cdot v_\lambda]) P_{S_0}\bigr) = \braket{T_J, \lambda}. \]
  Let $R_0$ denote the orthogonal complement of $S_0$ in $V_0$.
  The action of $U(S_0) \times U(R_0)$ commutes with $P_{S_0}$ and hence we can assume that $F^{S_0}$ is the standard flag on $S_0$, while $F_{V_0/S_0}$ has the adapted basis $e(J^c(b))+S_0$ for $b\in[r-d]$.
  Thus we see that $\lim_{t\to-\infty} \exp(P_{S_0}t) F = w_J F_0$.
  It follows that $\lim_{t\to-\infty} [\exp(P_{S_0}t) u \cdot v_\lambda] = [w_J \cdot v_\lambda]$ and hence, using~\eqref{eq:moment map equivariance}, that
  \[ \lim_{t\to-\infty} \tr\bigl(\Phi_{L(\lambda)}([\exp(P_{S_0}t) u \cdot v_\lambda]) P_{S_0}\bigr)
   = \frac {\braket{v_\lambda, \rho_\lambda(w_J^{-1} P_{S_0} w_J) v_\lambda}} {\lVert v_\lambda\rVert^2}
   = \braket{T_J, \lambda}. \qedhere \]
\end{proof}

We now use the preceding lemma to obtain from any nonzero invariant an $s$-tuple of flags with nonnegative slope:

\begin{lemma}\label{lem:semistable slope}
  Let $p \in (L(N\lambda_1)\ot\cdots\ot L(N\lambda_s))^*$ a $\GL(r)$-invariant homogeneous polynomial such that $p(g_1\cdot v_{N\lambda_1}\ot\cdots\ot g_s\cdot v_{N\lambda_s})\neq0$, and define $\CF=(g_1 F_0,\dots,g_s F_0)$.
  Then $\sum_{k=1}^s (T_{J_k}, \lambda_k)\leq0$ for all $\CJ=\Pos(S,\CF)$, where $S$ is an arbitrary nonzero subspace of $\C^r$.
\end{lemma}
\begin{proof}
  Consider the representation $M=L(N\lambda_1)\ot\cdots\ot L(N\lambda_s)$ with its moment map $\Phi_M$, and $m := g_1\cdot v_{N\lambda_1}\ot\cdots\ot g_s\cdot v_{N\lambda_s}$.
  Let $P_S$ denote the orthogonal projector onto the subspace $S$.
  As $p$ is $\GL(r)$-invariant, $p(\exp(P_St) \cdot m)=p(m)\neq0$, which implies that $\exp(P_St) \cdot m\not\to0$ as $t\to-\infty$.
  Thus \cref{lem:moma nonpositive} implies that
  \begin{equation*}
    \lim_{t\to-\infty} \tr\bigl(\Phi_M([\exp(P_St) \cdot m]) P_S\bigr) \leq 0.
  \end{equation*}
  On the other hand, \cref{lem:single slope} shows that the left-hand side of this inequality is equal to
  \[
    \sum_{k=1}^s \lim_{t\to-\infty} \tr\bigl(\Phi_{L(N\lambda_k)}([\exp(P_St) g_k \cdot v_{N\lambda_k}]) P_S\bigr)
    = \sum_{k=1}^s \lim_{t\to-\infty} (T_{J_k}, \lambda_k). \qedhere
  \]
\end{proof}

\begin{corollary}\label{cor:refined intersecting}
  Let $\vec\lambda$ and $\CI$ as in~\eqref{eq:lambda vs CI}, $f' \in (L_{BW}(N\lambda_1^*) \ot \cdots \ot L_{BW}(N\lambda_s^*))^{\GL(r)}$.
  Let $\CF\in\flag(V_0)^s$, $\{0\}\neq S\subseteq\C^r$, and $\CJ=\Pos(S,\CF)$.
  If $f'(\CF)\neq0$ then $\CI^\CJ$ is intersecting.
\end{corollary}
\begin{proof}
  Write $\CF=(g_1 \cdot F_0, \dots, g_s \cdot F_0)$ for suitable $g_1,\dots,g_s \in\GL(r)$.
  Then, using~\eqref{eq:borel weil dual iso}, there exists a $\GL(r)$-invariant homogeneous polynomial $p\in(L(N\lambda_1)\ot\cdots\ot L(N\lambda_s))^*$ such that $p(g_1 \cdot v_{N\lambda_1} \ot\cdots\ot g_s \cdot v_{N\lambda_s})=f'(g_1,\dots,g_s)\neq0$.
  Thus the assumptions of \cref{lem:semistable slope} are satisfied.

  We now show that $\CI^\CJ$ is intersecting.
  For this, we use \cref{thm:belkale} and verify the Horn inequalities.
  Thus let $0<m\leq d=\dim S$ and $\CK\in\Horn(m,d,s)=\Intersecting(m,d,s)$:
  Since $\CK$ is intersecting, there exists some subspace $S'\in\Omega_\CK(\CF^S)$.
  Hence $S'\in\Omega_{\CJ\CK}(\CF)$ by the chain rule~\eqref{eq:chain rule pos}.
  According to \cref{lem:schubert variety characterization}, $\CJ'=\Pos(S',\CF)$ is such that $J'_k(a) \leq J_k K_k(a)$ for all $k\in[s]$ and $a\in[m]$.
  Thus we obtain the first inequality in
  \begin{align*}
    \edim\CI^\CJ\CK - \edim\CK
  = \edim\CI(\CJ\CK) - \edim\CJ\CK
  = -\sum_{k=1}^s (T_{J_k K_k}, \lambda_k)
  \geq -\sum_{k=1}^s (T_{J'_k}, \lambda_k)
  \geq 0;
  \end{align*}
  the first equality is~\eqref{eq:exp edim}, the second is \cref{lem:I to lambdaprime}, and the last inequality is \cref{lem:semistable slope}, applied to $S'$.
  This concludes the proof.
\end{proof}

\subsection{Sherman's refined lemma}
\label{subsec:sherman fulton}

We now study the behavior of $\dim H_\CI(\CF, \CG)$ in more detail.
We proceed as in \cref{sec:horn sufficient}, but for a \emph{fixed} $s$-tuple of flags $\CF\in\flag(V_0)^s$.
Specifically, we consider the following refinement of the true dimension~\eqref{eq:tdim} for fixed $\CF$:
\[ \tdim_\CF \CI := \min_{\CG} \dim H_\CI(\CF,\CG) \]
Thus we study the variety
\[ \PF(\CI) := \{ (\CG,\phi) \in \flag(Q_0)^s \times \Hom(V_0,Q_0) : \phi\in H_\CI(\CF,\CG) \}. \]
Restricting to those $\CG$ such that $\dim H_\CI(\CF,\CG)=\tdim_\CF \CI$, we obtain open sets $\BFt(\CI)\subseteq\flag(Q_0)^s$ and $\PFt(\CI)\subseteq\PF(\CI)$.
Let $\kdim_\CF(\CI)$ denote the minimal (and hence generic) dimension of $\ker\phi$ for $(\CG,\phi)\in\PFt(\CI)$.
The following lemma is proved just like \cref{cor:kdim zero intersecting}:  

\begin{lemma}\label{lem:kdim_F zero}
  If $\kdim_\CF\CI=0$ then $\tdim_\CF\CI=\edim \CI$.
\end{lemma}

Let us now assume that $\kdim_\CF\CI>0$. Let $\kPos_\CF(\CI)$ denote the kernel position, defined as in \cref{def:kerpos} but for fixed $\CF$.
We thus obtain an irreducible variety $\PFkpt(\CI)$ over a Zariski-open subset $\BFkpt(\CI)$ of $\BFt(\CI)$.
To compute its dimension, we again define $\PFkp(\CI)\subseteq\PF(\CI)$, where we fix the kernel dimension and position, but \emph{not} the dimension of $H_\CI(\CF,\CG)$.
In contrast to \cref{lem:dim P_kpt nonzero}, the variety $\PFkp(\CI)$ is in general neither smooth nor irreducible.
However, we can describe it similarly as before:
We first constrain $S=\ker\phi$ to be in $\Omega^0_\CJ(\CF)$ (which may not be irreducible), then $\phi$ is determined by $\bar\phi\in\Hom^\times(V_0/S,Q_0)$ and $\CG$ by $G_k\in\flag^0_{I_k/J_k}((F_k)_{V_0/S},\bar\phi)$.
Thus we obtain for each irreducible component $C \subseteq\Omega^0_\CJ(\CF)$ a corresponding irreducible component $\PFkpC(\CI)$.
In particular, there exists some component $C_\CF$ such that $\PFkpCF(\CI)$ is the closure of $\PFkpt(\CI)$ in $\PFkp(\CI)$, namely the irreducible component containing the elements $S=\ker\phi$ for $(\phi,\CG)$ varying in the irreducible variety $\PFkpt(\CI)$.
As a consequence, $\dim\PFkpt(\CI) = \dim\PFkpCF(\CI)$, and so we obtain, using completely analogous dimension computations, the following refinement of~\eqref{eq:tdim minus edim first}:
\begin{equation}\label{eq:tdim minus edim refined}
  \tdim_\CF \CI - \edim\CI = \dim C_\CF - \edim\CI^\CJ
\end{equation}
Indeed, when we apply~\eqref{eq:tdim minus edim refined} to generic $\CF\in\flag(V_0)^s$ then $\CJ$ is intersecting and $\dim C_\CF=\edim\CJ$, so we recover~\eqref{eq:tdim minus edim first}.
We now instead apply the above to generic $\CF$ in a component of the zero set $Z_f$ of the unique nonzero invariant~$f$.
Thus we obtain the following variant of the key recursion relation~\eqref{eq:sherman recurrence}:

\begin{lemma}[Sherman]\label{lem:sherman refined}
  Let $f,\CI$ as above in \cref{subsec:nonzero}, and $Z\subseteq Z_f$ an irreducible component such that $\kdim_\CF\CI\neq0$ for all $\CF\in Z$.
  Then there exists $\CJ$ and a nonempty Zariski-open subset of $\CF\in Z$ such that $\kPos_\CF\CI=\CJ$ and
  \[ \tdim_\CF \CI - \edim \CI \leq \tdim \CI^\CJ - \edim \CI^\CJ. \]
\end{lemma}
\begin{proof}
  We choose $d$ and $\CJ$ as the kernel dimension and position for generic $\CF\in Z$.
  We note that $d<r$, since $d=r$ would imply that $H_\CI(\CF,\CG)=\{0\}$, in contradiction to \cref{lem:Z vs H_I}.
  Let $U\subseteq Z$ denote the Zariski-open subset such that $\kPos_\CF\CI=\CJ$ for all $\CF\in U$.
  We proceed as in \cref{lem:sherman upper bound}.
  Let
  \begin{align*}
    \CX &:= \{ (\CF,\CG,\phi) : \CF \in U, (\CG,\phi) \in \PFkpCF(\CI) \}, \\
    \CY &:= \{ (S,\tilde\CF,\CG) : S \in \grass(d, V_0), \tilde\CF \in \flag(S)^s, \CG\in\flag(Q_0)^s \}.
  \end{align*}
  Both $\CX$ and $\CY$ are irreducible varieties and we have a morphism
  \[ \pi\colon\CX\to\CY, (\CF,\CG,\phi) \mapsto (\ker\phi,\CF^{\ker\phi},\CG). \]
  As before, we argue that $\pi$ is dominant.
  By construction, the image of $\CX$ by the map $(\CF,\CG,\phi)\mapsto\CG$ contains a Zariski-open subset $U'$ of $\flag(Q_0)^s$.
  We may also assume that $\hat f(\CG)\neq0$ for all $\CG\in U'$, where $\hat f$ is the map from~\eqref{eq:delta decomposed}.
  We now show that the image of $\pi$ contains all elements $(S,\tilde\CF,\CG)$ with $S\in\grass(d,V_0)$, $\tilde\CF\in\flag(S)^s$, and $\CG\in U'$.
  For this, let $(\CF_0,\CG,\phi_0)\in\CX$ be the preimage of some arbitrary $\CG\in U'$.
  Let $S_0 :=\ker\phi_0$ and choose some $g \in \GL(V_0)$ such that $g \cdot S_0 = S$.
  Using the corresponding diagonal action, define $\CF := g \cdot \CF_0$ and $\phi := g \cdot \phi_0$.
  Then $(\CF,\CG,\phi)\in\CX$, since $Z$ is stable under the diagonal action of $\GL(V_0)$, and $\ker\phi=S$.
  Now consider the group $G \subseteq \GL(V_0)^s$ consisting of all elements $\vec h\in\GL(V_0)^s$ such that $h_k S\subseteq S$ and $h_k$ acts trivially on $V_0/S$ for all $k\in[s]$.
  Note that $G$ is an irreducible algebraic group.
  By construction, $\phi\in H_\CI(\vec h\cdot\CF,\CG)$, while $d<r$ implies that $\phi\neq0$.
  This means that $H_\CI(\vec h\cdot\CF,\CG)\neq 0$, and so we obtain from \cref{lem:Z vs H_I} that $\vec h\cdot\CF\in Z_f$.
  It follows that, in fact, $\vec h\cdot\CF\in Z$, as it is obtained by the action of the irreducible algebraic group $G$ on $\CF\in Z$, and so stays in the same irreducible component.
  For given $\tilde\CF\in\flag(S)^s$, we now choose $\vec h\in G$ such that $h_k \cdot F^S_k = \tilde F_k$ for $k\in[s]$.
  Then $(\vec h\cdot\CF,\CG,\phi)\in\CX$ is a preimage of $(S,\tilde\CF,\CG)$, and we conclude that $\pi$ is dominant.

  As before, the dominance implies that we can find a nonempty Zariski-open set of $(\CF,\CG,\phi)\in\CX$ such that $\dim H_{\CI^\CJ}(\CF^{\ker\phi},\CG)=\tdim\CI^\CJ$.
  We may assume in addition that $\ker\phi \in C_\CF$ is a smooth point.
  For any $\CF$ in this set, $\bar\phi$ injects $H_\CJ(\CF^{\ker\phi}, \CF_{V_0/\ker\phi})$ into $H_{\CI^\CJ}(\CF^{\ker\phi},\CG)$ (\cref{lem:exp vs composition}).
  Thus,
  \[ \dim C_\CF = \dim T_{\ker\phi} C_\CF \leq H_\CJ(\CF^{\ker\phi},\CF_{V_0/\ker\phi})\leq\dim H_{\CI^\CJ}(\CF^{\ker\phi},\CG) = \tdim\CI^\CJ, \]
  where in the first inequality we have used that the intersection $\Omega^0_\CJ(\CF)$ is not necessarily transversal at $S$ and so the tangent space of the intersection is in general only a subspace of the intersection of the tangent spaces~\eqref{eq:tangent space schubert cell}.
  In view of~\eqref{eq:tdim minus edim refined}, we obtain the desired inequality.
\end{proof}

\begin{theorem}[Belkale]
  Let $c(\vec\lambda)=1$.
  Then $c(N\vec\lambda)=1$ for all $N>1$.
\end{theorem}
\begin{proof}
  Let $f$, $\CI$ as in \cref{subsec:nonzero} and recall that $\edim\CI=0$.
  Assume for sake of finding a contradiction that $c(N\vec\lambda)\neq1$ for some $N>0$.
  Then there exists another invariant $f'$, as in \cref{subsec:nonzero}, such that $f'(\CF)\neq0$ for a nonempty Zariski-open subset of some irreducible component $Z\subseteq Z_f$.
  If $\kdim_\CF\CI=0$ for some $\CF\in Z$ then $\tdim_\CF\CI=0$ by \cref{lem:kdim_F zero}.
  Otherwise, we may apply \cref{lem:sherman refined} to the component $Z$.
  We find that there exists some $\CJ$ and another Zariski-open subset of $\CF$ in $Z$ such that $\kPos_\CF\CI=\CJ$ and
  \begin{equation}\label{eq:sherman for belkale}
    \tdim_\CF \CI - \edim \CI \leq \tdim \CI^\CJ - \edim \CI^\CJ.
  \end{equation}
  As a consequence, there exists some $\CF\in Z_f$ for which all three of the properties $f'(\CF)\neq0$, $\kPos_\CF\CI=\CJ$ and~\eqref{eq:sherman for belkale} hold true.
  By \cref{cor:refined intersecting}, the first two properties imply that $\CI^\CJ=0$, and hence the right-hand side of~\eqref{eq:sherman for belkale} is equal to zero by \cref{cor:tdim edim}.
  This again implies that $\tdim_\CF\CI=0$.

  It follows that in either case there exist some $\CG$ such that $H_\CI(\CF,\CG)=\{0\}$.
  According to \cref{lem:Z vs H_I}, this can only be if $\CF\not\in Z_f$.
  But $\CF\in Z_f$. This is the desired contradiction.
\end{proof}

\begin{remark*}
  It likewise holds that $c(\vec\lambda)=2$ implies that $c(N\vec\lambda)=N+1$~\cite{ikenmeyer2012small,sherman2015geometric}.
  However, in general it is \emph{not} true that $c(\vec\lambda)=c$ implies $c(N\vec\lambda) = O(N^{c-1})$.
  Belkale has a found a counterexample for $c=6$.
\end{remark*}

\section*{Acknowledgements}
We would like to thank Velleda Baldoni and Vladimir Popov for interesting discussions.
We also thank Prakash Belkale and Cass Sherman for their helpful remarks.
We would also like to thank an anonymous referee for their suggestions.
MW acknowledges financial support by the Simons Foundation and AFOSR grant FA9550-16-1-0082.

\appendix
\counterwithin{equation}{section}
\section{Horn triples in low dimensions}\label{app:examples horn}

In this appendix, we list all Horn triples $\CJ\in\Horn(d,r,3)$ for $d<r\leq4$, as defined in \cref{def:horn}, as well as the expected dimensions $\edim\CJ$. The triples with $\edim\CJ=0$ are highlighted in bold.

\begin{example}[$d=1$]\label{ex:horn 1}
  As discussed in \cref{ex:base case}, only the dimension condition $\edim\CJ\geq0$ is necessary.
  The following are the triples in $\Horn(1,r,3)$ (up to permutations):
  \begin{center}
  \begin{longtable}{rcccr}
    \toprule\toprule
    r & $J_1$ & $J_2$ & $J_3$ & $\edim\CJ$ \\
    \midrule
    \endfirsthead
    \toprule
    r & $J_1$ & $J_2$ & $J_3$ & $\edim\CJ$ \\
    \midrule
    \endhead
2 & \bf \{1\} & \bf \{2\} & \bf \{2\} & \bf 0 \\
& \{2\} & \{2\} & \{2\} & 1 \\
\midrule
3 & \bf \{1\} & \bf \{3\} & \bf \{3\} & \bf 0 \\
& \bf \{2\} & \bf \{2\} & \bf \{3\} & \bf 0 \\
& \{2\} & \{3\} & \{3\} & 1 \\
& \{3\} & \{3\} & \{3\} & 2 \\
\midrule
4 & \bf \{1\} & \bf \{4\} & \bf \{4\} & \bf 0 \\
& \bf \{2\} & \bf \{3\} & \bf \{4\} & \bf 0 \\
& \{2\} & \{4\} & \{4\} & 1 \\
& \bf \{3\} & \bf \{3\} & \bf \{3\} & \bf 0 \\
& \{3\} & \{3\} & \{4\} & 1 \\
& \{3\} & \{4\} & \{4\} & 2 \\
& \{4\} & \{4\} & \{4\} & 3 \\
    \bottomrule
    \bottomrule
  \end{longtable}
  \end{center}
\end{example}

\begin{example}[$d=2$]\label{ex:horn 2}
  The dimension condition $\edim\CJ\geq0$ reads
  \[ \left( J_1(1) + J_1(2) \right) + \left( J_2(1) + J_2(2) \right) + \left( J_3(1) + J_3(2) \right) \geq 4r+1. \]
  In addition, we have to satisfy three Horn inequalities, corresponding to $\CK = (\{1\},\{2\},\{2\})$ and its permutations, which are the only elements in $\Horn(1,2,3)$ with $\dim\CK=0$. The resulting Horn inequalities, $\edim\CJ\CK\geq0$, are
  \begin{align*}
    J_1(1) + J_2(2) + J_3(2) &\geq 2r+1, \\
    J_1(2) + J_2(1) + J_3(2) &\geq 2r+1, \\
    J_1(2) + J_2(2) + J_3(1) &\geq 2r+1.
  \end{align*}
  Thus we obtain the following triples in $\Horn(2,r,3)$ (up to permutations):
  \begin{center}
  \begin{longtable}{rcccr}
    \toprule\toprule
    r & $J_1$ & $J_2$ & $J_3$ & $\edim\CJ$ \\
    \midrule
    \endfirsthead
    \toprule
    r & $J_1$ & $J_2$ & $J_3$ & $\edim\CJ$ \\
    \midrule
    \endhead
3 & \bf \{1, 2\} & \bf \{2, 3\} & \bf \{2, 3\} & \bf 0 \\
& \bf \{1, 3\} & \bf \{1, 3\} & \bf \{2, 3\} & \bf 0 \\
& \{1, 3\} & \{2, 3\} & \{2, 3\} & 1 \\
& \{2, 3\} & \{2, 3\} & \{2, 3\} & 2 \\
\midrule
4 & \bf \{1, 2\} & \bf \{3, 4\} & \bf \{3, 4\} & \bf 0 \\
& \bf \{1, 3\} & \bf \{2, 4\} & \bf \{3, 4\} & \bf 0 \\
& \{1, 3\} & \{3, 4\} & \{3, 4\} & 1 \\
& \bf \{1, 4\} & \bf \{1, 4\} & \bf \{3, 4\} & \bf 0 \\
& \bf \{1, 4\} & \bf \{2, 4\} & \bf \{2, 4\} & \bf 0 \\
& \{1, 4\} & \{2, 4\} & \{3, 4\} & 1 \\
& \{1, 4\} & \{3, 4\} & \{3, 4\} & 2 \\
& \bf \{2, 3\} & \bf \{2, 3\} & \bf \{3, 4\} & \bf 0 \\
& \bf \{2, 3\} & \bf \{2, 4\} & \bf \{2, 4\} & \bf 0 \\
& \{2, 3\} & \{2, 4\} & \{3, 4\} & 1 \\
& \{2, 3\} & \{3, 4\} & \{3, 4\} & 2 \\
& \{2, 4\} & \{2, 4\} & \{2, 4\} & 1 \\
& \{2, 4\} & \{2, 4\} & \{3, 4\} & 2 \\
& \{2, 4\} & \{3, 4\} & \{3, 4\} & 3 \\
& \{3, 4\} & \{3, 4\} & \{3, 4\} & 4 \\
    \bottomrule\bottomrule
  \end{longtable}
  \end{center}
\end{example}

\begin{example}\label{ex:horn 3}
  We find the following triples in $\Horn(3,4,3)$ (up to permutations):
  \begin{center}
  \begin{longtable}{rcccr}
    \toprule\toprule
    r & $J_1$ & $J_2$ & $J_3$ & $\edim\CJ$ \\
    \midrule
    \endfirsthead
    \toprule
    r & $J_1$ & $J_2$ & $J_3$ & $\edim\CJ$ \\
    \midrule
    \endhead
4 & \bf \{1, 2, 3\} & \bf \{2, 3, 4\} & \bf \{2, 3, 4\} & \bf 0 \\
& \bf \{1, 2, 4\} & \bf \{1, 3, 4\} & \bf \{2, 3, 4\} & \bf 0 \\
& \{1, 2, 4\} & \{2, 3, 4\} & \{2, 3, 4\} & 1 \\
& \bf \{1, 3, 4\} & \bf \{1, 3, 4\} & \bf \{1, 3, 4\} & \bf 0 \\
& \{1, 3, 4\} & \{1, 3, 4\} & \{2, 3, 4\} & 1 \\
& \{1, 3, 4\} & \{2, 3, 4\} & \{2, 3, 4\} & 2 \\
& \{2, 3, 4\} & \{2, 3, 4\} & \{2, 3, 4\} & 3 \\
    \bottomrule\bottomrule
  \end{longtable}
  \end{center}
\end{example}

\begin{remark*}
It is not an accident that both $\Horn(1,4,3)$ and $\Horn(3,4,3)$ have the same number of elements.
In fact, we can identify $\Intersecting(d,r,s)\cong\Intersecting(r-d,r,s)$ via $J_k \mapsto \{ r+1-a : a \in J_k^c \}$.
This can be seen by using the canonical isomorphism $\grass(d,\C^r)\cong\grass(r-d,(\C^r)^*)$.
However, the corresponding Horn inequalities are distinct (see below).
\end{remark*}

\section{Kirwan cones in low dimensions}\label{app:examples kirwan}

In this appendix, we list necessary and sufficient conditions on highest weights $\lambda,\mu,\nu\in\Lambda_+(r)$ such that $\left(L(\lambda)\ot L(\mu)\ot L(\nu) \right)^{U(r)}\neq\{0\}$, up to $r=4$.
That is, these conditions describe the Kirwan cones as in \cref{cor:horn and saturation}.
We use the abbreviation $\Horn_0(d,r,s)$ for the set of Horn triples in $\CJ\in\Horn(d,r,s)$ such that $\edim\CJ=0$ (highlighted bold in \cref{app:examples horn}).

\begin{example}[$r=1$]
  Clearly, the only condition is $\lambda(1)+\mu(1)+\nu(1)=0$.
\end{example}

\begin{example}[$r=2$]
  We always have the Weyl chamber inequalities $\lambda(1)\geq\lambda(2)$, $\mu(1)\geq\mu(2)$, and $\nu(1)\geq\nu(2)$, and the equation
  \[ \left( \lambda(1)+\lambda(2) \right) + \left( \mu(1)+\mu(2) \right) + \left( \nu(1)+\nu(2) \right) = 0. \]
  Using \cref{ex:horn 1}, we obtain three Horn inequalities, namely
  \[ \lambda(1) + \mu(2) + \nu(2) \leq 0, \]
  corresponding to the triple $(\{1\},\{2\},\{2\})\in\Horn_0(d,r,s)$, and its permutations.
  These are the well-known conditions for the existence of nonzero invariants in a triple tensor product of irreducible $U(2)$-representations.
  We remark that the Weyl chamber inequalities are redundant.
\end{example}

\begin{example}[$r=3$]
  In addition to the Weyl chamber inequalities and $\lvert\lambda\rvert+\lvert\mu\rvert+\lvert\nu\rvert=0$, we obtain the following two inequalities from $\Horn_0(1,3,3)$ and \cref{ex:horn 1},
  \begin{align*}
    \lambda(1) + \mu(3) + \nu(3) &\leq 0, \\
    \lambda(2) + \mu(2) + \nu(3) &\leq 0,
  \end{align*}
  and the following from $\Horn_0(2,3,3)$ and \cref{ex:horn 2},
  \begin{align*}
    \left( \lambda(1) + \lambda(2) \right) + \left( \mu(2) + \mu(3) \right) + \left( \nu(2) + \nu(3) \right) &\leq 0, \\
    \left( \lambda(1) + \lambda(3) \right) + \left( \mu(1) + \mu(3) \right) + \left( \nu(2) + \nu(3) \right) &\leq 0,
  \end{align*}
  as well as their permutations.
\end{example}

\begin{example}[$r=4$]
  Again we have the Weyl chamber inequalities and $\lvert\lambda\rvert+\lvert\mu\rvert+\lvert\nu\rvert=0$.
  We have the following two inequalities and their permutations from $\Horn_0(1,4,3)$ and \cref{ex:horn 1},
  \begin{align*}
    \lambda(1) + \mu(4) + \nu(4) &\leq 0, \\
    \lambda(2) + \mu(3) + \nu(4) &\leq 0, \\
    \lambda(3) + \mu(3) + \nu(3) &\leq 0,
  \end{align*}
  the following six and their permutations from $\Horn_0(2,4,3)$ and \cref{ex:horn 2},
  \begin{align*}
    \left( \lambda(1) + \lambda(2) \right) + \left( \mu(3) + \mu(4) \right) + \left( \nu(3) + \nu(4) \right) &\leq 0, \\
    \left( \lambda(1) + \lambda(3) \right) + \left( \mu(2) + \mu(4) \right) + \left( \nu(3) + \nu(4) \right) &\leq 0, \\
    \left( \lambda(1) + \lambda(4) \right) + \left( \mu(1) + \mu(4) \right) + \left( \nu(3) + \nu(4) \right) &\leq 0, \\
    \left( \lambda(1) + \lambda(4) \right) + \left( \mu(2) + \mu(4) \right) + \left( \nu(2) + \nu(4) \right) &\leq 0, \\
    \left( \lambda(2) + \lambda(3) \right) + \left( \mu(2) + \mu(3) \right) + \left( \nu(3) + \nu(4) \right) &\leq 0, \\
    \left( \lambda(2) + \lambda(3) \right) + \left( \mu(2) + \mu(4) \right) + \left( \nu(2) + \nu(4) \right) &\leq 0,
  \end{align*}
  and the following three and their permutations from $\Horn_0(3,4,3)$ and \cref{ex:horn 3},
  \begin{align*}
    \left( \lambda(1) + \lambda(2) + \lambda(3) \right) + \left( \mu(2) + \mu(3) + \mu(4) \right) + \left( \nu(2) + \nu(3) + \nu(4) \right) &\leq 0, \\
    \left( \lambda(1) + \lambda(2) + \lambda(4) \right) + \left( \mu(1) + \mu(3) + \mu(4) \right) + \left( \nu(2) + \nu(3) + \nu(4) \right) &\leq 0, \\
    \left( \lambda(1) + \lambda(3) + \lambda(4) \right) + \left( \mu(1) + \mu(3) + \mu(4) \right) + \left( \nu(1) + \nu(3) + \nu(4) \right) &\leq 0.
  \end{align*}
\end{example}

\begin{remark*}
In low dimensions, all Horn triples with $\edim\CJ=0$ are such that the intersection is one point, i.e., $c(\CI)=\cint(\CI)=1$.
This implies that the equations are irredundant~\cite{belkale2006geometric,ressayre2010geometric} (cf.\ \cref{rem:redundant}), and it can also be explicitly checked in the examples above.
In general, however, this is not the case, and so the Horn inequalities are still redundant.
An example of such a Horn triple is the one given in \cref{ex:interesting}.
\end{remark*}


\bibliographystyle{amsplain}
\bibliography{belkale}

\end{document}